\documentclass[11pt]{amsart}
\usepackage{amsmath,amssymb,latexsym,dsfont}
\usepackage[left=2.6cm,right=2.6cm,top=2.7cm,bottom=2.7cm]{geometry}

\usepackage{amsmath,amsfonts,amssymb,epsfig, textcomp}
\usepackage{graphicx,tikz}
\usepackage{hyperref, cleveref} 
\usepackage{cases,listings}
\usepackage{color}

\newtheorem{theorem}{Theorem}[section]
\newtheorem{lemma}{Lemma}[section]

\newtheorem{proposition}{Proposition}[section]

\newtheorem{remark}{Remark}[section]

\newtheorem{definition}{Definition}[section]
\numberwithin{equation}{section}

      \newcommand{\hy}{\hat y}
      \newcommand{\hu}{\hat u}
      \newcommand{\hv}{\hat v}

      \newcommand{\hg}{\hat g}

      \newcommand{\hh}{\hat h}

      \newcommand{\tlambda}{\bar \lambda}

   \newcommand{\ty}{\widetilde y}

      \newcommand{\cN}{{\mathcal N}}

      \newcommand{\R}{{\mathbb{R}}}

      \newcommand{\N}{\mathbb{N}}

      \newcommand{\loc}{\operatorname{loc}}
      
      \newcommand{\eps}{\varepsilon}
      \newcommand{\mR}{\mathbb{R}}
      \newcommand{\mZ}{\mathbb{Z}}

      \newcommand{\mC}{\mathbb{C}}

      \newcommand{\supp}{\operatorname{supp}}

      \newcommand{\M}{{\mathcal M}}

      \makeatletter
      \def\@setcopyright{}
      \def\serieslogo@{}
      \makeatother
   \newcommand{\tr}{^\mathsf{T}}

\newcommand{\diff}[1][-3]{\mathop{}\mkern#1mu{d}}

\newcommand{\cM}{\mathcal M}
\newcommand{\cL}{\mathcal L}
\newcommand{\proj}{\mbox{proj}}

\newcommand{\by}{{\bf y}}

\newcommand{\tvarphi}{\widetilde \varphi}

\newcommand{\tu}{\widetilde u}

\newcommand{\cA}{{\mathcal A}}

\newcommand{\cV}{{\mathcal V}}

\newcommand{\dy}{\delta y}

\newcommand{\be}{\begin{equation}}
\newcommand{\ee}{\end{equation}}
   
\title[KDV]{Local controllability of the Korteweg-de Vries equation with the right Dirichlet control}

\author[H.-M. Nguyen]{Hoai-Minh Nguyen}
\address[H.-M. Nguyen]{Laboratoire Jacques Louis Lions, \newline\indent
Sorbonne Universit\'e\newline\indent
Paris, France}
\email{hoai-minh.nguyen@sorbonne-universite.fr}

\begin{document}

\maketitle

\begin{abstract} 
The Korteweg-de Vries (KdV)  equation with the right Dirichlet control was initially investigated more than twenty years ago. It was shown that this system is small time, locally, exactly controllable for all non-critical lengths and its linearized system is not controllable for {\it all} critical lengths.  Even though the controllability of the KdV system has been studied extensively in the last two decades, the local controllability of this system for critical lengths remains an open question.  In this paper, we give a definitive answer to this question. First,  we characterize all critical lengths and the corresponding unreachable space for the linearized system. In particular, we show that the unreachable space is always of dimension 1.  Second, we prove that the KdV system with the right Dirichlet control is not locally null controllable in small time. Third, we give a criterion to determine whether the system is locally exactly controllable in finite time or {\it not} locally null controllable in any positive time for {\it all} critical lengths. Consequently, we show that there exist critical lengths such that the system is not locally null controllable in small time but is locally exactly controllable in finite time. These facts are surprising and distinct in comparison with related known results. First,  it is known that the corresponding KdV system with the right zero Dirichlet is locally exactly controllable in small time using internal controls. Second, the unreachable space of the linearized system of
the corresponding KdV system with the right Neumann control might be of arbitrary dimension. Third, the KdV system with the right Neumann control is locally exactly controllable in small time if the corresponding unreachable space is of dimension 1. 
\end{abstract}

\medskip
\noindent{\bf Key words:} Controllability, KdV equations, critical lengths, unreachable space,  power series expansion, Hilbert uniqueness method

\medskip
\noindent{\bf AMS subject classification:} 35C20, 35Q53, 93B05, 93B07, 93C20.

\tableofcontents

\section{Introduction and the statement of the main results}

This paper is devoted to the local controllability of the Korteweg-de Vries (KdV) equation using the right Dirichlet control. More precisely, we consider the following control problem, for $T>0$,  
\begin{equation}\label{sys-KdV}\left\{
\begin{array}{cl}
y_t  + y_x  + y_{xxx}  + y  y_x  = 0 &  \mbox{ in } (0, T) \times (0, L), \\[6pt]
y(\cdot, 0) = y_x(\cdot, L) = 0 & \mbox{ in }   (0, T), \\[6pt]
y(\cdot, L) = u & \mbox{ in }  (0, T), \\[6pt]
y(0, \cdot)  = y_0  &  \mbox{ in }  (0, L). 
\end{array}\right.
\end{equation}
Here  $y$ is the state, $y_0 \in L^2(0, L)$ is an initial datum,  and $u$ is a control, belonging to an appropriate functional space.  The KdV equation has been introduced by Boussinesq \cite{1877-Boussinesq} and Korteweg
and de Vries \cite{KdV} as a model for the propagation of surface water waves along a
channel. This equation also furnishes a very useful nonlinear approximation model
including a balance between weak nonlinearity and weak dispersive effects, see e.g.~\cite{Whitham74, Miura76, Kato83}.  The KdV
equation has been investigated from various aspects of mathematics, including the well-posedness, the existence, and stability of solitary waves, the integrability, the
long-time behavior, etc., see e.g.~\cite{Whitham74, Miura76, Kato83, Tao06, LP15}.

\subsection{State of the art} The local controllability for the KdV equation has been studied extensively in the literature, see, e.g., the surveys \cite{RZ09, Cerpa14} and the references therein. We briefly review here some results concerning boundary controls.  When the controls are $y(\cdot, 0)$, $y(\cdot, L)$,   $y_x(\cdot, L)$, Russell and Zhang \cite{RZ96} proved that the KdV equation is small time,  locally exactly controllable. The case of left boundary control ($y(\cdot, L) = y_x(\cdot, L) = 0$) was investigated by Rosier \cite{Rosier04} (see also \cite{GG08}). The small-time local, null controllability holds in this case. The exact controllability does not hold for initial and final data in the $L^2(0, L)$ due to the regularisation effect  but 
holds for a subclass of infinitely smooth initial and final data \cite{MRRR19}.    

A very close setting to the one considered here is the setting in which one controls the right Neumann boundary, i.e.,  $y(\cdot, 0) = y(\cdot, L) = 0$ and $y_x(\cdot, L)$ is a control. For initial and final data in $L^2(0, L)$, and controls in $L^2(0, T)$, Rosier~\cite{Rosier97} proved  that the KdV system with the right Neumann control
is small time, locally, exactly controllable provided that the length $L$ is not critical,
i.e., $L \notin \cN_N$ \footnote{The letter $N$ stands for the Neumann boundary control.}, where \footnote{$\N_* = \N \setminus \{0 \}.$}
\begin{equation}\label{def-cN}
\cN_N : = \left\{ 2 \pi \sqrt{\frac{k^2 + kl + l^2}{3}}; \, k, l \in \N_*\right\}.
\end{equation}
To this end, Rosier studied the controllability of the corresponding linearized system and showed that the linearized system is exactly controllable if $L \not \in \cN_N$. He as well established that when $L \in \cN_N$,  the linearized system is not controllable. More precisely, Rosier showed
that there exists a non trivial, finite dimensional subspace $\M_{N}$  of $L^2(0, L)$  such that its orthogonal space is reachable from $0$ for small time whereas  $\M_N$ is not for any time.  
To tackle the control problem for a critical length $L \in \cN_N$ with initial and final data in $L^2(0, L)$ and controls in $L^2(0, T)$, Coron and Cr\'epeau
introduced the power series expansion method \cite{CC04}. The idea is to take into account the effect of the nonlinear term $y y_x$  absent in the corresponding linearized system. Using this 
method, Coron and Cr\'epeau showed \cite{CC04} (see also \cite[section 8.2]{Coron07}) that the KdV system is small time,  locally,  exactly controllable when  $\dim \M_N = 1$. Cerpa \cite{Cerpa07} 
developed the analysis in \cite{CC04} to prove that the KdV system  is {\it finite time}, locally, exactly 
controllable  in the case $\dim \M_N = 2$.  Later, Cr\'epeau and Cerpa \cite{CC09}
succeeded to extend the ideas in \cite{Cerpa07} to obtain the local, exact 
controllability in {\it finite time} for all other critical lengths. Recently, with Coron and Koenig \cite{CKN20}, we prove that such a system is not small time, locally, null controllable for a class of critical lengths.  This fact is surprising when compared with known results on internal controls for the KdV equation. It is known, see \cite{CPR15, PVZ02, Pazoto05}, that the KdV system \eqref{sys-KdV} with $u=0$ is small time,  locally controllable using internal controls {\it whenever} the control region contains an {\it arbitrary}, open subset of $(0, L)$.  The power expansion method is also a starting point of the analysis in \cite{CKN20}. Part of the analysis is to characterize all controls that bring 0 at 0 to 0 at time T for the corresponding linearized system. 
This idea is then used in the study of the water tank problem \cite{CKN-WT}. 
It is interesting to note that there are other types of boundary controls for the KdV equation for which there is no critical length, see   \cite{Rosier97, Rosier04, GG10, Cerpa14}.
There are also results on the internal controllability for the KdV equation, see
\cite{RZ96, LRZ10, CPR15} and the references therein. A minimal time of the null controllability is also required for some linear partial differential equations.
This is the case for equations with a finite speed of propagation, such as
the transport equation, the wave equation,  or the hyperbolic system, see,  e.g., \cite{Coron07} and the references therein. But this can also happen for equations with the infinite speed of propagation, such as
some parabolic systems~\cite{BAM20}, Grushin-type equations \cite{BCG14, BMM15, Koenig17, BDE20}, Kolmogorov-type equations~\cite{BHHR15},   and the references therein.

We now turn back to the control problem \eqref{sys-KdV}. This control problem was first investigated by Glass and Guerrero \cite{GG10}. To this end, in the spirit of Rosier's work mentioned above, they introduced the corresponding set of critical lengths \footnote{The letter $D$ stands for the Dirichlet boundary control.} 
\begin{equation}\label{def-ND}
\cN_D = \Big\{L \in \mR_+; \, \exists z_1, z_2 \in \mC:  \eqref{cond-z1z2} \mbox{ holds} \Big\}, 
\end{equation}
where 
\begin{equation}\label{cond-z1z2}
z_1 e^{z_1} = z_2 e^{z_2} = - (z_1+ z_2) e^{-(z_1  + z_2)} \quad \mbox{ and } \quad  L^2 = -(z_1^2 + z_1 z_2 + z_2^2). 
\end{equation}
They proved that the set $\cN_D$ is infinite and has no accumulation point. Concerning \eqref{sys-KdV},  Glass and Guerrero proved that the corresponding linearized KdV system is small time, exactly controllable with initial and final data in $H^{-1}(0, L)$ using controls in $L^2(0, T)$  if $L \not \in \cN_D$. Developing this result, they also established that the KdV system \eqref{sys-KdV}  is small-time locally controllable for initial and final data in $L^2(0, L)$ and controls in $H^{1/6_-}$ \footnote{Controls in $H^{1/6_-}$ means controls in $H^{1/6 - \eps}$ for all $\eps > 0$.}  for {\it non-critical} lengths, i.e., $L \not \in \cN_D$.

Even the local controllability of the KdV system using boundary controls has been investigated extensively in the last two decades, to our knowledge, there is no result on the local controllability of system \eqref{sys-KdV} for critical lengths.  In comparison with the set of critical lengths $\cN_N$, the set $\cN_D$ is less explicit. Moreover, the unreachable space for the linearized system related to \eqref{sys-KdV}  has not been determined due to the lack of an appropriate observability inequality for the linearized system in a suitable functional setting. As revealed later in this paper, the control properties related to the unreachable space for the right Dirichlet control are quite distinct from the one for the right Neumann control.

\subsection{Statement of the main results}
The main goal of this paper is to give a rather complete picture of the local controllability of the control KdV system \eqref{sys-KdV}.  We first characterize all the critical lengths and describe the corresponding unreachable space $\cM_D$ for the corresponding linearized system. The first result in this direction is. 

\begin{theorem} \label{thm-MD} Let $L \in \cN_D$. There exists a unique pair $(a, b) \in \mR^2$ such that  $a < 0$, $b > 0$, 
\be \label{thm-MD-ab1}
(a + i b)e^{a + i b} = (a - ib) e^{a - i b} = - 2a e^{-2 a}, 
\ee
and 
\be \label{thm-MD-L}
L^2 = b^2 - 3 a^2. 
\ee
Set 
\be \label{thm-MD-q}
q = - \frac{2 a (a^2 + b^2)}{L^3} > 0, 
\ee
\be \label{thm-MD-phi}
\phi(x) = - \beta e^{\alpha x} \cos (\beta x) +  \beta e^{- 2 \alpha x} + 3 \alpha e^{\alpha x} \sin (\beta x)  \mbox{ for } x \in [0, L], 
\ee
with 
\be \label{thm-MD-alphabeta}
\mbox{$\alpha = - a / L \quad$ and $ \quad \beta = - b / L$},
\ee
and 
\be \label{thm-MD-Phi}
\Phi (t, x) = e^{q t} \phi(x) \mbox{ for } (t, x) \in \mR \times [0, L]. 
\ee
Then $\Phi$ is a solution of the system 
\be \label{sys-Phi}
\left\{\begin{array}{cl}
\Phi_t + \Phi_x  + \Phi_{xxx}  = 0 &  \mbox{ in }   \mR 
\times  (0, L), \\[6pt]
\Phi(\cdot, 0) = \Phi_x(\cdot, 0) = \Phi (\cdot, L) = \Phi_{xx}(\cdot, L)= 0 & \mbox{ in }  \mR, 
\end{array}\right.
\ee
and the unreachable space $\cM_D$ of the linearized system of \eqref{sys-KdV} is given by 
\be \label{def-MD}
\cM_D = \mbox{span } \Big\{ \phi \Big\}. 
\ee
Consequently, 
\be \label{dim-MD}
\dim \cM_D = 1. 
\ee
\end{theorem}

\begin{remark} \rm The main part of  \Cref{thm-MD} is to show that  the complex numbers $z_1$, $z_2$ from the definition of $\cN_D$ in \eqref{def-ND} can be chosen as 
$$
z_1 = a + ib \quad \mbox{ and } \quad z_2 = a - ib,  
$$
for some $a < 0$ and $b > 0$. This is non-trivial, see \Cref{sect-CL}. 
\end{remark}

\begin{remark} \rm  Implicit information given in \Cref{thm-MD} is that the function $\phi$, defined in \eqref{thm-MD-phi}, satisfies the following boundary condition: 
$$
\phi(0) = \phi_x(0) = \phi (L) = \phi_{xx}(L)= 0. 
$$
\end{remark}

Using \Cref{thm-MD}, we can precisely describe the set $\cN_D$ in the following result. 

\begin{proposition}\label{corCha-L} Let $L  > 0$. Then $L  \in \cN_D$ if and only if  $L = L_n$ for some $n \in \N$ where $L_n$ is given by 
$$
L_n^2 = 4 a_n^2 (e^{- 6 a_n} - 1), 
$$
with $a_n < 0$ being uniquely determined by 
$$
b_n \cos b_n + a_n \sin b_n = 0, 
$$
where $b_n$ satisfies 
$$
\pi + 2 n \pi <  b_n <    3 \pi/ 2 + 2 n \pi \quad \mbox{ and } \quad b_n^2 = 4 a_n^2 (e^{- 6 a_n} - 1/4). 
$$
\end{proposition}

After presenting the results on the unreachable space of the linearized system and the characterization of the critical lengths, we next discuss the local controllability of \eqref{sys-KdV}.  Set, for $T> 0$,  
\begin{equation}\label{def-XT}
X_T: = C\big([0, T]; L^2(0, L) \big) \cap L^2\big((0, T); H^1(0, L) \big)
\end{equation}
equipped with the corresponding norm: 
$$
\|y\|_{X_T} =\|y \|_{C\big([0, T]; L^2(0, L) \big)} + \| y\|_{L^2\big((0, T); H^1(0, L) \big)}. 
$$

We are ready to state the local controllability property of \eqref{sys-KdV} in small time.

\begin{theorem}\label{thm1} Let $L \in \cN_D$. There exist $T =T_0>0$ and $\eps_0 >0$ such that for all solutions $y \in X_{T}$  of the system
\begin{equation}\label{sys-y-O}\left\{
\begin{array}{cl}
y_t + y_x  + y_{xxx}  + y y_x = 0 &  \mbox{ in }   (0, T)
\times  (0, L), \\[6pt]
y(\cdot, 0) = y_x(\cdot, L) = 0 & \mbox{ in }  (0, T), \\[6pt]
y(\cdot, L) = u & \mbox{ in }  (0, T),  \\[6pt]
y(0, \cdot) = \eps \phi & \mbox{ in }  (0,L),
\end{array}\right.
\end{equation}
with $0< \eps < \eps_0$ and  $\| u \|_{H^{1/2}(0, T)} < \eps_0$, we have
\[
y(T, \cdot) \neq 0.
\]
\end{theorem}

Recall that $\phi$ is defined in \eqref{thm-MD-phi}.

\medskip 
We next present  a criterion on the local controllability property of \eqref{sys-KdV} in finite time. Set 
\be \label{def-Omega}
\Omega (z) =  \int_{0}^L \left|\sum_{j=1}^3 \big(\lambda_{j}e^{\lambda_{j} L } - \lambda_{j+1} e^{\lambda_{j+1} L } \big) e^{\lambda_{j+2} x} \right|^2 \phi_x(x) \, dx, 
\ee
where $\lambda_j = \lambda_j (z)$ with $j=1, 2, 3$ are the three solutions of the equation 
$$
\lambda^3 + \lambda + i z = 0
$$ 
and the convention  $\lambda_{j+3} = \lambda_j$ for $j \ge 1$ is used.  It is clear that $\Omega$ is continuous with respect to $z$. One can show that, for $c \in \mR$, 
$$
\lim_{|z| \to \infty, z \in \mR} \Omega (z + i c) = + \infty. 
$$

The criterion on the local controllability property of \eqref{sys-KdV} in finite time is stated in the following result. 

\begin{theorem} \label{thm2} Let $L \in \cN_D$ and let $q>0$ be defined by \eqref{thm-MD-q}. Set 
$$
\omega = \min_{z \in \mR}  \Omega (z + i q/2), 
$$
where $\Omega$ is defined by \eqref{def-Omega}. The following two facts hold. 

\medskip 
$i)$ If $\omega \ge 0$, then system \eqref{sys-KdV} is not locally null controllable in any positive time with controls in $H^{1/2}$. More precisely, given $T>0$ arbitrary, there exists $\eps_T > 0$ such that  for all solution of $y \in X_T$ of  \eqref{sys-y-O} with $0 < \eps < \eps_T$ and $\| u\|_{H^{1/2}(0, T)} < \eps_T$, we have 
$$
y(T, \cdot) \neq 0. 
$$

\medskip  $ii)$ If $\omega < 0$, then  system \eqref{sys-KdV} is locally exactly controllable in finite time with controls in $H^{1/3}$ and initial and final data in $L^2(0, L)$. More precisely,  there exist $T_0 > 0$ and $\eps_0 > 0$ such that for all $y_0, y_1 \in L^2(0, L)$  with $\| y_0\|_{L^2(0, L)} \le \eps_0$ and   $\| y_1\|_{L^2(0, L)} \le \eps_0$, there exists $u \in H^{1/3}(0, T_0)$ with  $\| u\|_{H^{1/3}(0, T_0)} \le C( \| y_0\|_{L^2(0, L)} + \| y_1\|_{L^2(0, L)} )^{1/2}$ such that 
$$
y(T_0, \cdot) = y_1, 
$$
where $y \in X_{T_0}$ is the unique solution of \eqref{sys-KdV}. Here $C$ denotes 
a positive constant independent of $y_0$ and $y_1$. 
\end{theorem}

\begin{remark} \label{rem-scilab} \rm Using Scilab, the program is given in \Cref{appendix-Scilab}, one can show that
\begin{itemize}
\item if $n=0$, then $L=4.5183604$,  $a=-0.5065520$, $b=4.6027563$,  $q=0.2354919$, and  $\Omega (iq/2) =-0.0140641$. 

\item if $n=1$, then $L=10.866906$,  $a=-0.6903700$,  $b=10.932497$,  $q=0.1291104$, and $\Omega (iq/2) = -0.0061256$. 

\item if $n=2$, then $L=17.177525$,  $a=-0.7947960$,  $b=17.232599$, $q=0.0933315$, and  $\Omega (iq/2) =-0.0036196$. 

\item if $n=3$, then $L=23.476776$,  $a=-0.8687610$,  $b=23.524949$,  $q=0.0744156$, and  $\Omega (iq/2) = -0.0024525$. 

\end{itemize}
Parameters of the program ($k$, $j$, $j0$) are given for $n=0$. The parameters for $n=1, 2, 3$ are also given there. 
\end{remark}

As a consequence of \Cref{thm1} and \Cref{thm2} (see also \Cref{rem-scilab}), there exist critical lengths for which system \eqref{sys-KdV} is not locally null controllable in small time but locally exactly controllable in finite time even the unreachable space of the corresponding linearized system is of dimension 1. This result is surprising and distinct when compared with known related results on the local controllability of the  KdV equation. First,  it is known that the corresponding KdV system with the right zero Dirichlet is locally exactly controllable in small time using internal controls. Second, the unreachable space of the linearized system of
the corresponding KdV system with the right Neumann control might be of arbitrary dimension. Third,   the KdV system with the right Neumann control is locally exactly controllable in small time if the corresponding unreachable space is of dimension 1.

\medskip 
Finally,  in the case $L \not \in \cN_D$, we can prove the following result, which sharpens the results in \cite{GG10} mentioned previously.  

\begin{theorem}\label{thm3} Let $L  \not \in \cN_D$ and $T>0$. There exist  $C>0$ and $\eps_0 > 0$ depending only on $L$ and $T$ such that  for all $y_0, \, y_1 \in L^2(0, L)$ with $\| y_0\|_{L^2(0, L)} \le \eps_0$ and   $\| y_1\|_{L^2(0, L)} \le \eps_0$, there exists $u \in H^{1/3}(0, T)$ with  $\| u\|_{H^{1/3}(0, T)} \le C( \| y_0\|_{L^2(0, L)} + \| y_1\|_{L^2(0, L)} )$ such that 
$$
y(T, \cdot) = y_1, 
$$
where $y \in X_T$ is the unique solution of \eqref{sys-KdV}. 
\end{theorem}

\subsection{Ideas of the proofs}

The characterization of the critical lengths and the unreachable space is based on the study of system \eqref{cond-z1z2}. In comparison with the characterization of the set $\cN_N$, the study of $\cN_D$ is more complex. The main goal and the difficulty in this direction, see \Cref{sect-CL},  are to obtain/characterize {\it all} solutions of \eqref{cond-z1z2}. The existence of a class of solutions is easier as observed previously in \cite{GG10}. As a consequence, we derive that $\dim \cM_D = 1$ which is completely different from what is known for the right Neumann control. 

Our approaches to \Cref{thm1} and \Cref{thm2} are inspired by the power
series expansion method introduced by Coron and Cr\'epeau \cite{CC04}.
The idea of this method is to search for/understand a control $u$ of the form
$$
u = \eps u_1 + \eps^2 u_2 +  \cdots.
$$
The corresponding solution then formally has the form
$$
y = \eps y_1 + \eps^2 y_2  + \cdots, 
$$
and the non-linear term $y y_x$ can be written as
$$
y y_x = \eps^2 y_1 y_{1, x} + \cdots.
$$
One then obtains the following systems for $y_1$ and $y_2$: 
\begin{equation}\left\{
\begin{array}{cl}
y_{1, t}  + y_{1, x}  + y_{1, xxx}  = 0 &  \mbox{ in } 
(0, T) \times (0, L), \\[6pt]
y_1(\cdot, 0) = y_{1, x} (\cdot, L) = 0 & \mbox{ in }  (0, T), \\[6pt]
y_{1}(\cdot, L) = u_1 & \mbox{ in } (0, T),
\end{array}\right.
\end{equation}
\begin{equation}\left\{
\begin{array}{cl}
y_{2, t} + y_{2, x}  + y_{2, xxx}  + y_1  y_{1, x}   = 0 &
\mbox{ in } (0, T) \times (0, L), \\[6pt]
y_2(\cdot, 0) = y_{2, x}(\cdot, L)  = 0 & \mbox{ in }  (0, T), \\[6pt]
y_{2}(\cdot, L) = u_2 & \mbox{ in } (0, T). 
\end{array}\right.
\end{equation}
The idea   is then to find the corresponding controls $u_1$ and $u_2$ such that, if  $y_1(0, \cdot)=y_2(0, \cdot)=0$, then $y_1(T, \cdot) = 0$ and the $L^2$-orthogonal projection of  $y_2(T, \cdot)$ on $\M_D$ is a given (non-zero) element in $\M_D$. To this end,  in \cite{CC04, Cerpa07, CC09}, the authors 
used delicate contradiction arguments to capture the structure of their studied KdV system.

We here use the ideas in the spirit of the joint work with Coron and Koenig \cite{CKN20}.  The starting point of the analysis is also the power
series expansion method. The strategy is to characterize all possible
$u_1$ which steers 0 at time 0 to $0$ at time $T$ (\Cref{pro-Gen}).  This is done by taking the  Fourier transform with respect to time of the solution $y_1$  and applying Paley-Wiener's theorem. We then investigate the projection of $y_2$ into the unreachable space $\cM_D$ (in \Cref{sect-0-0}).  This is where our analysis deviates from \cite{CKN20} to take the advantage of the fact that $q$ is real and to use the characterization via Paley-Wiener's theorem to construct $u_1$ for which the projection of $y_2$ into $\cM_D$ has a desired direction (see the proof of Assertion $ii)$ of \Cref{thm2}). Concerning the projection of $y_2$ into $\cM_D$, we use the information of  $|\hat u(z + i q/2)|^2$ instead of $\hat u(z) \overline{\hat u(z + i q)}$ as suggested in \cite{CKN20}. This new proposal has two advantages. First, the way to obtain 
the projection of $y_2$ into the unreachable space $\cM_D$ is less complex.  More importantly, it allows us to obtain the criterion on the local controllability of \eqref{sys-KdV} in finite time given in \Cref{thm2} by taking into account the fact that $|\hat u(z + i q/2)|^2 \ge 0$. Another important part in the proof of the local controllability in \Cref{thm1} and \Cref{thm2} is to obtain the observability for the linearized system with initial/final data in  $\cM_D^\perp$. To this end, one needs to establish the optimal results on the stability and the well-posedness of the linearized system of \eqref{sys-KdV} and related ones in which the boundary conditions are in fractional Sobolev spaces, see \Cref{sect-Pro-KdV}.  The proof of the local controllability in finite time given in Assertion $ii)$ of \Cref{thm2} is based on a new strategy to handle the lack of symmetry in our control system in comparison with the one using the right Neumann control. Concerning the KdV system using the right Neumann control, one can bring a state forward or backward easily thanks to the symmetry of the boundary conditions. This is not the case for  \eqref{sys-KdV} since the Dirichlet condition on the right is controlled and the Dirichlet condition on the left is imposed to be 0. 
To overcome this, we introduce a new approach that involves a KdV system with a new type boundary condition, investigated in \Cref{sect-Pro-KdV},  in a backward way.  Our analysis uses the Banach fixed point arguments while the Brouwer fixed point theorem is usually applied to related contexts. It is worth noting that the analysis in this paper is based on the expansion up to the second order even if the unreachable space is of dimension 1. This explains the cost of the control given in Assertion $ii)$ of \Cref{thm2}. This way is different from what has been done for the KdV system with the right Neumann control where the expansion up to the third order is required. 

\medskip 
Here are other comments on the analysis. 

\medskip 
$i)$ Various interpolation inequalities involving fractional Sobolev spaces and their dual spaces are established and used in the proof of \Cref{thm1} and \Cref{thm2}.

$ii)$ The analysis for the well-posedness of the KdV equation with various boundary conditions is in the spirit of the joint work with Coron and  Koenig \cite{CKN20}. This partly involves a connection between the linear KdV equation and the linear KdV-Burgers equation as previously used by Bona et al. \cite{Bona09} and inspired by the work of Bourgain \cite{Bourgain93}, and  Molinet and Ribaud \cite{MR02}. In this work, 
to deal with the boundary conditions containing the information of $y_{xx}$  and derive {\it optimal} estimates on $y$, $y_x$, and $y_{xx}$ in fractional Sobolev spaces and their dual spaces, new ideas and technique are implemented (see \Cref{rem-RW}).  It is worth noting that similar estimates for $y_{xx}$ are not known for the KdV equation in the real line setting and it is not clear whether such estimates hold for that setting. Our estimates particularly allow us to get the observability inequality for the critical lengths, which plays a role in determining the unreachable space for the linearized KdV system (see \Cref{sect-Pro-KdV}). The unreachable space is not known previously and this determination is also a contribution of the paper.   

$iii)$ The arguments used in this paper to disprove the small time, local controllability improve the ones in \cite{CKN20}. One of the key steps in the proof is to establish new (optimal) estimates for solutions of the KdV equations with various boundary conditions (see \Cref{sect-Pro-KdV}). The proof also relies on the positivity of a scalar product which comes naturally in the study of the controllability of the KdV system for small time.  The positivity of a scalar product for small time was also a crucial point of several lacks of small time, local controllability results for systems with infinite propagation speed. Nevertheless, the previous ways used to derive the positivity are different, see, e.g., \cite{Coron06, Marbach18, BM20}.

\subsection{Organisation of the paper}
The paper is organized as follows. \Cref{sect-KdV-B} is devoted to the study of the linear KdV-Burgers equation with the periodic boundary condition.  \Cref{sect-Pro-KdV} is devoted to the study of the (linear and nonlinear) KdV equations with various boundary conditions. \Cref{sect-CL} is devoted to the study of critical lengths and some properties of the unreachable space $\cM_D$.  We there prove \eqref{thm-MD-ab1}, \eqref{thm-MD-L}, and \eqref{sys-Phi}, and establish \Cref{corCha-L}. In \Cref{sect-US}, we prove \eqref{def-MD} and establish the corresponding observability inequality (\Cref{lem-obs}).  The proof of \Cref{thm3} is given at the end of this section.  In \Cref{sect-0-0}, we study properties of the controls which steer 0 at time 0 to 0 at time T. Using results in \Cref{sect-0-0}, we study attainable directions for small time in \Cref{sect-dir}. The proofs of \Cref{thm1} and  \Cref{thm2}  are given in \Cref{sect-thm1} and \Cref{sect-thm2}, respectively. Several technical results are stated and proved in the appendix. The Scilab program is also given there.

\section{Linear KdV-Burgers equations} \label{sect-KdV-B}

This section is devoted to the study of the linear KdV-Burgers equations with the periodic boundary condition.  Here is the main result of this section, which plays an essential role in the study of the linearized KdV equation with various boundary conditions investigated in \Cref{sect-Pro-KdV}.

\begin{proposition}\label{pro-kdv-B} Let $g \in L^1\big(\mR_+; L^2 (0, L)\big)$ with $\int_0^L g(t, x) \, dx = 0$ for a.e. $t > 0$, and let $y_0 \in L^2(0, L)$ be such that $\int_0^L y_0 (x) \, dx = 0$.  There exists a unique solution $y \in C\big([0, + \infty); L^2(0, L) \big) \cap L^2\big([0, +\infty); H^1([0, L]) \big)$ which is periodic in space of the system 
\be \left\{
\begin{array}{cl}
y_t  + 4 y_{x} + y_{xxx} - 3 y_{xx} = g  & \mbox{ in } \mR_+ \times (0, L), \\[6pt]
y(0, \cdot) = y_0 & \mbox{ in } (0, L). 
\end{array} \right. 
\ee 
Moreover,  for $t \in \mR_+$ and  $x \in [0, L]$,   
\begin{multline}\label{lem-kdvB-cl3}
\| y(t, \cdot) \|_{L^2(0, L)} + \| y(\cdot, x)\|_{H^{1/3}(\mR_+)} + \| y_x(\cdot, x)\|_{L^2(\mR_+)} + \| y_{xx}(\cdot, x)\|_{[H^{1/3}(\mR_+)]^*}  \\[6pt]
\le C \Big(\| y_0\|_{L^2(0, L)} +   \|g \|_{L^1\big(\mR_+; L^2(0, L)\big)} \Big), 
\end{multline}
and
\be\label{lem-kdvB-cl4}
\|y(\cdot, x)\|_{L^2(\mR_+)} +   \| y_x(\cdot, x)\|_{[H^{1/3}(\mR_+)]^*}  
\le C \Big(\| y_0\|_{H^{-1}(0, L)} +   \|g \|_{L^1\big(\mR_+; H^{-1}(0, L)\big)} \Big)
\ee
for some positive constant $C$ depending only on $L$. 
\end{proposition}

The rest of this section consisting of two subsections is organized as follows. In the first subsection, we present and prove a useful lemma in the spirit of Bourgain's. The second subsection is devoted to the proof of \Cref{pro-kdv-B}.  

\subsection{A useful lemma}

The following result in the spirit of Bourgain's is used in the proof of \Cref{pro-kdv-B}.  

\begin{lemma} \label{lem-A} Let $(a_n)_{n \in \mZ \setminus \{0 \}} \subset \mR$ be such that $\sum_{n \in \mZ \setminus \{0 \}} a_n^2 < + \infty$.  For $- 2 <  s \le 3$, there exists a positive constant $C = C_s$  such that 
\be
\int_{\mR} (1 + |z|)^{2s/3}  \left(\sum_{n \in \mZ \setminus \{0 \}} \frac{|a_n| |n|^{3-s} }{(z + 4n - n^3)^2 + n^4} \right)^2 \, dz \le C \sum_{n \in \mZ \setminus \{0 \}} a_n^2. 
\ee
\end{lemma}

\begin{proof} For $s > -5/2$ and  for $-1 \le z  \le 1$, one has 
$$
 \sum_{n \in \mZ \setminus \{0 \}} \frac{|a_n| |n|^{3-s} }{(z + 4n - n^3)^2 + n^4} \, \le C \left( \sum_{n \in \mZ \setminus \{0 \}} a_n^2 \right)^{1/2}.  
$$
It hence suffices to prove that, for $- 2 <  s \le 3$,  
\be\label{lem-A-c}
\int_{z \in \mR; |z| > 1} (1 + |z|)^{2s/3}  \left(\sum_{n \in \mZ \setminus \{0 \}} \frac{|a_n| |n|^{3-s} }{(z + 4n - n^3)^2 + n^4} \right)^2 \, dz \le C \sum_{n \in \mZ \setminus \{0 \}} a_n^2. 
\ee

We have
\begin{multline}\label{lem-A-p0}
\int_{z > 1} (1 + |z|)^{2s/3}  \left(\sum_{n \in \mZ \setminus \{0 \}} \frac{|a_n| |n|^{3-s} }{(z + 4n - n^3)^2 + n^4} \right)^2 \, dz \\[6pt]
 \le C \sum_{m \in \N_*} \int_{m^3}^{(m+1)^3} m^{2 s} \left(\sum_{n \in \mZ \setminus \{0 \}} \frac{|a_n| |n|^{3-s} }{(z + 4n - n^3)^2 + n^4} \right)^2  \, dz.
\end{multline}
For $m^3 \le z \le (m+1)^3$,  one gets
\begin{multline}\label{lem-A-p1}
\sum_{n \in \mZ \setminus \{0 \}} \frac{|a_n| |n|^{3-s} }{(z + 4n - n^3)^2 + n^4}  \\[6pt]
=   \sum_{k; \,  m + k \in \mZ \setminus \{0 \}} \frac{|m+k|^{3 -s } |a_{m+k}|}{(z + 4(m+k) - (m+k)^3)^2 + (m+k)^4} \\[6pt]
\le  C \mathop{\sum_{k; |k| \le 2 m}}_{\, m + k \in \mZ \setminus \{0 \}}  \frac{|m+k|^{3 -s } |a_{m+k}|}{m^4 (|k|+1)^2} + C \mathop{\sum_{k; \; |k| \ge 2m + 1}}_{m + k \in \mZ \setminus \{0 \}}  \frac{|k|^{3 -s } |a_{m+k}|}{k^4 (|k|+1)^2}.  
\end{multline}
Combining \eqref{lem-A-p0} and \eqref{lem-A-p1} yields  
\begin{multline}\label{lem-A-p2}
\int_{z > 1} (1 + |z|)^{2s/3}  \left(\sum_{n \in \mZ \setminus \{0 \}} \frac{|a_n| |n|^{3-s} }{(z + 4n - n^3)^2 + n^4} \right)^2 \, dz \\[6pt]
\le C \sum_{m \in \N_*}   \left( \mathop{\sum_{k; |k| \le 2 m}}_{m + k \in \mZ \setminus \{0 \}}  \frac{|m+k|^{3 -s } |a_{m+k}|}{m^{3-s} (|k|+1)^2} + C \mathop{\sum_{k; \; |k| \ge 2m + 1}}_{m + k \in \mZ \setminus \{0 \}}  \frac{|m|^{1+s}  |a_{m+k}|}{|k|^{1 + s} (|k|+1)^2} \right)^2 \\[6pt] 
\le C   \sum_{m \in \N_*} \left( \mathop{\sum_{k}}_{m+k \in \mZ \setminus \{0\}} \frac{|a_{m+k}|^2}{(|k| + 1)^{1+\eps}} \right) \left( \mathop{\sum_{k; |k| \le 2 m}}_{m+k \in \mZ \setminus \{0\}} \frac{|m+k|^{6-2s}}{m^{6 - 2s}(|k|+1)^{3 -\eps}} + \mathop{\sum_{k; |k| \ge 2m +1}}_{m+k \in \mZ \setminus \{0\}} \frac{m^{2+2s} }{|k|^{2+2s}(|k|+1)^{3 -\eps}} \right).
\end{multline}
Here $0< \eps \le 1$ is fixed such that $s > (\eps - 4)/2$ (the fact $s > -2$ is used).  We have, for $m \in \N_*$, $|k| \le 2m$,  and $m + k \in \mZ \setminus \{0 \}$, 
$$
\frac{|m+k|^{6-2s}}{m^{6-2s}} \le C \mbox{ for } s \le 3,  
$$
and, for $m \in \N_*$, 
$$
\mathop{\sum_{k; |k| \ge 2m +1}}_{m+k \in \mZ \setminus \{0\}} \frac{m^{2+2s} }{|k|^{2+2s}(|k|+1)^{3 -\eps}} = m^{2+2s}  \mathop{\sum_{k; |k| \ge 2m +1}}_{m+k \in \mZ \setminus \{0\}} \frac{1}{k^{2+2s}(|k|+1)^{3 - \eps}} \le \frac{C}{m^{2-\eps}} \le C \mbox{ for } s > (\eps-4)/2. 
$$

Using the fact, for $m \in \N_*$,  
$$
\sum_{k; \, m+k \in \mZ \setminus \{0\}} \frac{1}{(|k|+1)^{3 - \eps}}  < + \infty \mbox{ for } 0 < \eps < 1, 
$$
we derive from \eqref{lem-A-p2} that 
\begin{multline}
\int_{z > 1} (1 + |z|)^{2s/3}  \left(\sum_{n \in \mZ \setminus \{0 \}} \frac{|a_n| |n|^{3-s} }{(z + 4n - n^3)^2 + n^4} \right)^2 \, dz \\[6pt]
  \le C      \sum_{m \in \N_*}  \sum_{k; \, m+k \in \mZ \setminus \{0\}} \frac{|a_{m+k}|^2}{(|k| + 1)^2} \le C      \sum_{k \in \mZ}  \frac{1}{(|k|+1)^2} \sum_{m ; \,  m+k \in \mZ \setminus \{0\}}|a_{m+k}|^2, 
\end{multline}
which yields 
\be\label{lem-A-a}
\int_{z > 1} (1 + |z|)^{2s/3}  \left(\sum_{n \in \mZ \setminus \{0 \}} \frac{|a_n| |n|^{3-s} }{(z + 4n - n^3)^2 + n^4} \right)^2 \, dz  \le C \sum_{n \in \mZ \setminus \{0 \}} |a_n|^2. 
\ee

Similarly, we have 
\be\label{lem-A-b}
\int_{z < - 1} (1 + |z|)^{2s/3}  \left(\sum_{n \in \mZ \setminus \{0 \}} \frac{|a_n| |n|^{3-s} }{(z + 4n - n^3)^2 + n^4} \right)^2 \, dz  \le C \sum_{n \in \mZ \setminus \{0 \}} |a_n|^2. 
\ee
Estimate \eqref{lem-A-c} now follows from \eqref{lem-A-a} and \eqref{lem-A-b}. The proof is complete.
\end{proof}

\subsection{Proof of \Cref{pro-kdv-B}}
We only derive the estimates. The uniqueness follows from the estimates and the existence follows from the proof of these estimates as well.  

Multiplying the equation of $y$ by $y$ and integrating by parts, we have 
$$
\frac{1}{2}\frac{d}{dt} \int_0^L |y(t, x)|^2 \, dx + 3 \int_0^L |y_x(t, x)|^2 \, dx= \int_0^L g(t, x) y(t, x) \, dx. 
$$
This yields 
$$
\frac{1}{2}\frac{d}{dt} \int_0^L |y(t, x)|^2 \, dx \le  \|g(t, \cdot) \|_{L^2(0, L)} \| y(t, \cdot) \|_{L^2(0, L)}. 
$$
Applying the Gr\"onwall lemma, we obtain the desired estimate for  $\| y(t, \cdot) \|_{L^2(0, L)}$.

We next  establish the estimates for $\| y(\cdot, x)\|_{H^{1/3}(\mR_+)}$, $\| y_x(\cdot, x)\|_{L^2(\mR_+)}$, and $\| y_{xx}(\cdot, x)\|_{[H^{1/3}(\mR_+)]^*}$.  For notational ease, we assume that $L = 2 \pi$. It suffices to consider the case $g \equiv 0$ and the case  $y_0 \equiv 0$ separately. 

\medskip 
We first consider the case $g \equiv 0$. 
Write the solution under the form 
$$
y(t, x) = \sum_{n \in \mZ } a_n(t) e^{i n x} \mbox{ in } \mR_+ \times (0, L) = \mR_+ \times (0, 2 \pi).  
$$ 
We then derive that 
$$
y_0 (x) =  \sum_{n \in \mZ} a_n(0) e^{i  n x} \mbox{ for } x \in [0, 2 \pi]
$$ 
and 
$$
a_n'(t) = \big(- 3 n^2 - i (4n - n^3) \big) a_n (t) \mbox{ for } t \in \mR_+. 
$$
We thus have 
$$
a_n(t) = e^{\big(- 3 n^2 - i (4n - n^3) \big) t} a_n(0) \mbox{ for } t \in \mR_+, n \in \mZ. 
$$
Since $\int_0^L y_0 \, dx = 0$, it follows that $
a_0 = 0$.  This in turn implies that $a_0(t) = 0$ for $t \in \mR_+$. 
We hence obtain 
$$
y(t, x) = \sum_{n \in \mZ \setminus \{0 \}}  e^{i n x} e^{\big(- 3 n^2 - i (4n - n^3) \big) t} a_n(0)  \mbox{ in } \mR_+ \times (0, 2 \pi).  
$$

Extend $y(t, x)$ for $t < 0$ by 
\be\label{pro-kdv-B-y}
y(t, x) = \sum_{n \in \mZ \setminus \{0 \}}  e^{i n x} e^{\big(3 n^2 - i (4n - n^3) \big) t} a_n(0)  \mbox{ in } \mR_- \times (0, 2 \pi),  
\ee
and still denote this extension by $y(t, x)$. We have, for $z \in \mR$ and $n \in \mZ \setminus \{0\}$,  
\begin{multline}\label{lem-kdv-B-Fourier}
\int_0^\infty e^{\big(-3 n^2 - i (4n - n^3) \big) t} e^{- i t z} \, dt + \int_{-\infty}^0 e^{\big(3 n^2  - i (4n - n^3) \big) t} e^{- i t z} \, dt \\[6pt]
= \frac{1}{3 n^2 + i (4n - n^3 + z)} +  \frac{1}{3 n^2 - i (4n - n^3 + z)} = \frac{6 n^2}{9n^4 + (4n - n^3 + z)^2}. 
\end{multline}
This implies 
\be
\| y(\cdot, x) \|_{H^{1/3}(\mR_+)}^2 \le C \int_{z \in \mR} (1 + |z|)^{2/3} \left(\sum_{n \in \mZ \setminus 
\{0 \}} \frac{n^2 |a_n(0)|}{n^4 + (z + 4n - n^3)^2}\right)^2 \, dz. 
\ee
Applying \Cref{lem-A}  with $s = 1$,  we obtain  
\be
\| y(\cdot, x) \|_{H^{1/3}(\mR_+)}^2 \le C \| y_0\|_{L^2(0, 2 \pi)}^2.  
\ee

Similarly, we have 
\be
\| y_x(\cdot, x) \|_{L^2(\mR_+)}^2 \le C \int_{z \in \mR}  \left(\sum_{n \in \mZ \setminus 
\{0 \}} \frac{|n|^3 |a_n(0)|}{(z + 4n - n^3)^2 + n^4}\right)^2 
\ee
and 
\be
\| y_{xx}(\cdot, x) \|_{[H^{1/3}(\mR_+)]^*}^2 \le C \int_{z \in \mR} (1 + |z|)^{-2/3} \left(\sum_{n \in \mZ \setminus 
\{0 \}} \frac{n^4 |a_n(0)|}{(z + 4n - n^3)^2 + n^4}\right)^2. 
\ee
Applying \Cref{lem-A}  with $s = 0$ and $s=-1$, we get 
\be
\| y_x(\cdot, x) \|_{L^2(\mR_+)}^2 \le C \| y_0\|_{L^2(0, 2 \pi)}^2.  
\ee
and 
\be
\| y_{xx}(\cdot, x) \|_{[H^{1/3}(\mR_+)]^*}^2 \le C \| y_0\|_{L^2(0, 2 \pi)}^2.  
\ee
The proof of \eqref{lem-kdvB-cl3} in the case $g \equiv 0$ is complete. 

We next deal with $\eqref{lem-kdvB-cl4}$. From \eqref{pro-kdv-B-y} and \eqref{lem-kdv-B-Fourier}, we have 
$$
\| y(\cdot, x) \|_{L^2(\mR_+)}^2 \le C \int_{z \in \mR}  \left(\sum_{n \in \mZ \setminus 
\{0 \}} \frac{|n|^3 (|a_n(0)|/|n|))}{n^4 + (z + 4n - n^3)^2}\right)^2 \, dz  
$$
and
\be
\| y_x(\cdot, x) \|_{[H^{1/3}(\mR_+)]^*}^2 \le C \int_{z \in \mR}  (1 + |z|)^{-2/3}\left(\sum_{n \in \mZ \setminus 
\{0 \}} \frac{|n|^4 (|a_n(0)|/|n|) }{(z + 4n - n^3)^2 + n^4}\right)^2.  
\ee
Applying \Cref{lem-A} with $s =0$ and $s=-1$, we obtain 
\be \label{lem-kdvB-cl4-p1}
\| y(\cdot, x) \|_{L^2(\mR_+)}^2 \le C  \sum_{n \in \mZ \setminus \{0\}} \frac{|a_n(0)|^2}{n^2} \le C \| y_0\|_{H^{-1}(0, 2 \pi)}^2  
\ee
and 
\be \label{lem-kdvB-cl4-p2}
\| y_x(\cdot, x) \|_{[H^{1/3}(\mR_+)]^*}^2 \le C  \sum_{n \in \mZ \setminus \{0\}} \frac{|a_n(0)|^2}{n^2} \le C \| y_0\|_{H^{-1}(0, 2 \pi)}^2.  
\ee
This yields \eqref{lem-kdvB-cl4}.

We next deal with the case $y_0 \equiv 0$. The proof, in this case, can be derived from the previous case as follows.  For $t > 0$, let $W(t)$ be the mapping which maps $y_0$ to $y(t, \cdot)$ with $g \equiv 0$. We then have \footnote{This identity can be derived from \eqref{rem-pro-kdv-B2-p1} and \eqref{rem-pro-kdv-B2-p2} in \Cref{rem-pro-kdv-B2}.}
\be
y(t, x) = \int_0^t W(t-s) g(s, x) \, ds. 
\ee
This implies
\be
y_x(t, x) = \int_0^t \partial_x \Big(W(t-s) g(s, x) \Big) \, ds = \int_0^{\infty} \mathds{1}_{(0, t)}(s)\partial_x \Big(W(t-s) g(s, x) \Big) \, ds.
\ee
Hence 
\begin{align*}
\| y_x(\cdot, x) \|_{L^2_t(\mR_+)}   & \le \int_0^{+ \infty} \| \mathds{1}_{(0, t)}(s) \partial_x \big(W(t-s) g(s, x) \big) \|_{L^2_t(\mR_+)} \, ds \\[6pt]
& = \int_0^{+ \infty} \left( \int_0^{\infty} |\mathds{1}_{(0, t)} (s) \partial_x \big(W(t-s) g(s, x) \big) |^2\, d t \right)^{1/2} \, ds \\[6pt]
& =  \int_0^{+ \infty} \left( \int_s^{\infty} | \partial_x \big(W(t-s) g(s, x) \big) |^2 \, d t \right)^{1/2} \, ds. 
\end{align*}
By applying the results in the previous case, we have 
$$
\left( \int_s^{\infty} | \partial_x \big( W(t-s) g(s, x) \big) |^2 \, d t \right)^{1/2} \le C\| g(s, x)\|_{L^2_x(0, L)}. 
$$
We thus obtain 
\be \label{lem-kdv-B-c1}
\| y_x(\cdot, x) \|_{L^2_t(\mR_+)}   \le  C \int_0^{+ \infty} \| g(s, x)\|_{L^2_x} \, ds. 
\ee

By the same arguments,  we have
\begin{multline}\label{lem-kdv-B-c2}
\| y_{xx}(\cdot, x) \|_{[H^{1/3}_t(\mR_+)]^*}   \le \int_0^{+ \infty} \| \mathds{1}_{(0, t)} (s) \partial_{xx} \big( W(t-s) g(s, x) \big) \|_{[H^{1/3}_t(\mR_+)]^*} \, ds \\[6pt]
= \int_0^{+ \infty} \| \partial_{xx} \big( W(t-s) g(s, x) \big) \|_{[H^{1/3}_t (s, + \infty)]^*} \, ds  \le 
C \int_0^{+ \infty} \| g(s, x)\|_{L^2_x(0, L)} \, ds. 
\end{multline} 

Similarly, we obtain 
\begin{multline}\label{lem-kdv-B-c3}
\| y(\cdot, x) \|_{H^{1/3}_t(\mR_+)}   \le \int_0^{+ \infty} 
\| \mathds{1}_{(0, t)}(s) W(t-s) g(s, x) \|_{H^{1/3}_t(\mR_+)} \, ds \\[6pt]
= \int_0^{+ \infty} \|  W(t-s) g(s, x) |\|_{H^{1/3}_t (s, + \infty)} \, ds  \le 
C \int_0^{+ \infty} \| g(s, x)\|_{L^2_x(0, L)} \, ds. 
\end{multline} 

Assertion \eqref{lem-kdvB-cl3} in the case $y_0 \equiv 0$ now follows from \eqref{lem-kdv-B-c1}, \eqref{lem-kdv-B-c2}, and 
\eqref{lem-kdv-B-c3}.

Assertion \eqref{lem-kdvB-cl4} in the case $y_0 \equiv 0$ follows similarly and the details are omitted. 

\medskip 
The proof is complete.  \qed

\begin{remark} \rm The proof gives as well that 
$$
 y \in C([0, L]; H^{1/3}(0, +\infty)), \; \; y_x \in C([0, L]; L^2(0, +\infty)), \; \;  \mbox{ and } \; \;    y_{xx} \in C([0, L]; [H^{1/3}(0, +\infty)]^*).
$$ 
\end{remark}

\begin{remark}\label{rem-pro-kdv-B1} \rm Assume that $g =0$ in $\mR_+ \times (0, L)$ and $y_0 \in C^\infty([0, L])$ is such that  $\int_{0}^L y_0(x) \, dx  = 0$,  $y_0 - c \in C^\infty_c((0, L))$ for some constant $c$. Using  \eqref{pro-kdv-B-y}, one can show that $y \in C^\infty([0, + \infty) \times [0, L])$. Moreover, using the equation of $y$, one can show that $\partial_t^k y(0, x) = 0$ for all $k \ge 1$.  
\end{remark}

\begin{remark}\label{rem-pro-kdv-B2} \rm Assume that $y_0 =0$ in $(0, L)$ and $g \in C^\infty_c\big( (0, + \infty) \times [0, L] \big)$ being such that $\int_0^L g(t, x) \, dx =0$ for $t > 0$. One can prove that the solution is then smooth and is 0 around the time $0$.  Indeed, for notational ease, assume that $L = 2 \pi$. One then can show that 
\be \label{rem-pro-kdv-B2-p1}
y(t, x) = \sum_{n \in \mZ \setminus 0} e^{i n x }\int_0^t e^{\big(-3 n^2 - i (4n - n^3) \big) (t-s)} g_n(s) \, ds,  
\ee  
where 
\be \label{rem-pro-kdv-B2-p2}
g(t, x) = \sum_{n \in \mZ \setminus \{0 \}} g_n(t) e^{in x}. 
\ee
The properties of $y$ follow. 
\end{remark}

\begin{remark} \rm There is no term $\| y_{xx}\|_{[H^{2/3}(\mR_+)]^*}$ in \eqref{lem-kdvB-cl4} 
since we applied  \Cref{lem-A}, which requires $s > -2$.
\end{remark}

\begin{remark} \rm In \cite[Lemma 4.1]{CKN20}, we proved that, when $y_0 = 0$, 
\be
\|y(\cdot, x)\|_{L^2(\mR_+)} \le C \|g \|_{L^1\big(\mR_+ \times (0, L)\big)}. 
\ee
Estimate \eqref{lem-kdvB-cl4} improves this. 
\end{remark}

\section{Linearized KdV equations} \label{sect-Pro-KdV}

In this section, we establish several well-posedness results (\Cref{pro-kdv1}, \Cref{pro-kdv2}, and \Cref{pro-kdv3}) for the linearized KdV equation equipped with various boundary conditions in the energy space $X_T$ defined in \eqref{def-XT} and present some of their consequences. \Cref{pro-kdv1} implies the local well-posedness of \eqref{sys-KdV} in $X_T$ (\Cref{pro-WP}).  \Cref{pro-kdv1} and  \Cref{pro-kdv2} are the starting points of our analysis in deriving the unreachable space and establishing the corresponding observability inequality in \Cref{sect-US} (\Cref{pro-C}). \Cref{pro-kdv3} and its consequence (\Cref{pro-kdv3-NL}) will be used in the proof of \Cref{thm2}.

\medskip 
Here is the first main result of this section. 

\begin{proposition}\label{pro-kdv1} Let $L>0$,  $T>0$, $(h_1, h_2, h_3) \in H^{1/3} (0, T) \times
H^{1/3}(0, T) \times L^2 (0, T)$,  $ f \in L^1\big((0, T); L^2(0, L)\big)$, and $y_0 \in L^2(0, L)$. There exists a unique solution $y \in X_T$  of the system
\begin{equation}\label{sys-y-LKdV}\left\{
\begin{array}{cl}
y_t + y_x  + y_{xxx}  = f &  \mbox{ in } (0, T) \times  (0, L),
\\[6pt]
y(\cdot, 0) = h_1,  \;  y(\cdot, L) = h_2, \;  y_x(\cdot, L)   = h_3 & \mbox{ in } (0,
T), \\[6pt]
y(0, \cdot)  = y_0 & \mbox{ in } (0, L).
\end{array}\right.
\end{equation}
Moreover, for $x \in [0, L]$,
\begin{multline}\label{pro-kdv1-cl1} 
\| y\|_{X_T} + \| y(\cdot, x)\|_{H^{1/3}(0, T)} + \| y_x(\cdot, x)\|_{L^2(0, T)}   + \| y_{xx}(\cdot, x)\|_{[H^{1/3}(0, T)]^*}  \\[6pt]
\le C_{T, L} \Big( \| y_0\|_{L^2(0, L)} + \| f\|_{L^1\big( (0, T);  L^2(0, L) \big)} +  \| (h_1, h_2) \|_{H^{1/3}(0, T)} + \|
h_3 \|_{L^2(0, T)}\Big),  
\end{multline}
and
\begin{multline} \label{pro-kdv1-cl2} 
\| y(\cdot, x)\|_{L^2(0, T)} + \| y_x(\cdot, x)\|_{[H^{1/3}(0, T)]^*} \\[6pt]
\le C_{T, L} \Big( \| y_0\|_{[H^{1}(0, L)]^*} + \| f\|_{L^1\big( (0, T);  [H^{1}(0, L)]^* \big)} +  \| (h_1, h_2) \|_{L^2(0, T)} + \|
h_3 \|_{[H^{1/3}(0, T)]^*}\Big),
\end{multline}
where  $C_{T, L}$ denotes a positive constant  independent of $x$,  $y_0$, $f$, and $h_1, \, h_2, \,  h_3$. 
\end{proposition}

Here and in what follows, $[H^s]^*$ denotes the dual space of $H^s$ for $s > 0$  and it is equipped with the standard corresponding norm. 

\begin{remark} \rm In \Cref{pro-kdv1}, we implicitly admit that 
$$
 y \in C([0, L]; H^{1/3}(0, T)), \quad y_x \in C([0, L]; L^2(0, T)), \quad \mbox{ and } \quad   y_{xx} \in C([0, L]; [H^{1/3}(0, T)]^*).
$$ 
These facts are derived from the proof. Results and analysis which are related to \Cref{pro-kdv1}  will be discussed in \Cref{rem-RW}.  
\end{remark}

The next result is on the well-posedness of the linearized KdV system for which the Dirichlet condition and the second derivative in $x$ on the right are described. 

\begin{proposition}\label{pro-kdv2} Let $L>0$,  $T>0$, $(h_1, h_2, h_3) \in H^{1/3} (0, T) \times
H^{1/3}(0, T) \times [H^{1/3}(0, T)]^*$, \\
$ f \in L^1\big((0, T); L^2(0, L) \big)$, and $y_0 \in L^2(0, L)$. There exists a unique solution  $y \in X_T$ of the system
\begin{equation}\label{sys-y-LKdV2}\left\{
\begin{array}{cl}
y_t + y_x  + y_{xxx}  = f &  \mbox{ in } (0, T) \times  (0, L),
\\[6pt]
y(\cdot, 0) = h_1,  \;  y(\cdot, L) = h_2, \;  y_{xx}(\cdot, L)   = h_3& \mbox{ in } (0,
T),\\[6pt]
y(0, \cdot)  = y_0 &  \mbox{ in } (0, L).
\end{array}\right.
\end{equation}
Moreover, for $0 \le x \le L$, 
\begin{multline}\label{pro-kdv2-cl1}
\| y\|_{X_T} + \| y(\cdot, x)\|_{H^{1/3}(0, T)} +
\| y_x(\cdot, x)\|_{L^2(0, T)} + \| y_{xx}(\cdot, x)\|_{[H^{1/3}(0, T)]^*}
 \\[6pt] \le C_{T, L} \Big(\| y_0\|_{L^2(0, T)} + \| f\|_{L^1\big( (0, T);  L^2(0, L) \big)} +  \| (h_1, h_2) \|_{H^{1/3}(0, T)} + \|
h_3 \|_{[H^{1/3}(0, T)]^*}\Big), 
\end{multline}
and
\begin{multline} \label{pro-kdv2-cl2} 
\| y(\cdot, x)\|_{L^2(0, T)} + \| y_x(\cdot, x)\|_{[H^{1/3}(0, T)]^*} \\[6pt]
\le C_{T, L} \Big( \| y_0\|_{[H^{1}(0, L)]^*} + \| f\|_{L^1\big( (0, T);  [H^{1}(0, L)]^* \big)} +  \| (h_1, h_2) \|_{L^2(0, T)} + \|
h_3 \|_{[H^{2/3}(0, T)]^*}\Big), 
\end{multline}
for some positive constant $C_{T, L}$  independent of $x$, $y_0$, $f$, and $(h_1, h_2, h_3)$.
\end{proposition}

Here is the third main result of this section.

\begin{proposition}\label{pro-kdv3}
Let $L>0$, $T>0$, $(h_1, h_2, h_3) \in L^2(0, T)  
\times H^{1/3} (0, T) \times  L^2(0, T)$, 
$ f \in L^1\big((0, T); L^2(0, L) \big)$, and $y_0 \in L^2(0, L)$. There exists a unique solution  $y \in X_T$ of the system
\begin{equation}\label{sys-y-LKdV3}\left\{
\begin{array}{cl}
y_t + y_x  + y_{xxx}  = f &  \mbox{ in } (0, T) \times  (0, L),
\\[6pt]
y_x(\cdot, 0) = h_1, \; y(\cdot, L) = h_2, \;  y_{x}(\cdot, L)   = h_3& \mbox{ in } (0,
T),\\[6pt]
y(0, \cdot)  = y_0 &  \mbox{ in } (0, L).
\end{array}\right.
\end{equation}
Moreover, for $0 \le x \le L$, 
\begin{multline}\label{pro-kdv3-cl1}
\| y\|_{X_T} + \| y(\cdot, x)\|_{H^{1/3}(0, T)} +
\| y_x(\cdot, x)\|_{L^2(0, T)} + \| y_{xx}(\cdot, x)\|_{[H^{1/3}(0, T)]^*}
 \\[6pt] \le C_{T, L} \Big(\| y_0\|_{L^2(0, T)} + \| f\|_{L^1\big( (0, T);  L^2(0, L) \big)} +  \| h_2 \|_{H^{1/3}(0, T)} + \|(h_1, 
h_3) \|_{L^2(0, T)}\Big) 
\end{multline}
and
\begin{multline} \label{pro-kdv3-cl2} 
\| y(\cdot, x)\|_{L^2(0, T)} + \| y_x(\cdot, x)\|_{[H^{1/3}(0, T)]^*} \\[6pt]
\le C_{T, L} \Big( \| y_0\|_{[H^{1}(0, L)]^*} + \| f\|_{L^1\big( (0, T);  [H^{1}(0, L)]^* \big)} +  \| (h_1, h_2) \|_{L^2(0, T)} + \|
h_3 \|_{[H^{1/3}(0, T)]^*}\Big), 
\end{multline}
for some positive constant $C_{T, L}$  independent of $x$, $y_0$, $f$, and $(h_1, h_2, h_3)$. 
\end{proposition}

As a consequence of \Cref{pro-kdv1}, we obtain the following well-posedness result, which particularly yields the well-posedness for system \eqref{sys-KdV}. 

\begin{proposition}\label{pro-WP} Let $L > 0$ and $T>0$. There exists $\eps_0 > 0$ such that for  $ (h_1, h_2, h_3) \in H^{1/3}(0, T) \times H^{1/3} (0, T) \times  L^2(0, T)$, 
$ f \in L^1\big((0, T); L^2(0, L)\big)$, and $y_0 \in L^2(0, L)$ 
with 
$$
\| y_0 \|_{L^2(0, L)} + \| f\|_{L^1\big((0, T); L^2(0, L)\big)} + \| (h_1, h_2) \|_{H^{1/3}(0, T)} + \|
h_3 \|_{L^2(0, T)} \le \eps_0,
$$ 
there exists a unique solution $y \in X_T$ of the system 
\begin{equation}\label{pro-WP-S}\left\{
\begin{array}{cl}
y_t  + y_x  + y_{xxx}  + y y_x = f  &  \mbox{ in } (0, T) \times  (0, L), \\[6pt]
y(\cdot, 0) = h_1, \; y(\cdot, L) = h_2, \;  y_{x}(\cdot, L)   = h_3& \mbox{ in } (0,
T),\\[6pt] 
y(0, \cdot) = y_0 & \mbox{ in } (0, L).  
\end{array}\right.
\end{equation}
Moreover, we have 
\begin{multline}\label{pro-WP-statement}
\| y\|_{X_T} + \| y(\cdot, x)\|_{H^{1/3}(0, T)} +
\| y_x(\cdot, x)\|_{L^2(0, T)} + \| y_{xx}(\cdot, x)\|_{[H^{1/3}(0, T)]^*} \\[6pt]
 \le C_{T, L} \Big( \| y_0 \|_{L^2(0, L)} + \| f\|_{L^1\big((0, T); L^2(0, L) \big)} + \| (h_1, h_2) \|_{H^{1/3}(0, T)} + \|
h_3 \|_{L^2(0, T)}  \Big), 
\end{multline}
for some positive constant $C_{T, L}$  independent of $x$, $y_0$, $f$, and $(h_1, h_2, h_3)$. 
\end{proposition}

As a consequence of \Cref{pro-kdv3}, we obtain the following well-posedness result which will be used in the proof of Assertion $ii)$ of  \Cref{thm2}.
 
\begin{proposition}\label{pro-kdv3-NL}
Let $L>0$ and  $T>0$. There exists $\eps_0 > 0$ such that for  $(h_1, h_2, h_3) \in L^2(0, T)  
\times H^{1/3} (0, T) \times  L^2(0, T)$,  
$ f \in L^1\big((0, T); L^2(0, L) \big)$, and $y_0 \in L^2(0, L)$ 
with $$
\| y_0 \|_{L^2(0, L)} + \| f\|_{L^1\big((0, T); L^2(0, L)\big)} +\| h_2 \|_{H^{1/3}(0, T)} + \|(h_1, 
h_3) \|_{L^2(0, T)}  \le \eps_0,
$$ 
there exists a unique solution  $y \in X_T$ of the system
\begin{equation}\label{sys-y-LKdV3-NL}\left\{
\begin{array}{cl}
y_t + y_x  + y_{xxx}  + y y_x= f &  \mbox{ in } (0, T) \times  (0, L),
\\[6pt]
y_x(\cdot, 0) = h_1, \; y(\cdot, L) = h_2,  \;  y_x(\cdot, L) = h_3, & \mbox{ in } (0,
T),\\[6pt]
y(0, \cdot)  = y_0 &  \mbox{ in } (0, L).
\end{array}\right.
\end{equation}
Moreover, for $0 \le x \le L$, 
\begin{multline}\label{pro-kdv3-cl2-NL}
\| y\|_{X_T} + \| y(\cdot, x)\|_{H^{1/3}(0, T)} +
\| y_x(\cdot, x)\|_{L^2(0, T)} + \| y_{xx}(\cdot, x)\|_{[H^{1/3}(0, T)]^*} \\[6pt]
 \le C_{T, L} \Big( \| y_0 \|_{L^2(0, L)} + \| f\|_{L^1\big((0, T); L^2(0, L)\big)} + \| h_2 \|_{H^{1/3}(0, T)} + \|(h_1, 
h_3) \|_{L^2(0, T)}  \Big),  
\end{multline}
for some positive constant $C_{T, L}$  independent of $x$, $y_0$, $f$, and $u$. 
\end{proposition}

\begin{remark} \rm The meaning of the solutions considered in \Cref{pro-kdv1}, \Cref{pro-kdv2}, and \Cref{pro-kdv3} are given in \Cref{def-pro-kdv1}, \Cref{def-pro-kdv2}, \Cref{def-pro-kdv3}, respectively. These are motivated by the integration by parts arguments.  The meaning of the solutions in \Cref{pro-WP} and \Cref{pro-kdv3-NL} are understood in the same manner where the nonlinear term $y y_x$ plays as a part of the source $f$.    
\end{remark}

\Cref{pro-WP} and \Cref{pro-kdv3-NL} follows from  \Cref{pro-kdv1} and \Cref{pro-kdv3} by standard arguments. We just mention here that 
\begin{multline*}
\| y y_x \|_{L^1\big( (0, T); L^2 (0, L) \big)}  \le \int_0^T \sup_{x \in [0, L]} |y(t, x)| \left(\int_0^L |y_x(t, x)|^2\right)^{\frac{1}{2}} \\[6pt]
\le   C\int_0^T \left(\int_0^L |y(t, x)|^2 + |y_x(t, x)|^2\right)^{\frac{1}{2}}\left(\int_0^L |y_x(t, x)|^2\right)^{\frac{1}{2}} \le C \|y \|_{X_T}^2. 
\end{multline*}
The details of the proof are left to the reader. A sharper estimate can be derived from \Cref{lem-interpolationL1L2}

\medskip 
The rest of this section is organized as follows. The proof of 
\Cref{pro-kdv1}, \Cref{pro-kdv2}, and \Cref{pro-kdv3} are given in \Cref{sect-pro-kdv1}, \Cref{sect-pro-kdv2}, and \Cref{sect-pro-kdv3}, respectively.

\subsection{Proof of \Cref{pro-kdv1}} \label{sect-pro-kdv1}

We first give the meaning of the solutions considered in \Cref{pro-kdv1}. 

\begin{definition}\label{def-pro-kdv1}  Let $L>0$,  $T>0$, $(h_1, h_2, h_3) \in H^{1/3} (0, T) \times
H^{1/3}(0, T) \times L^2 (0, T)$,  $ f \in L^1\big((0, T); L^2(0, L)\big)$, and $y_0 \in L^2(0, L)$. A  solution $y \in X_T$  of the system \eqref{sys-y-LKdV} is a function $y \in X_T$ such that 
\begin{multline}
\int_0^T \int_0^L f(t, x) \varphi (t, x) \, dx \, dt + \int_0^L y_0 (x) \varphi(0, x) \, dx   + \int_0^T  h_3 (t) \varphi_x (t, L) \, dt \\[6pt]
 - \int_0^T  h_2(t) \varphi_{xx}(t, L) \, dt  + \int_0^T  h_1(t) \varphi_{xx}(t, 0) \, dt  = - \int_0^T \int_0^L y (\varphi_t + \varphi_x + \varphi_{xxx}) \, dx \, dt  
\end{multline}
for all $\varphi \in C^3([0, T] \times [0, L])$ with $\varphi(T, \cdot) = 0$ and  $\varphi(\cdot, 0) = \varphi(\cdot, L) = \varphi_x(\cdot, 0) =0$. 
\end{definition}

\begin{remark} \rm
One can check that if $y$ is a smooth solution of \eqref{sys-y-LKdV} then $y$ is a solution of \eqref{sys-y-LKdV} in the sense of \Cref{def-pro-kdv1} by standard integration by parts arguments. In fact, \Cref{def-pro-kdv1} is motivated by such arguments.    
\end{remark}

The proof of \Cref{pro-kdv1} is divided into two parts given in the following two subsections. The first part is on the existence and the estimates for a constructed solution. The second part is on the uniqueness.

\subsubsection{Existence and estimates}\label{sect-WP-estimate}

The proof of \eqref{pro-kdv1-cl1} for a solution constructed below is given in two steps.

\medskip 
\noindent{\it Step 1:} We first consider the case where $y_0 = 0$ and $f = 0$. 
By the linearity, it suffices to
consider the three cases $(h_1, h_2, h_3) = (0, 0, h_3)$, $(h_1, h_2, h_3)  = (h_1, 0,
0)$, and $(h_1, h_2, h_3)  = (0, h_2, 0)$ separately. In what follows, we extend $h_1, h_2, h_3$ by $0$ for $t > T$ and still denote $y$ the corresponding solution and these extensions by $h_1, h_2, h_3$.

In what follows in this proof, for an appropriate function $v$ defined
on $\mR \times (0, L)$,  we denote by
$\hat v$ its Fourier transform with respect to $t$, i.e., for $z \in \mC$,
\begin{equation*}
\hat v(z, x) = \frac{1}{\sqrt{2 \pi} }\int_0^{+\infty} v(t, x) e^{- i z t} \diff t.
\end{equation*}
Extend $y$ and $h_1, h_2, h_3$ by $0$ for $t < 0$ and still denote these extensions by $y$,  and $h_1, h_2, h_3$. Then 
\be
y_t + y_x + y_{xxx} = 0 \mbox{ in } \mR \times (0, L). 
\ee
Taking the Fourier transform with respect to $t$, we obtain, for $z \in \mR$,  
\be
i z \hat y + \hat y_x + \hat y_{xxx} = 0 \mbox{ in } (0, L). 
\ee

For $z \in \mC$,  let $\lambda_j = \lambda_j(z)$ with $j=1, 2, 3$ be the three solutions of the equation $\lambda^3 + \lambda + i z = 0$. Set 
$$
Q (z) : = 
 \begin{pmatrix}
  1&1&1 \\
  e^{\lambda_1L}&e^{\lambda_2L}&e^{\lambda_3L}\\
  \lambda_1 e^{\lambda_1  L}& \lambda_2 e^{\lambda_2  L}& \lambda_3 e^{\lambda_3  L}
 \end{pmatrix},
$$
and 
\begin{equation*}
\Xi = \Xi(z) :=  \det \begin{pmatrix}1&1&1\\ \lambda_1&\lambda_2&\lambda_3\\
\lambda_1^2&\lambda_2^2&\lambda_3^2\end{pmatrix}, 
\end{equation*}
with the convention $\lambda_{j+3} = \lambda_{j}$ for $j \ge 1$. It is useful to note that 
\be\label{def-H2}
H(z) : = \det Q (z) \Xi(z) \mbox{ in } \mC. 
\ee
is an analytic function, see, e.g., \cite[Lemma A1]{CKN20}. Moreover,  $H$ only has a finite number of zeros on the real line since $H (z) \neq 0$ for $z \in \mR$ with large $|z|$.

We first consider the case $(h_1, h_2, h_3)  = (0, 0, h_3)$. Taking into account the equation of $\hy$, we search for the solution of the form
\begin{equation}\label{lem-kdv2-p0}
\hy(z, \cdot) = \sum_{j=1}^3 a_j  e^{\lambda_j x},
\end{equation}
where  $a_j = a_j(z)$ for $j = 1, 2, 3$. Taking the boundary condition, we then have 
\begin{equation*}
\begin{array}{c}
\sum_{j=1}^3 a_j  =0, \\[6pt]
\sum_{j=1}^3 a_j e^{\lambda_j L}  =0, \\[6pt]
\sum_{j=1}^3 a_j \lambda_j e^{\lambda_j L}  = \hh_3, \\[6pt]
\end{array}
\end{equation*}
This implies, with the convention $\lambda_{j+3} = \lambda_j$,  
\begin{equation}\label{lem-kdv1-p1}
a_j =  \frac{e^{\lambda_{j+2}L} - e^{\lambda_{j+1} L}}{\det Q} \hh_3 \mbox{ for } j=1, 2, 3. 
\end{equation}

To estimate the solution, we proceed as follows.  By \Cref{lem-zeros} in the appendix, there exists $g_{3} \in C^\infty(\mR)$ with $\supp g_{3} \subset [T, 3 T]$ such that if
$z$ is a real solution of the equation $H(z) = 0$  of order $m$ then $z$ is also a real solution of order $m$ of $\hat h_3(z) -
\hat g_{3}(z)$, and, for $k \ge 1$,  
\begin{equation}\label{lem-kdv1-g3}
\|g_3 \|_{H^k(\mR)} \le C_{k} \| h_3\|_{H^{-1/3}(\mR)}.
\end{equation}

We now establish \eqref{pro-kdv1-cl1}.  Let $y_3$ be the solution of \eqref{sys-y-LKdV} where $(h_1, h_2, h_3)$ are  $(0, 0, h_3 - g_3)$, and  $f=0$, and $y_0 = 0$.  We have, by applying \eqref{lem-kdv1-p1} to $y_3$, 
\begin{equation}\label{lem-kdv1-y}
\hy_3(z, x) = \frac{\hh_3(z) - \hg_3(z) }{\det Q(z)}  \sum_{j=1}^3 \big(e^{\lambda_{j+2} L } -
e^{\lambda_{j+1} L } \big) e^{\lambda_j x} \mbox{ for a.e. } x \in (0, L).
\end{equation}
We derive that, by the choice of $g_3$, for $z \in \mR$ and $|z| \le \gamma$,
\begin{equation}\label{lem-kdv-1-p1}
\left| \frac{\hh_3(z) - \hg_3(z)}{H(z)}  \right| \left| \Xi(z) \sum_{j=1}^3 \big(e^{\lambda_{j+2} L } -
e^{\lambda_{j+1} L } \big) e^{\lambda_j x} \right| \le C_{T, \gamma} \|
h_3 - g_3\|_{H^{-1/3}(\mR)},
\end{equation}
and, by \Cref{lem-lambda} below on the behaviors of $\lambda_j(z)$ for large $z$ (the case large negative $z$ can be obtained by considering the conjugate), we obtain,  for $z \in \mR$, $|z| \ge \gamma$ with sufficiently large
$\gamma$,
\begin{equation}\label{lem-kdv-1-p2}
\left| \frac{1}{\det Q}  \sum_{j=1}^3 \big(e^{\lambda_{j+2} L } - e^{\lambda_{j+1} L }
\big) e^{\lambda_j x} \right| \le \frac{C}{ (1 + |z|)^{1/3}}.
\end{equation}
Combining \eqref{lem-kdv-1-p1} and \eqref{lem-kdv-1-p2} yields
\[
\|y_3(\cdot, x)\|_{H^{1/3}(\mR)}  \le C_{T, L} \| h_3 - g_3\|_{L^2(\mR)}
\]
and
\[
\|y_3(\cdot, x)\|_{L^2(\mR)}  \le C_{T, L} \| h_3 - g_3\|_{H^{-1/3}(\mR)}. 
\]
Similarly, we have
\[
\|y_{3, x} (\cdot, x)\|_{L^2(\mR)} + 
\|y_{3, xx} (\cdot, x)\|_{H^{-1/3}(\mR)}  \le C_{T, L} \| h_3 - g_3\|_{L^2(\mR)} 
\]
and 
\[
\|y_{3, x} (\cdot, x)\|_{H^{-2/3}(\mR)}  \le C_{T, L} \| h_3 - g_3\|_{H^{-1/3}(\mR)}. 
\]
The estimates for $\|y(\cdot, x)\|_{H^{1/3}(0, T)}$, $\|y_{x} (\cdot, x)\|_{L^2(0, T)}$, and  $\|y_{xx} (\cdot, x)\|_{[H^{1/3}(0, T)]^*}$  and the estimate for  $\|y(\cdot, x)\|_{L^2(0, T)}$, $\|y_{x} (\cdot, x)\|_{[H^{1/3}(0, T)]^*}$ follow by noting that $y = y_3$ in $(0, T) \times (0, L)$ \footnote{More precisely, one can take $y= y_3$ in $(0, T) \times (0, L)$.}. 

We also have, by integration by parts,  for $0 \le \tau_1 < \tau_2 \le T$, 
\begin{multline}\label{pro-kdv-XT}
\frac{1}{2}\int_{0}^L |y(\tau_2, x)|^2 \, dx - \frac{1}{2}\int_{0}^L |y(\tau_1, x)|^2 \, dx  + \frac{1}{2}\int_{\tau_1}^{\tau_2} \Big(  |y(t, L)|^2 - |y(t, 0)|^2 \Big) \, dt \\[6pt]
+ \int_{\tau_1}^{\tau_2} \Big( y_{xx} (t, L) y(t, L) - y_{xx} (t, 0) y(t, 0) \Big) \, dt -  \frac{1}{2}\int_{\tau_1}^{\tau_2} \Big( |y_x(t, L)|^2 - |y_x(t, 0)|^2 \Big) \, dt  =0. 
\end{multline}
Using the estimates for $\|y(\cdot, x)\|_{H^{1/3}(0, T)}$, $\|y_{x} (\cdot, x)\|_{L^2(0, T)}$, and  $\|y_{xx} (\cdot, x)\|_{[H^{1/3}(0, T)]^*}$, we obtain the one for $\| y \|_{X_T}$. 
 
\medskip 

The proof in the case $(h_1, h_2, h_3) = (h_1, 0, 0)$ or in the case $(h_1, h_2, h_3) =
(0, h_2, 0)$ is similar after noting \Cref{lemH1/3} in the appendix. We mention here that  the solution corresponding to the
triple $(h_1, 0, 0)$ is given by
\[
\hy(z, x) = \frac{\hh_1(z) }{\det Q}  \sum_{j=1}^3 (\lambda_{j+2} e^{-\lambda_j L } - \lambda_{j+1}e^{-\lambda_j L })
e^{\lambda_j x}
\mbox{ for a.e. } x \in (0, L),
\]
and the solution corresponding to the triple $(0, h_2, 0)$ is given by 
\[
\hy(z, x) = \frac{\hh_2(z) }{\det Q}  \sum_{j=1}^3 (\lambda_{j+1} e^{\lambda_{j+1} L } -
\lambda_{j+2} e^{\lambda_{j+2} L }) e^{\lambda_j x}
\mbox{ for a.e. } x \in (0, L).
\]
The details are left to the reader.

\begin{remark} \label{rem-pro-kdv1-Step1} \rm If $h_1, h_2, h_3 \in C^\infty_c(\mR)$, the constructed solution is also smooth. 
\end{remark}

\medskip 
\noindent{\it Step 2:} We now deal with the general case.  The starting point of the proof is  a connection between  the linearized  KdV equation and the linear 
KdV-Burgers equation.  Set $v(t, x) = e^{-2 t + x} y (t, x)$, which is equivalent to
$y(t, x)  = e^{2t - x} v(t, x)$. One can check that  if $y$ satisfies the equation
\[
y_t  + y_x + y_{xxx}  = f \mbox{ in } \mR_+ \times  (0, L),
\]
if and only if 
\[
v_t  + 4 v_x   + v_{xxx} - 3 v_{xx}  = f  e^{-2 t + x}
\mbox{ in } \mR_+ \times  (0, L).
\]

Set, in $\mR_+ \times (0, L)$,
\begin{equation}\label{lem-kdv3-def-hg}
\psi(t, x) =  \psi(t) : =  \frac{1}{L} \int_0^L f(t, \xi) e^{-2 t + \xi} \diff  \xi \quad
\mbox{ and } \quad
g(t, x) : = f(t, x) e^{-2 t + x} - \psi(t, x).
\end{equation}
Then
\[
\int_0^L g(t, x) \diff x = 0 \mbox{ for } t > 0. 
\]
Let $y_1 \in  C\big([0, + \infty); L^2(0, L) \big) \cap L^2_{\loc}\big([0, + \infty);
H^1(0, L) \big)$ be the unique solution periodic in space  of the system
\begin{equation}\label{pro-kdv1-sys-y1}
y_{1, t}  + 4 y_{1, x}  + y_{1, xxx}  - 3 y_{1, xx} = g 
\mbox{ in } (0, +\infty) \times (0, L),
\end{equation}
and
\begin{equation}
y_1(0, \cdot)  = y_0  e^x  \mbox{ in } (0, L).
\end{equation}
Set 
$$
\alpha = \frac{1}{L} \int_0^L y_1(0, x) \, dx. 
$$

By \Cref{pro-kdv-B}, we have, for $x \in [0, L]$, 
\begin{multline}\label{lem-kdv1-y1}
\| y_{1}(\cdot, x) - \alpha\|_{H^{1/3}(\mR_+)} + \| y_{1, x} (\cdot, x) \|_{L^2(\mR_+)} +  \| y_{1, xx} (\cdot, x) \|_{[H^{1/3}(\mR_+)]^*} \\[6pt]
\le C \Big( \|f \|_{L^1(\mR_+; L^2(0, L))} + \| y_0 \|_{L^2(0, L)} \Big)
\end{multline}
and 
\begin{multline}\label{lem-kdv1-y1-2}
\| y_{1}(\cdot, x) - \alpha\|_{L^2(\mR_+)} + \| y_{1, x} (\cdot, x) \|_{[H^{1/3}(\mR_+)]^*} 
\le C \Big(\| y_0 \|_{[H^{1}(0, L)]^*}  + \|f \|_{L^1(\mR_+; [H^{1}(0, L)]^*)} \Big).
\end{multline}

Let $y_2 \in X_{T}$ be the unique solution of 
\begin{equation}\label{pro-kdv1-sys-y2}\left\{
\begin{array}{cl}
y_{2, t}  + y_{2, x}  + y_{2, xxx}  = 0 &  \mbox{
in } (0, T) \times (0, L), \\[6pt]
y_2(t, 0) = h_1(t)  - e^{2t} \Big( y_1(t, 0) + \int_0^t \psi(s) \, ds  \Big)   & \mbox{ in } (0, T), \\[6pt]
y_2(t, L) = h_2 (t) -  e^{2t - L} \Big( y_1(t, L) + \int_0^t \psi(s) \, ds  \Big)  & \mbox{ in } (0, T),
\\[6pt]
y_{2, x}(t, L) = h_3(t) - \big(e^{2t - \cdot} (y_1(t, \cdot) + \int_0^t \psi(s) \, ds) \big)_x (t, L)
& \mbox{ in } (0, T),\\[6pt]
y_2(t = 0, \cdot)  = 0 &  \mbox{ in } (0, L).
\end{array}\right.
\end{equation}
Applying  the results of Step 1 to $y_2$,  and using \eqref{lem-kdv1-y1} and \eqref{lem-kdv1-y1}, we derive that, for $x \in [0, L]$,  
\begin{multline*}
\| y_2\|_{X_T}  +  \| y_{2}(\cdot, x) \|_{H^{1/3}(0, T)}  + \| y_{2, x}(\cdot, x) \|_{L^2(0, T)} + \| y_{2, xx}(\cdot, x) \|_{[H^{1/3}(0, T)]^*}  \\[6pt]
\le C_T  \Big( \| (h_1, h_2) \|_{H^{1/3}(\mR_+)}
+ \|h_3 \|_{L^2(\mR_+)} + \| f\|_{L^1([0, T]; L^2 (0, L))} \Big) 
\end{multline*}
and
\begin{multline*}
\| y_{2}(\cdot, x) \|_{L^2(0, T)}  + \| y_{2, x}(\cdot, x) \|_{[H^{1/3}(0, T)]^*} \\[6pt]
\le C_T  \Big( \| (h_1, h_2) \|_{L^2(\mR_+)}
+ \|h_3 \|_{[H^{1/3}(\mR_+)]^*} + \| f\|_{L^1([0 T]; [H^{1} (0, L)]^*)} \Big). 
\end{multline*}
The conclusion follows by noting that $y = e^{2t -x} \Big(y_1 + \int_0^t \psi(s) \, ds \Big) + y_2$ in $(0, T) \times (0, L)$.

\medskip 
The proof of the existence of one solution and its  estimates is complete if one can show that using this process, one can construct
a solution which verifies \Cref{def-pro-kdv1}. To this end, one first notes that if $h_1$, $h_2$, $h_3$, $y_0$, $g$ are smooth and $\supp h_1, \supp h_2, \supp h_3 \Subset (0, T)$, $\supp y_0 \Subset (0, L)$ and $\supp g \Subset (0, T) \times (0, L)$ then one can construct a solution with the desired estimates. Indeed,  the solution $y_1$ given by \eqref{pro-kdv1-sys-y1} is smooth and has the property  that $y_1(t, 0)$ and $y_1(t, L)$ are 0 for $t$ close to $0$ (see \Cref{rem-pro-kdv-B1} and \Cref{rem-pro-kdv-B2}). One can then construct a smooth solution $y_2$ of \eqref{pro-kdv1-sys-y2} in the interval $(0, 2T)$ where the the expression of $y_2(t, 0)$, $y_2(t, L)$, and $y_{2, x}(t, L)$ are replaced by the expressions given in \eqref{pro-kdv1-sys-y2} multiplied by a cutoff function $\chi$, which is 1 in $(0, T)$ and is $0$ for $t > 3T/2$. The solution $y_2$ is also smooth (see \Cref{rem-pro-kdv1-Step1}). 
Then $e^{2t -x} \Big(y_1 + \int_0^t \psi(s) \, ds \Big) + y_2$ is still a solution in the time interval $(0, T)$ with the desired estimates.  Using this observation, by standard approximation arguments, one can construct a solution in the sense of \Cref{def-pro-kdv1} with the required estimates.  \qed

\medskip

In the proof of \Cref{pro-kdv1},  we used the following elementary result.  
\begin{lemma}\label{lem-lambda}
For  $p \in \mC$ and $z$ in a sufficiently small conic neighborhood of $\R_+$, let $\lambda_j$ with $j=1, \, 2, \, 3$ be the three solutions of the equation $\lambda^3 + \lambda + iz =0 $.  Consider the convention $\Re(\lambda_1) < \Re(\lambda_2) < \Re(\lambda_3)$ and similarly
for $\tlambda_j$. We have,  in the
limit $|z|\to \infty$,
\begin{equation}
\lambda_j = \mu_jz^{1/3} - \frac1{3\mu_j}z^{-1/3} + O(z^{-2/3}) \quad \text{ with } \mu_j
= e^{-i\pi/6-2ji\pi/3},
\end{equation}
Here $z^{1/3}$ denotes the cube root of $z$ with the real part positive.  
\end{lemma}

 Here and in what follows, for $s \in \mR$, $O(z^s)$ denotes a quantity bounded by $C z^s$ for large positive $z$. Similar convention is used for $O(|z|^s)$ for $z \in \mC$.

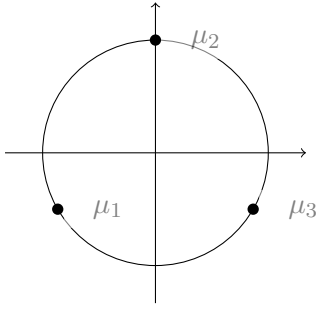
\begin{figure}[htp]
 \begin{minipage}[c]{0.4\textwidth}
  {\centering
  \begin{tikzpicture}[scale=0.5]
  \draw[->] (-4,0) -- (4,0);
  \draw[->] (0,-4) -- (0,4);
  
  \draw (0,0) circle[radius=3];
  
  \foreach \j in {1,2,3} {
   \fill (-30-120*\j:3) circle[fill, radius=0.15] 
     ++(0.15,0) node[right, fill=white, fill opacity=0.5, text opacity = 1]{$\mu_\j$};
         }
\end{tikzpicture}}
 \end{minipage}\hfill%
 \begin{minipage}[c]{0.6\textwidth}
 
\caption{The roots $\lambda_j$ of $\lambda^3 + \lambda + i z = 0$ satisfy, when
$z>0$ is large, $\lambda_j \sim \mu_j z^{1/3}$ where $\mu_j^3 = -i$.}
\label{fig:mu}
\end{minipage}
\end{figure}

\medskip

The same proof as in Step 1 of the proof of \Cref{pro-kdv1} gives the following result. 
\begin{lemma}\label{lem-kdv1} Let $T>0$ and $(h_1, h_2, h_3) \in H^{1/3} (0, T) \times
H^{1/3}(0, T) \times L^2 (0, T)$. Then the unique solution $y \in X_T$  of the system \eqref{sys-y-LKdV} with $y_0= 0 $ and $f = 0$ satisfy, for $x \in [0, L]$,
\be\label{lem-kdv1-cl2}
\| y(\cdot, x)\|_{[H^{2/3}(0, T)]^*}  \le C_{T, L} \Big(  \| (h_1, h_2) \|_{[H^{2/3}(0, T)]^*} + \|
h_3 \|_{[H^{1}(0, T)]^*}\Big),   
\ee
where  $C_{T, L}$ denotes a positive constant  independent of $x$,  $y_0$,  and $h_1, \, h_2, \,  h_3$. 
\end{lemma}

\subsubsection{Uniqueness} \label{sect-WP-uniqueness} In this section, we establish the uniqueness of the solutions given in \Cref{pro-kdv1}. The uniqueness is somehow known for a close definition of the solution which requires $\varphi$ less regular. The definition considered here is somehow more advantageous, see \Cref{rem-uniqueness}. We shall highlight the analysis in such a way that one can extend it to other settings considered in this paper; thus we do not need to repeat the arguments as much as possible.  
Let $y \in X_T$ be a solution with the zero data, i.e., $f=0$, $y_0 = 0$, and $(h_1, h_2, h_3) = (0, 0, 0)$. Fix $\psi \in C^\infty_c \big((0, T) \times (0, L) \big)$ (arbitrarily). Let $\ty \in X_T$ be a solution of the backward system  
\begin{equation*}\left\{
\begin{array}{cl}
\ty_{t} + \ty_{x}  + y_{xxx}  = \psi &  \mbox{ in } (0, T) \times  (0, L),
\\[6pt]
\ty(\cdot, 0) = 0,  \;  \ty_{x}(\cdot, 0) = 0, \;  \ty(\cdot, L)   = 0 & \mbox{ in } (0,
T), \\[6pt]
\ty(T, \cdot)  = 0 & \mbox{ in } (0, L).
\end{array}\right.
\end{equation*}
Using the construction given in \Cref{sect-WP-estimate}, one can assume that $\ty$ is smooth (see the last part of Step 2 of \Cref{sect-WP-estimate}). Taking $\varphi = \ty$ in the definition, we derive that 
$$
\int_{0}^T \int_0^L \psi (t, x) y(t, x) \, dt \, dx =0. 
$$
Since $\psi \in C^\infty_c \big((0, T) \times (0, L) \big)$ is arbitrary, we deduce that 
$$
y = 0 \mbox{ in } (0, T) \times (0, L). 
$$
The uniqueness is proved. \qed 

\begin{remark} \rm \label{rem-uniqueness}The proof of the uniqueness also gives the uniqueness of the solutions in  $L^1 \big((0, T) \times (0, L) \big)$, i.e., one requires $y \in L^1 \big((0, T) \times (0, L) \big)$ instead of $y \in X_T$ in \Cref{def-pro-kdv1}. 
\end{remark}

\begin{remark} \rm \label{rem-RW}
We end this section with some comments on \Cref{pro-kdv1} and its proof. 

\smallskip 
 
\quad i) The well-posedness of \eqref{sys-y-LKdV} in $X_T$ is  proved in \cite[Theorem 2.10 and Proposition 2.16]{Bona03} for $L=1$ when  $(0, T) = \mR_+$ and  the estimate  for $\| y_x(\cdot, x)\|_{L^2(0, T)}$ in this case is a  consequence of their results. For $L=1$, their results imply that, when  $(h_1, h_2, h_3) \equiv (0, 0, 0)$,  
\be \label{Bona03}
\| y_x(\cdot, x)\|_{L^2(\mR_+)} \le C \| y_0 \|_{L^2(0, 1)} \mbox{ for } x  \in [0, 1]. 
\ee
Note that $L=1$ is smaller than the smallest critical length in $\cN_N$ which is $2 \pi$. The estimate as \eqref{Bona03}  cannot hold for arbitrary $L$. Such an estimate is not valid for any critical length by considering a non-zero initial datum in  $\cM_N$. 

\smallskip 
\quad ii) Similar estimates for  $\| y(\cdot, x)\|_{H^{1/3}(0, T)}$ and $\| y_x(\cdot, x)\|_{L^2(0, T)}$ as in \eqref{pro-kdv1-cl1}  in the real line space setting can be found in \cite{CK02, Holmer06}. 
To our knowledge, variants of the estimate for  $\| y_{xx}(\cdot, x)\|_{[H^{1/3}(0, T)]^*}$ in \eqref{pro-kdv1-cl1} are not known even in the real line space setting.    Our proof of \Cref{pro-kdv1} is in the spirit of \cite{CKN20}, which involves the Fourier transform with respect to time of the solution,  as in \cite{Bona03}, and a connection between the linearized KdV and the linear KdV-Burger equations. However, in the study of the linear KdV-Burger equation with periodic boundary conditions, the singularity of the kernel is appropriately {\it compensated} (see the proof of \Cref{pro-kdv-B}), which is one of the novelties of the analysis. The proof given here is self-contained and is different from the ones in \cite{CK02, Holmer06} for the whole line setting, which are based on the Riemann-Liouville fractional integrals and the theory of Airy functions. 
\end{remark}

\subsection{Proof of \Cref{pro-kdv2}} \label{sect-pro-kdv2} We first give the definition of the solutions considered  in \Cref{pro-kdv2}.

\begin{definition}\label{def-pro-kdv2}  Let $L>0$,  $T>0$, $(h_1, h_2, h_3) \in H^{1/3} (0, T) \times
H^{1/3}(0, T) \times [H^{1/3} (0, T)]^*$,  $ f \in L^1\big((0, T); L^2(0, L)\big)$, and $y_0 \in L^2(0, L)$. A  solution $y \in X_T$  of the system \eqref{sys-y-LKdV2} is a function $y \in X_T$ such that 
\begin{multline}
\int_0^T \int_0^L f(t, x) \varphi (t, x) + \int_0^L y_0 (x) \varphi(0, x) \, dx   - \int_0^T  h_3 (t) \varphi (t, L) \, dt \\[6pt]
 - \int_0^T  h_2 (t) \varphi_{xx}(t, L) \, dt +  \int_0^T  h_1 (t) \varphi_{xx}(t, 0) \, dt = - \int_0^T \int_0^L y (\varphi_t + \varphi_x + \varphi_{xxx}) \, dx \, dt   
\end{multline}
for all $\varphi \in C^3([0, T] \times [0, L])$ with $\varphi(T, \cdot) = 0$, $\varphi(\cdot, 0) = \varphi_x(\cdot, 0) = \varphi_x(\cdot, L) = 0$. 
\end{definition}

The proof of \Cref{pro-kdv1}
is similar to the one of \Cref{pro-kdv1}. The details are left to the reader. We just present here the formula of the solution of the system when  $y_0 = 0$,   $h_1 =0$,  $h_2 =0$, and $f=0$ and discuss the uniqueness. 

Extend $h_3$ by $0$ for $t \not \in [0, T]$ and still denote this extension by $h_3$. Then $h_3 \in H^{-1/3}(\mR)$. Denote $y$ be the corresponding solution for $t \ge 0$ and extend $y$ by 0 for $t < 0$. Still denote this extension by $y$. One then has 
\begin{equation}\left\{
\begin{array}{cl}
y_t + y_x  + y_{xxx}  = 0 &  \mbox{ in } \mR \times  (0, L),
\\[6pt]
y(\cdot, 0) = 0,  \;  y(\cdot, L) = 0, \;  y_{xx}(\cdot , L)   = h_3 & \mbox{ in } \mR. 
\end{array}\right.
\end{equation}

Taking the Fourier transform with respect to $t$, from the system of $y$, we have
\begin{equation}\label{sys-hy}
\left\{\begin{array}{cl}
i z \hy(z, x) + \hy_x (z, x) + \hy_{xxx} (z, x) = 0 &  \mbox{ in } \mR \times  (0, L),
\\[6pt]
\hy(z, 0) = \hy(z, L) = 0 & \mbox{ in }  \mR, \\[6pt]
\hy_{xx}(z, L) = \hh_3(z) & \mbox{ in } \mR.
\end{array}\right.
\end{equation}
Taking into account the equation of $\hy$, we search for the solution of the form
\begin{equation}
\hy(z, \cdot) = \sum_{j=1}^3 a_j  e^{\lambda_j x},
\end{equation}
where $\lambda_j = \lambda_j(z) $ with $j=1, 2, 3$ are the three solutions of the equation $\lambda^3 + \lambda + i z = 0$, and  $a_j = a_j(z)$ for $j = 1, 2, 3$.  We then have 
\begin{equation*}
\begin{array}{c}
\sum_{j=1}^3 a_j  =0, \\[6pt]
\sum_{j=1}^3 a_j e^{\lambda_j L}  =0, \\[6pt]
\sum_{j=1}^3 a_j \lambda_j^2 e^{\lambda_j L}  = \hh_3, \\[6pt]
\end{array}
\end{equation*}
This implies, with the convention $\lambda_{j+3} = \lambda_j$,  
\begin{equation}\label{lem-kdv2-p1}
a_j =  \frac{e^{\lambda_{j+2}L} - e^{\lambda_{j+1} L}}{\sum_{k=1}^3 e^{-\lambda_k L} (\lambda_{k+2}^2 - \lambda_{k+1}^2)} \hh_3 \mbox{ for } j=1, 2, 3. 
\end{equation}

We now deal with the uniqueness. Let $y \in X_T$ be a solution with the zero data, i.e., $f=0$, $y_0 = 0$, and $(h_1, h_2, h_3) = (0, 0, 0)$. Let $\psi \in C^\infty_c \big((0, T) \times (0, L) \big)$. Let $\ty \in X_T$ be a solution of the backward system  
\begin{equation*}\left\{
\begin{array}{cl}
\ty_{t} + \ty_{x}  + \ty_{xxx}  = \psi &  \mbox{ in } (0, T) \times  (0, L),
\\[6pt]
\ty(\cdot, 0) = 0,  \;  \ty_{x}(\cdot, 0) = 0, \;  \ty_x(\cdot, L)   = 0 & \mbox{ in } (0,
T), \\[6pt]
\ty(T, \cdot)  = 0 & \mbox{ in } (0, L).
\end{array}\right.
\end{equation*}
One can show that $\ty$ is smooth (see the construction of the solutions given in \Cref{pro-kdv3} and the last part of Step 2 of \Cref{sect-WP-estimate}; in fact, the arguments are simpler since the data on the boundary are 0 and the initial datum is 0).  Taking $\varphi = \ty$ in \Cref{def-pro-kdv1}, we derive that 
$$
\int_{0}^T \int_0^L \psi (t, x) y(t, x) \, dt \, dx =0. 
$$
Since $\psi \in C^\infty_c \big((0, T) \times (0, L) \big)$ is arbitrary, we deduce that 
$$
y = 0 \mbox{ in } (0, T) \times (0, L). 
$$
The uniqueness is proved.  \qed

\begin{remark} \rm The proof of the uniqueness also gives the uniqueness of the solutions in  $L^1 \big((0, T) \times (0, L) \big)$, i.e., one requires $y \in L^1 \big((0, T) \times (0, L) \big)$ instead of $y \in X_T$ in \Cref{def-pro-kdv2}. 
\end{remark}

\subsection{Proof of \Cref{pro-kdv3}} \label{sect-pro-kdv3}

We first introduce the notion of the solutions considered in \Cref{pro-kdv3}. \begin{definition}\label{def-pro-kdv3}  Let $L>0$,  $T>0$, $(h_1, h_2, h_3) \in L^2(0, T) \times H^{1/3} (0, T) \times
 L^2 (0, T)$, \\ $ f \in L^1\big((0, T); L^2(0, L)\big)$, and $y_0 \in L^2(0, L)$. A  solution $y \in X_T$  of the system \eqref{sys-y-LKdV3} is a function $y \in X_T$ such that 
\begin{multline}
\int_0^T \int_0^L f(t, x) \varphi (t, x) + \int_0^L y_0 (x) \varphi(0, x) \, dx   + \int_0^T  h_3 (t) \varphi_x (t, L) \, dt \\[6pt]
 - \int_0^T  h_1(t) \varphi_{x}(t, 0) \, dt  - \int_0^T  h_2(t) \varphi_{xx}(t, L) \, dt  = - \int_0^T \int_0^L y (\varphi_t + \varphi_x + \varphi_{xxx}) \, dx \, dt  
\end{multline}
for all $\varphi \in C^3([0, T] \times [0, L])$ with $\varphi(T, \cdot) = 0$, $\varphi(\cdot, 0) = \varphi(\cdot, L) = \varphi_{xx}(\cdot, 0) =0$. 
\end{definition}

The proof of \Cref{pro-kdv3} is similar to the one of \Cref{pro-kdv1}. We only give here the formula of the solutions in the case $f=0$, $y_0 = 0$, and  $(h_1, h_2, h_3) = (0, 0, h_3)$. Taking the Fourier transform with respect to $t$, from the system of $y$, we have
\begin{equation}\label{sys-hy-kdv3}
\left\{\begin{array}{cl}
i z \hy(z, x) + \hy_x (z, x) + \hy_{xxx} (z, x) = 0 &  \mbox{ in } \mR \times  (0, L),
\\[6pt]
\hy_x(z, 0) = \hy(z, L) = 0 & \mbox{ in }  \mR, \\[6pt]
\hy_{x}(z, L) = \hh_3(z) & \mbox{ in } \mR.
\end{array}\right.
\end{equation}
Taking into account the equation of $\hy$, we search for the solution of the form
\begin{equation}\label{pro-kdv3-p0}
\hy(z, \cdot) = \sum_{j=1}^3 a_j  e^{\lambda_j x},
\end{equation}
where $\lambda_j = \lambda_j(z) $ with $j=1, 2, 3$ are the three solutions of the equation $\lambda^3 + \lambda + i z = 0$, and  $a_j = a_j(z)$ for $j = 1, 2, 3$.  We then have 
\begin{equation*}
\begin{array}{c}
\sum_{j=1}^3 \lambda_j a_j  =0, \\[6pt]
\sum_{j=1}^3 a_j e^{\lambda_j L}  =0, \\[6pt]
\sum_{j=1}^3 a_j \lambda_j e^{\lambda_j L}  = \hh_3, \\[6pt]
\end{array}
\end{equation*}
This implies, with the convention $\lambda_{j+3} = \lambda_j$,  
\begin{equation}\label{pro-kdv3-p1}
a_j =  \frac{\lambda_{j+1}e^{\lambda_{j+2}L} - \lambda_{j+2} e^{\lambda_{j+1} L}}{\sum_{k=1}^3 \lambda_k e^{\lambda_k L}(\lambda_{k+1}e^{\lambda_{k+2}L} - \lambda_{k+2} e^{\lambda_{k+1} L})} \hh_3 \mbox{ for } j=1, 2, 3. 
\end{equation}
The rest of the proof is almost the same as the one of \Cref{pro-kdv1} and is left to the reader. \qed

\section{The set of the critical lengths} \label{sect-CL}

The main goal of this section is to establish \Cref{thm-MD}, except for \eqref{def-MD} and its consequence \eqref{dim-MD}, and \Cref{corCha-L}. The results in this section, in particular in \eqref{thmCha-L} will be complemented with the analysis in the next section (\Cref{sect-US}) to derive \eqref{def-MD}  and \eqref{dim-MD}.  We begin with 

\begin{proposition} \label{thmCha-L} Assume that $L \in \cN_D$. There exists $a < 0$ and $b > 0$ such that 
\be \label{thmCha-L-st1}
(a+ ib) e^{(a + ib)} = (a - ib) e^{(a - i b)}  = - 2 a e^{ - 2a}
\ee
and 
\be \label{thmCha-L-st2}
(a+ ib)^2 + (a - ib)^2 + (a + i b) (a - ib ) = - L^2.  
\ee
These conditions are equivalent to   
\be \label{thmCha-L-st3}
b^2 = L^2 + 3 a^2, \quad a^2 + b^2 = 4 a^2 e^{-6 a},  \quad  b \cos b + a \sin b =0, 
\ee
\be \label{thmCha-L-st4}
a \cos b - b \sin b > 0. 
\ee
Moreover, 
\be \label{thmCha-L-st5}
\mbox{$a$ and $b$ are uniquely determined by $L$.} 
\ee
\end{proposition}

\begin{proof} The proof is divided into two steps. 

\medskip 
\noindent \underline{Step 1}: We prove \eqref{thmCha-L-st1} - \eqref{thmCha-L-st4}. 

Let $z_1 = a + i b$ and $z_2 = c + i d$ where $a, b, c, d \in \mR$. 
Set  
$$
z_3 = - (z_1 + z_2) = - (a + c) - i (b + d). 
$$
Note that, if $z_1 + z_2 + z_3 = 0$, then 
\be \label{thmCha-L-z1z2z3}
z_1^2 + z_2^2 + z_1 z_2 = - (z_1 z_2 + z_2 z_3 + z_3 z_1). 
\ee
Condition \eqref{cond-z1z2} can be then written as 
\begin{equation}\label{cond-z1z2z3}
z_1 + z_2 + z_3 =0, \quad z_1 e^{z_1} = z_2 e^{z_2} = z_3 e^{z_3}, \quad \mbox{ and } \quad  z_1 z_2 + z_2 z_3 + z_3 z_1 = L^2. 
\end{equation}

In the three numbers $a$, $c$, and $- (a + c)$, there are two numbers of the same sign. Note also that if $(z_1, z_2, z_3)$ is a solution of \eqref{cond-z1z2z3}, then $(\bar z_1, \bar z_2, \bar z_3)$ is also a solution.  Without loss of generality, one might assume from now on that 
\be\label{thmCha-L-sign-ac}
a c \ge 0,  \quad  |c| \ge |a|,  
\ee
and 
\be\label{thmCha-L-sign-b1}
b \ge 0.   
\ee

From \eqref{cond-z1z2}, we have 
$$
(a + i b)^2 + (c + i d)^2 + (a + i b) (c + i d) = - L^2. 
$$
This yields 
$$
a^2 - b^2 + c^2 - d^2  + a c - b d  + i (2 ab + 2 cd + ad + bc) = - L^2.  
$$
We  thus obtain 
\be \label{thmCha-L-cond1}
a^2 + ac + c^2 + L^2 = b^2 + bd + d^2
\ee
and 
\be \label{thmCha-L-cond2}
2 ab + 2 cd + ad + bc = 0. 
\ee

From \eqref{cond-z1z2}, we have 
\be \label{thmCha-L-cond-PP1}
(a + i b) e^{a + i b} = (c + i d) e^{c + i d} = - \big(a + c + i (b + d) \big) e^{- (a + c) - i (b + d)}.
\ee
Considering the modulus in \eqref{thmCha-L-cond-PP1}, we get
\be\label{thmCha-L-cond-modulus}
(a^2 + b^2) e^{2 a} = (c^2 + d^2) e^{2 c} = \big( (a+ c)^2 + (b+d)^2 \big) e^{ - 2 (a + c)}. 
\ee
Considering  the imaginary parts in \eqref{thmCha-L-cond-PP1}, we obtain 
\be\label{thmCha-L-cond-phase}
(b \cos b + a \sin b) e^a = (d \cos d + c \sin d) e^c =  \big( -(b+d) \cos (b+d) + (a + c) \sin (b +d) \big) e^{- (a +c)}. 
\ee

This implies that $c  \neq 0$ since otherwise $a$ is also 0 (since $|a| \le |c|$) and one reaches, by \eqref{thmCha-L-cond-modulus},  
$$
|b| = |d| = |b +d|, 
$$
which yields $b = d = 0$ and $L = 0$. This is impossible since $L>0$.

Using \eqref{thmCha-L-cond2} and the facts $ac \ge 0$ and $c \neq 0$, we have 
\be\label{thmCha-L-cond-d-tb}
d=  -  \frac{2 a + c}{2 c + a} b. 
\ee
This implies
\be \label{thmCha-L-cond-b+d}
b + d = \frac{c - a}{2 c + a} b.  
\ee
Combining \eqref{thmCha-L-sign-b1} and \eqref{thmCha-L-cond-d-tb} yields that 
\be\label{thmCha-L-sign-b2}
b > 0 
\ee
since $b^2 + bd + d^2 \ge L^2$. 

\medskip 
The proof now is divided into four cases.  

\medskip 

\noindent $\bullet$ Case 1: $a \ge 0$. Thus  $c \ge a \ge 0$. Since   $ (a+ c)^2 e^{ - 2 (a + c)} \le e^{-2}$ for $a + c \ge 0$,  it follows that 
\begin{multline*}
(b+d)^2 + e^{- 2}  \ge (b+d)^2 + (a+ c)^2  e^{ - 2 (a + c)} \\[6pt] \ge
 \big( (a+ c)^2 + (b+d)^2 \big) e^{ - 2 (a + c)}   \mathop{=}^{\eqref{thmCha-L-cond-modulus}}  (a^2 + b^2) e^{2 a} \ge b^2 e^{2 a} \ge b^2. 
\end{multline*}
In summary, one reaches
\be \label{thmCha-L-1.1-bd}
(b+d)^2 + e^{- 2}  \ge b^2. 
\ee

One has, for $c \ge a \ge 0$, 
\be \label{thmCha-L-bd-p1}
(b+d)^2 \mathop{=}^{\eqref{thmCha-L-cond-b+d}}  \left( \frac{c-a}{2 c + a} \right)^2 b^2 \le b^2/4. 
\ee
Combining \eqref{thmCha-L-1.1-bd} and \eqref{thmCha-L-bd-p1} yields 
\be\label{thmCha-L-b-p1}
e^{-2} \ge \frac{3}{4} b^2; \mbox{ thus, by \eqref{thmCha-L-sign-b2},  } 0 <  b \le 0.425. 
\ee
Since $a \ge 0$,  it follows from \eqref{thmCha-L-b-p1} that 
\be \label{thmCha-L-p1}
b \cos b + a \sin b > 0.   
\ee

We have
\begin{multline*}
a+c \le (2 a^2 + 2 c^2 + 2 ac )^{1/2} \mathop{\le}^{\eqref{thmCha-L-cond1}} 
\left( 2 ( b^2 + d^2 + bd) \right)^{1/2} = \left( b^2 + d^2 + (b+d)^2 \right)^{1/2} \\[6pt]
 \mathop{\le}^{\eqref{thmCha-L-cond-d-tb}, \eqref{thmCha-L-bd-p1}}  \left(2 b^2 + b^2/4 \right)^{1/2} 
= (9 b^2/4)^{1/2} \mathop{\le}^{\eqref{thmCha-L-b-p1}} (3 e^{-2})^{1/2} \le 0.638.
\end{multline*} 
Since $0 \le b+d \le b/2 \le 0.225$ by \eqref{thmCha-L-cond-b+d}, \eqref{thmCha-L-bd-p1},  and \eqref{thmCha-L-b-p1},  it follows that  
\begin{multline}\label{thmCha-L-p1-2}
 -(b+d) \cos (b+d) + (a + c) \sin (b +d) \le - (b+d)\cos (0.225) +  (a + c) (b+d)  \\[6pt]
  \le - 0.97 (b + d) + 0.638 (b+d) \le 0. 
\end{multline}

Combining \eqref{thmCha-L-cond-phase}, \eqref{thmCha-L-p1}, and \eqref{thmCha-L-p1-2}, we derive that the case $a \ge 0$ does not happen.

\medskip 

\noindent $\bullet$ Case 2: $a < 0$ and $a \neq c$. Thus $c <  a < 0$ since $|c| \ge |a|$.

\medskip 
$\ast$ Case 2.1: $a < 0$ and $ c \le 2 a $. We have 
$$
d^2 \mathop{=}^{\eqref{thmCha-L-cond-d-tb}} \left( \frac{2a + c}{ 2 c + a} \right)^2 b^2  \le (4/5)^2 b^2. 
$$ 
It follows that  
$$
b^2 \le (a^2 + b^2) \mathop{=}^{ \eqref{thmCha-L-cond-modulus} } e^{2  (c - a )} (c^2 + d^2)  \le c^2 e^{c} + (4/5)^2 b^2. 
$$
This yields, since $c \le 0$,  
$$
(9/25) b^2 \le c^2 e^{c} \le 2^2 e^{-2}. 
$$
We obtain 
\be\label{thmCha-L-p2}
b^2  \le \frac{100}{9} e^{-2} \le 1.51, \mbox{ which yields, by \eqref{thmCha-L-sign-b2},  } 0 < b \le 1.23.
\ee

We have
\be
a^2 + b^2  \mathop{=}^{\eqref{thmCha-L-cond-modulus}}  \big( (a+ c)^2 + (b+d)^2 \big) e^{ - 4 a - 2 c} \ge (a+c)^2 e^{ - 2 a - 2 c}
\ee
since $a<0$.  It follows from \eqref{thmCha-L-p2} that 
$$
1.51 \ge b^2 \ge  (a+c)^2 e^{ - 2 a - 2 c} - a^2 \ge  (a+c)^2 \Big( e^{ - 2 a - 2 c} - \frac{1}{9} \Big).
$$
Since $f_1(t) =  t^2 (e^{2 t} - 1/9) $ is an strictly increasing function for $t >  0$ and 
$$
f_1(0.66) \ge 1.55, 
$$
it follows that 
\be\label{thmCha-L-p3}
0.66 \ge  |a+c| = |c| + |a| \ge 3|a|. 
\ee

Using \eqref{thmCha-L-p2} and \eqref{thmCha-L-p3}, we obtain 
\be\label{thmCha-L-p4}
b \cos b + a \sin b  \ge b \cos  1.23 + a b > 0.33 b + a b > 0.  
\ee
and, since $0 <  b+d \le b/2 \le 0.64$ by \eqref{thmCha-L-cond-b+d} and \eqref{thmCha-L-p2},  
\begin{multline}
\label{thmCha-L-p5}
 -(b+d) \cos (b+d) + (a + c) \sin (b +d)  \\[6pt]
 \le - (b + d) \cos 0.64 +  (a+c) (b+d) \cos 0.64  \mathop{<}^{\eqref{thmCha-L-p3}}  0. 
\end{multline}
Here we used the fact $\sin x \ge x \cos \theta_0  $ for $0 \le x \le \theta_0 \le 1$. 

Combining \eqref{thmCha-L-cond-phase}, \eqref{thmCha-L-p4}, and \eqref{thmCha-L-p5}, we derive that Case 2.1 does not happen. 

\medskip 
$\ast$ Case 2.2: $a < 0$, $ 2 a \le c  <  a$ with $|a| \le 1$.  We have 
$$
a^2 + b^2 \mathop{=}^{\eqref{thmCha-L-cond-modulus}} (c^2 + d^2) e^{2 (c-a)} \mathop{\le}^{\eqref{thmCha-L-cond-d-tb}} ( c^2 + b^2) e^{2 (c-a)}. 
$$
This implies, with $s = c -a < 0$,  
\begin{multline*}
b^2 (1 - e^{2 s}) = b^2 (1 - e^{2 (c-a)}) \le  c^2 e^{2 (c-a)} - a^2  \\[6pt]
= (a + s)^2 e^{2 s} - a^2 = a^2 e^{2 s} + 2 a s e^{2 s} + s^2 e^{2s} - a^2  \le (2 a s + s^2) e^{2 s}. 
\end{multline*}
We derive that 
\be \label{thmCha-L-2.2-p1}
b^2 \le \frac{ |2 a s + s^2|}{|1 - e^{2 s}|} e^{2 s}. 
\ee

Set 
$$
g(t) = 2 t e^{-2t} - 1 + e^{-2 t} \mbox{ for } t \ge 0.  
$$
Then 
$$
g'(t) =( 2 - 4t - 2) e^{-2 t} = - 4 t e^{-2 t} <  0 \mbox{ for } t > 0,  
$$
which yields  $g(t) < g(0) = 0$ for $t>0$. This implies 
\be\label{thmCha-L-2.2-claim1}
|s| e^{2s} \le \frac{1}{2}(1 - e^{2s}). 
\ee

Combining \eqref{thmCha-L-2.2-p1} and \eqref{thmCha-L-2.2-claim1} yields 
\be\label{thmCha-L-2.2-p2}
b^2 \le \frac{1}{2} (2 |a| + |s|) = \frac{1}{2} (|a| + |c|) \le 3/2, 
\ee
where in the last inequality we used the fact $|c| \le 2 |a|$ and $|a| \le 1$. 

We have
\be
a^2 + b^2  \mathop{=}^{\eqref{thmCha-L-cond-modulus}}  \big( (a+ c)^2 + (b+d)^2 \big) e^{ - 4 a - 2 c} \ge (a+c)^2 e^{ - 2 a - 2 c}. 
\ee
It follows from \eqref{thmCha-L-2.2-p2} that 
$$
\frac{1}{2} (|a| + |c|) \ge  b^2 \ge (a+c)^2 e^{ - 2 a - 2 c} - a^2
\ge (a+c)^2 (e^{ - 2 a - 2 c}   - \frac{1}{4}), 
$$
\be\label{thmCha-L-2.2-p3}
\frac{1}{2} \ge |a+c| (e^{ - 2 a - 2 c}   - \frac{1}{4}), 
\ee
Since $f_2(t) =  t (e^{2 t} - 1/4) $ is an increasing function for $t \ge 0$ and 
$$
f_2(0.32) \ge 0.52, 
$$
it follows from \eqref{thmCha-L-2.2-p3} that 
\be \label{thmCha-L-2.2-p4}
|a| + |c| \le 0.32, \mbox{ which yields } |a| \le 0.16. 
\ee

Using \eqref{thmCha-L-2.2-p2} and \eqref{thmCha-L-2.2-p3}, we obtain 
\be\label{thmCha-L-2.2-p4-*}
b \cos b + a \sin b  \ge b \cos  \sqrt{3/2} + a b > 0.33 b + a b > 0,   
\ee
and, since $0 <  b+d \le b/2 \le \sqrt{3/2}/ 2 \le 0.62$ by \eqref{thmCha-L-cond-b+d} and \eqref{thmCha-L-2.2-p2},  
\begin{multline}
\label{thmCha-L-2.2-p5}
 -(b+d) \cos (b+d) + (a + c) \sin (b +d) \\[6pt]
 \le - (b + d) \cos 0.62 +  (a+c) (b+d) \cos 0.62 \mathop{<}^{\eqref{thmCha-L-2.2-p4}} 0. 
\end{multline}
Here we used the fact $\sin x \ge x \cos \theta_0  $ for $0 \le x \le  \theta_0 \le 1$. 

 Combining \eqref{thmCha-L-cond-phase}, \eqref{thmCha-L-2.2-p4-*}, and \eqref{thmCha-L-2.2-p5}, we derive that Case 2.2 does not happen. 

\medskip
$\ast$ Case 2.3: $a < 0$ and $ 2a \le c <  a$ with $|a| >  1$. We have, by \eqref{thmCha-L-cond-modulus},  
$$
(a^2 + b^2)e^{2a} = (c^2 + d^2) e^{2 c}  
$$
Since $|c| \ge |a| \ge 1$ and $d^2 \le b^2$ by \eqref{thmCha-L-cond-d-tb}, and the function $f_3(t) = t^2 e^{-2 t}$ is strictly decreasing in $[1, + \infty)$, we derive that 
$$
(a^2 + b^2)e^{2a} >   (c^2 + d^2) e^{2 c}.  
$$
Thus Case 2.3 does not happen.

\medskip 
\noindent $\bullet$ Case 3: $a < 0$ and $a = c$.  We derive from \eqref{thmCha-L-cond-b+d} that 
$$
b + d = 0. 
$$
Assertion \eqref{thmCha-L-st1}-\eqref{thmCha-L-st4} then follow from \eqref{thmCha-L-cond-modulus}, \eqref{thmCha-L-cond-phase} and the fact that
$$
\Re (a+ ib) e^{(a + ib)} = e^{a} (a \cos b - b \sin b) > 0. 
$$

\medskip 
\noindent $\bullet$ Case 4: $a = 0$ and $c<0$. It follows that 
$$
d = - b/2. 
$$
This case does not happen by \eqref{thmCha-L-cond-modulus}. 

\medskip 
\underline{Step 2}: We prove \eqref{thmCha-L-st5}. Since 
$$
L^2 = b^2 - 3 a^2  = 4 a^2 (e^{- 6 a} - 1) 
$$
and the function $h(t) = 4 t^2 (e^{6t} - 1)$ is strictly increasing in $(0, + \infty)$, it follows that  
$$
a \mbox{ is uniquely determined by } L. 
$$
This in turn implies that  
$$
b \mbox{ is uniquely determined by } L 
$$
since $L^2 = b^2 - 3 a^2$ and $b > 0$. 

\medskip 
The proof is complete. 
\end{proof} 

We are ready to give 

\begin{proof}[Proof of \Cref{corCha-L}] By \eqref{thmCha-L-st3} and \eqref{thmCha-L-st4}, we have 
$$
L^2 = 4 a^2 (e^{-6a} - 1), 
$$
where $a< 0$ is determined by, with $b^2 = 4a^2 (e^{-6a} - 1/4)$, $b>0$,  
$$
b \cos b + a \sin b = 0 \quad \mbox{ and } \quad  a \cos b  - b \sin b > 0. 
$$
Since $b > 0$ and $a < 0$, it follows that 
$$
\cos b< 0 \quad \mbox{ and } \quad  \sin b < 0. 
$$
Hence $b = b_n$ for some $\pi + 2 n \pi <  b_n <    3 \pi/ 2 + n 2 \pi$ with $n \in \N$. 
Since  $b_n \ge \pi$ and $4 a_n^2 (e^{- 6 a_n} - 1/4) = b_n^2$, it follows that 
$$
|a_n| \ge 0.43. 
$$
This yields 
\be
|\cot b_n| = |a_n/b_n| = \frac{1}{4 |a_n| (e^{6 |a_n|} - 1/4)} \le 0.045. 
\ee
We thus derive that
$$
|b_n - 3 \pi/2| \le 0.045. 
$$

Set 
$$
t_n = b_n - \pi - n 2 \pi. 
$$
Since 
$$
|b_n \cos b_n| = |a_n \sin b_n|,  
$$
it follows that 
\be \label{corCha-L-p1}
(t_n + \pi + n 2 \pi) \cos t_n = |a_n| \sin t_n. 
\ee
For $t_n \in (\pi/4, \pi/2)$, the LHS of \eqref{corCha-L-p1} is a strictly decreasing function of $t_n$ and the RHS of \eqref{corCha-L-p1} is a strictly increasing function of $t_n$ since $|a_n|$ is a  increasing function of $t_n$. Hence there exists a unique $t_n \in (\pi/4,  \pi/2)$ such that the identity holds. One can check that the corresponding $a_n$ fulfills all the requirements. 
\end{proof}

\begin{remark}\rm It is clear, as $n \to + \infty$, that $a_n / b_n \to 0$ so $b_n - 3 \pi/2 - 2 \pi n \to 0$ as $n \to + \infty$.  
\end{remark}

We end this section with the following result, which yields \eqref{thm-MD-phi} and \eqref{sys-Phi}. Recall that $\alpha$ and $\beta$ are defined by \eqref{thm-MD-alphabeta} and $q$ is defined by \eqref{thm-MD-q}. 

\begin{lemma} \label{lem-Phi-Psi} Let $L \in \cN_D$. 
Set
\be \label{def-eta}
\eta_1= \alpha + i \beta, \quad \eta_2 = \alpha - i \beta, \quad \mbox{ and } \quad  \eta_3 = - 2 \alpha, 
\ee
\be
\varphi (x) =  (\eta_3 - \eta_2) e^{\eta_1 x} + (\eta_1 - \eta_3) e^{\eta_2 x} + (\eta_2 - \eta_1) e^{\eta_3 x} \mbox{ in } [0, L],  
\ee
and 
$$
\Psi (t, x) = e^{qt} \psi(x) \mbox{ in } \mR \times [0, L]. 
$$
Then $\varphi (x) =  - 2i \phi$
and $\Psi$ satisfies \eqref{sys-Phi}. Consequently, \eqref{thm-MD-phi} and \eqref{sys-Phi} hold. 
\end{lemma}

\begin{proof}
Since 
$$
\eta_1 e^{- \eta_1 L} = \eta_2 e^{ - \eta_2 L } = \eta_3 e^{ - \eta_3 L}
$$
$$
\eta_1 + \eta_2 + \eta_3 = 0, \quad \eta_1 \eta_2 + \eta_2 \eta_3 + \eta_3 \eta_1 = 1,  \quad \mbox{ and } \quad \eta_1 \eta_2 \eta_3 = - q, 
$$
it follows that 
$$
\left\{\begin{array}{cl}
\varphi_{xxx} + \varphi_x + q \varphi = 0 &  \mbox{ in } (0, L), \\[6pt]
\varphi(0) = \varphi_x (0) = \varphi(L) = \varphi_{xx} (L) = 0 & \mbox{ in } [0, L]. 
\end{array} \right. 
$$
We derive that $\Psi$ is a solution of \eqref{sys-Phi}. A computation gives
\begin{align*}
\varphi (x)  = & (- 3 \alpha + i \beta) e^{(\alpha + i \beta ) x} + (3 \alpha + i \beta) e^{(\alpha - i \beta ) x} - 2 i \beta e^{- 2 \alpha x} \\[6pt]
= &  - 6 i \alpha e^{\alpha x} \sin (\beta x) + 2 i \beta e^{\alpha x} \cos (\beta x) - 2 i \beta e^{- 2 \alpha x} \\[6pt]
= &  2i \Big(\beta e^{\alpha x} \cos (\beta x) -  \beta e^{- 2 \alpha x} - 3 \alpha e^{\alpha x} \sin (\beta x) \Big) = - 2 i \phi.
\end{align*}
The proof is complete. 
\end{proof}

\section{The unreachable space of the linearized system for a critical length} \label{sect-US}

In this section, we prove that  the unreachable space is given by \eqref{def-MD} and study its controllability properties of  the linearized KdV system:
\begin{equation}\label{LKdV1}\left\{
\begin{array}{cl}
y_t  + y_x  + y_{xxx}  = 0 &  \mbox{ in }  (0, T) \times (0, L), \\[6pt]
y(\cdot, 0) = y_x(\cdot, L) = 0 & \mbox{ in }  (0, T), \\[6pt]
y(\cdot , L) = u & \mbox{ in }  (0, T), 
\end{array}\right.
\end{equation}
for a critical length. The main result of this section is \Cref{pro-C} which is based on an observability inequality with initial data and final data in $\cM_D^\perp$, where $\cM_D$ is defined by \eqref{def-MD}.  In comparison with the right Neumann boundary control system, this part in the Dirichlet setting is more complex and technical.  The proof of \Cref{pro-C} involves the well-posedness and the estimates for solutions in $X_T$ given in \Cref{pro-kdv1} and \Cref{pro-kdv2}. 

\medskip 
We begin with the following simple but useful result. 

\begin{lemma} \label{lem-projection} Let $L \in \cN_D$ and $T> 0$. For $y_0 \in L^2(0, L)$ and $u \in H^{1/3}(0, T)$, let $y \in X_T$ be the unique solution of the system 
\begin{equation*}\left\{
\begin{array}{cl}
y_t  + y_x  + y_{xxx}  = 0 &  \mbox{ in }  (0, T) \times (0, L), \\[6pt]
y(\cdot, 0) = y_x(\cdot, L) = 0 & \mbox{ in }  (0, T), \\[6pt]
y(\cdot , L) = u & \mbox{ in }  (0, T), \\[6pt]
y(0, \cdot) = y_0 & \mbox{ in } (0, L). 
\end{array}\right.
\end{equation*}
Then 
$$
\int_0^L y(t, x) \Phi(t, x) \, dx \mbox{ is constant on } [0, T]. 
$$
\end{lemma}

Recall that $\Phi$ is defined by \eqref{thm-MD-Phi}.

\begin{proof} Using the fact, by integration by parts, 
$$
\frac{d}{dt} \int_0^L y(t, x) \Phi(t, x) \, dx = 0, 
$$
the conclusion follows.  
\end{proof}

We are ready to state the main result of this section on the unreachable space for the linearized KdV system \eqref{LKdV1}. 

\begin{proposition}\label{pro-C} Let $L \in \cN_D$ and $T > 0$. We have

$i)$ for $\varphi  \in \cM_D \setminus \{0 \}$, there does not exist $u \in H^{1/3}(0, T)$ such that $y(T, \cdot) = \varphi$ where $y \in X_T$ is the unique solution of  \eqref{LKdV1} with $y(0, \cdot) = 0$. 

$ii)$ there exists a linear continuous operator $\cL:  \cM_D^\perp \to H^{1/3}(0, T)$ such that  $y(T, \cdot) = \varphi$ where $y \in X_T$ is the unique solution of \eqref{LKdV1}
with $y(0, \cdot) = 0$ and $u = \cL(\varphi)$.  

$iii)$ There exists a linear continuous operator $\hat \cL:  \cM_D^\perp \to H^{1/3}(0, T)$ such that  $y(T, \cdot) = 0$ where $y \in X_T$ is the unique solution of \eqref{LKdV1} with $y(0, \cdot) = \varphi$ and $u = \hat \cL(\varphi)$. 
\end{proposition}

\begin{proof} Assertion $i)$ is just a consequence of \Cref{lem-projection}. 

We next deal with assertions $ii)$ and $iii)$.   Let $\psi \in \M_D^\perp$. Let  $y^* \in X_T$ be the unique solution  of the backward linear KdV system
\begin{equation}\label{Sys-KdV-BW}\left\{
\begin{array}{cl}
y^*_t   + y^*_x  + y^*_{xxx}  = 0 &  \mbox{ in } (0, T) \times (0, L), \\[6pt]
y^*(\cdot, 0)  = y^*_x(\cdot, 0) = y^*(\cdot, L) = 0 & \mbox{ in } (0, T), \\[6pt]
y^*(T, \cdot) = \psi & \mbox{ in } (0, L).
\end{array}\right.
\end{equation}
Applying the observability inequality in \Cref{lem-obs} below to $y^*(T - \cdot, L - \cdot)$, we obtain 
\begin{equation}\label{pro-C-HUM}
\lambda^{-1} \|\psi \|_{L^2(0, L)}  \le   \| y^*_{xx} (\cdot, L) \|_{[H^{1/3}(0, T)]^*} \le \lambda \|\psi \|_{L^2(0, L)} 
\end{equation}
for some constant $\lambda \ge 1$.  Fix a continuous linear mapping 
\be \label{def-L1}
\cL_1: [H^{1/3}(0, T)]^* \to H^{1/3}(0, T)
\ee
such that 
\be \label{property-L1}
\langle \cL_1(h), h \rangle_{H^{1/3}(0, T); [H^{1/3}(0, T)]^*} \ge C \| h\|_{[H^{1/3}(0, T)]^*}^2 \quad  \mbox{ for all } h \in [H^{1/3}(0, T)]^*,  
\ee 
for some positive constant $C$.  This can be done using the Fourier series or the Fourier transform appropriately.  

We first prove assertion $ii)$. Equipped $\cM_D^\perp$ with the $L^2(0, L)$-scalar product. Define 
$$
\cA: \cM_D^\perp \to \cM_D^\perp
$$
by 
$$
\cA(\psi) = y(T, \cdot), 
$$
where $y \in X_T$ is the unique solution of the following system 
\begin{equation}\label{Sys-KdV-FW}\left\{
\begin{array}{cl}
y_t   + y_x  + y_{xxx}  = 0 &  \mbox{ in } (0, T) \times (0, L), \\[6pt]
y(\cdot, 0)  = y_x(\cdot, L) =0  & \mbox{ in } (0, T), \\[6pt]
y(\cdot, L)  = h & \mbox{ in } (0, T), \\[6pt]
y(0, \cdot) = 0 & \mbox{ in } (0, L),
\end{array}\right.
\end{equation}
with $h = \cL_1 (y_{xx}^*(\cdot, L))$ where $y^*$ is determined by \eqref{Sys-KdV-BW} ($y(T, \cdot)\in \cM_D^\perp$ since $y(0, \cdot) = 0$ by \Cref{lem-projection}). An integration by part yields 
 \begin{equation}\label{pro-C-LF}
\int_0^L  y(T, x) z^*(T, x)\, dx = - \int_0^T y(t, L) z^*_{xx} (t, L) \, dt, 
\end{equation}
for all solutions $z^* \in X_T$ of \eqref{Sys-KdV-BW} with $z^*(T, \cdot) \in \cM_D^\perp$. 

Using \eqref{pro-C-HUM} and applying the Lax-Milgram theory, we derive that $\cA$ is linear continuous and its inverse is also linear continuous.  The conclusion of $ii)$ follows by taking $\cL(\psi) = \cL_1(y^*_{xx}(\cdot, L))$ and $y^*$ is the solution of \eqref{Sys-KdV-BW} with $\psi$ being replaced by $\cA^{-1} (\psi)$.

We now deal with assertion $iii)$. 
Set 
$$
H = \Big\{ y^{*}_{xx} (\cdot, L) \in [H^{1/3}(0, T)]^*; \mbox{ where $y^*$ is determined by \eqref{Sys-KdV-BW} with $\psi \in \cM_D^\perp $} \Big\}.  
$$
It follows from \eqref{pro-C-HUM} that $H$ is a closed subspace of $[H^{1/3}(0, T)]^*$ so is a Hilbert space.  We next consider the following bilinear form on $H$: 
$$
\hat a( u, v) = \langle \cL_1(u), v \rangle_{H^{1/3}(0, T); [H^{1/3}(0, T)]^*}. 
$$
Using \eqref{pro-C-HUM}, we derive  from the Lax-Milgram theorem  that there exists a continuous linear application $\hat \cA: \cM_D^\perp \to H$ such that
 \begin{equation}\label{pro-C-LF-iii}
\int_0^L  \varphi (x) y^*(0, x)\, dx = \hat a( \hat \cA \varphi, y^*_{xx}(\cdot, L)) 
\end{equation}
for all solution $y^* \in X_T$ of \eqref{Sys-KdV-BW} with $\psi \in \cM_D^\perp$. 

Set 
$$
\hat \cL = \cL_1 \circ \hat \cA. 
$$
The conclusion of $iii)$ then follows after noting that if $y(0, \cdot) \in \cM_D^\perp$ then $y(T, \cdot) \in \cM_D^\perp$ by \Cref{lem-projection},  and the fact
\begin{equation*}
\int_0^L  y(T, x) y^*(T, x)\, dx - \int_0^L  y(0, x) y^*(0, x)\, dx = - \int_0^T y(t, L) y^*_{xx} (t, L) \, dt. 
\end{equation*}

The proof is complete. 
\end{proof}

Here is the observability inequality used in the proof of \Cref{pro-C}. 

\begin{lemma}\label{lem-obs} Let $L \in \cN_D$, $T>0$,  and let  $y \in X_T$ be a solution of the linearized KdV equation in $(0, T) \times (0, L)$ with $y(0, L - \cdot) \in \cM_D^\perp$ and with $y(\cdot, 0) = y(\cdot, L) = y_x( \cdot, L) = 0$. Then there exists $\Lambda \ge 1$ depending only on $L$ and $T$ such that 
\begin{equation}\label{lem-obs-st}
\Lambda^{-1} \| y(0, \cdot) \|_{L^2(0, L)} \le \| y_{xx}(\cdot, 0) \|_{[H^{1/3}(0, T)]^*} \le \Lambda \| y(0, \cdot) \|_{L^2(0, L)}. 
\end{equation}
\end{lemma}

\begin{proof} By \Cref{pro-kdv1}, we have 
\begin{equation}\label{lem-obs-st00}
 \| y_{xx}(\cdot, 0) \|_{[H^{1/3}(0, T)]^*} \le C \| y(0, \cdot) \|_{L^2(0, L)}. 
\end{equation}

We next prove 
\begin{equation}\label{lem-obs-st01}
\| y(0, \cdot) \|_{L^2(0, L)} \le C \| y_{xx}(\cdot, 0) \|_{[H^{1/3}(0, T)]^*}
\end{equation}
by contradiction. Assume that there exists a sequence $(\varphi_n)$ such that  $\varphi_n(L - \cdot) \in  \cM_D^\perp$ and  
\begin{equation}\label{lem-obs-p1}
\| y_{n, xx}(\cdot, 0) \|_{[H^{1/3}(0, T)]^*} \le  \frac{1}{n} \| y_n(0, \cdot) \|_{L^2(0, L)} = \frac{1}{n},  
\end{equation}
where $y_n \in X_T$  is the unique solution of the linearized KdV equation in $(0, T) \times (0, L)$ with  $y_n(0, \cdot) = \varphi_n$ and $y_n(\cdot, 0) = y_n(\cdot, L) = y_{n, x}( \cdot, L) = 0$. 
Set 
\be
\by_n(t, x) = y_n(T-t, L-x). 
\ee
Then $\by_n \in X_T$ is a solution of the equation $\by_{n, t}   + \by_{n, x}  + \by_{n, xxx}  = 0$ in $(0, T) \times (0, L)$.  
By the regularizing effect of the linearized KdV equation,  without loss of generality, one might assume that  $\by_n(0,  \cdot) = y_n(T, L - \cdot)$ converges in $L^2(0, L)$.  Applying \Cref{pro-kdv2} to $\by_n$, one derives that  $y_n(0, L - \cdot) = \by_n(T,\cdot)$ is a Cauchy sequence in $L^2(0, L)$. In other words, $\varphi_n = y_n(0, \cdot)$ is a Cauchy sequence in $L^2(0, L)$. Let $\varphi$ be the limit of $\varphi_n$ in $L^2(0, L)$ and denote $y \in X_T$ be the corresponding solution of the linearized KdV system. Then $\| \varphi\|_{L^2(0, L)} = 1 $ and $\varphi(L - \cdot ) \in \cM_D^\perp$. 

Set $\by(t, x) = y(T-t,  L - x)$. 
Then $\by \in X_T$ is a solution of the system 
\begin{equation}\label{lem-obs-p3}
\by_t + \by_x + \by_{xxx} = 0 \mbox{ in } (0, T) \times (0, L), 
\end{equation}
\begin{equation}\label{lem-obs-p4}
\by(\cdot, 0) = \by_x(\cdot, 0) = \by(\cdot, L) = \by_{xx}( \cdot , L)= 0 \mbox{ in } (0, T),
\end{equation}
and 
$$
\by(0, \cdot) \in \cM_D^\perp  \; (\mbox{by \Cref{lem-projection}}) \quad \mbox{ and } \quad  \quad \| \by(0, \cdot) \|_{L^2(0, L)} \mathop{\ge}^{\Cref{pro-kdv2}} C \| \varphi\|_{L^2(0, L)} = C. 
$$
By \Cref{lem-Phi} below, $\by (0, \cdot)= c \phi$ in $(0, L)$ for some $c \in \mR$ where $\phi$ is defined by \eqref{thm-MD-phi}. Since $\by (0, \cdot) \in \cM_D^\perp$, it follows that $c =0$ and hence $\by (0, \cdot)= 0$ in $(0, L)$
We obtain a contradiction. Thus \eqref{lem-obs-st01} is proved. 

The conclusion now follows from \eqref{lem-obs-st00} and \eqref{lem-obs-st01}. 
\end{proof}

In \Cref{lem-obs}, we used the following result, which is also helpful in the proof of Assertion ii) of \Cref{thm2}.

\begin{lemma} \label{lem-Phi}Let $L \in \cN_D$ and  $T > 0$. Assume that $y \in X_T$ is a solution of the system 
\begin{equation}\label{lem-Phi-st1}
\left\{\begin{array}{cl}
y_t + y_x + y_{xxx} = 0  & \mbox{ in } (0, T) \times (0, L), \\[6pt]
y(\cdot, 0) = y_x(\cdot, 0) = y(\cdot, L) = y_{xx}( \cdot , L)= 0 & \mbox{ in } (0, T).
\end{array} \right. 
\end{equation}
There exists some $c \in \mR$ such that 
\be \label{lem-Phi-cl1}
y = c \Phi \mbox{ in } (0, T)  \times (0, L). 
\ee
\end{lemma}

\begin{proof} Set 
\be
\cV = \Big\{\varphi \in L^2(0, T) \cap \cM_D^\perp;  \exists \, y \in X_T \mbox{ satisfying }  \eqref{lem-Phi-st1}   \mbox{ and } y(0, \cdot) = \varphi \Big\} \subset L^2(0, L). 
\ee

We claim that 
\be \label{lem-Phi-claim}
\cV = \{ 0 \}. 
\ee
Admitting the claim, the conclusion now follows from the claim as follows. Set 
$$
\varphi = \proj_{\cM_D} y(0, \cdot).    
$$
Then 
$$
\varphi = c \phi \mbox{ for some $c \in \mR$}. 
$$ 
Set 
$$
\ty = y - c \Phi \mbox{ in } (0, T) \times (0, L), 
$$
where $\Phi$ is defined by \eqref{thm-MD-Phi}. Then $\ty \in X_T$ is a solution of \eqref{lem-Phi-st1} and $\ty(0, \cdot) \in \cM_D^\perp$. By \eqref{lem-Phi-claim}, one has 
$$
\ty(0, \cdot) =0. 
$$
It follows that 
$$
y = c \Phi \mbox{ in } (0, T) \times (0, L),  
$$
which is the conclusion.

It remains to prove \eqref{lem-Phi-claim}. We will prove that $\cV = \{0 \}$ by contradiction.  Assume that 
$\cV \neq \{0 \}$. Using the regularity theory for the linear KdV equation and \Cref{pro-kdv2}, one can show that 
$$
\cV \subset C^\infty([0, L]). 
$$
For the same reason, one can show that any bounded sequence in $\cV$ (equipped $L^2(0, L)$-norm) has a subsequence converging in $\cV$. Thus $\cV$ is of finite dimension and is not $\{0 \}$.  

We can now involve the arguments via spectral theory  in the spirit of \cite{GG10} or even simpler (see also \cite{BLR92, Rosier97}) to show that there exists $\varphi \in \cV \setminus \{0 \}$  and $\lambda \in \mC$ such that 
\begin{equation}\label{lem-Phi-p1-1}
\lambda \varphi + \varphi_x + \varphi_{xxx} = 0 \mbox{ in } (0, L), 
\end{equation}
\begin{equation}\label{lem-Phi-p2-1}
\varphi(\cdot, 0) = \varphi(\cdot, 0) = \varphi_x(\cdot, L) = \varphi_{xx}( \cdot , L)= 0. 
\end{equation}
Indeed, this can be done by considering
\begin{equation}
\begin{array}{ccc}
\cA: \cV & \to & \cV \\[6pt]
\psi & \mapsto & -(\psi_x + \psi_{xxx} )
\end{array}
\end{equation}
and taking $\lambda \in \mC$ and $\varphi \in \cV \setminus \{0 \}$ such that $\cA \varphi =  \lambda \varphi$. 
The only point required to be checked is the fact that $\psi_x + \psi_{xxx} \in \cV$ for $\psi \in \cV$. To this end,  one just notes that $-(\psi_x + \psi_{xxx}) =  y_t(0, \cdot)$ where $y \in X_T$ is the corresponding solution (thus $ y(0, \cdot) = \psi$). 

Let $\eta_j$ with $j=1, 2, 3$ be the three solutions of the equations 
$$
\eta^3 + \eta + \lambda = 0.  
$$
Then  $\varphi$ has the form 
$$
\varphi = \gamma_1 e^{\eta_1 x} + \gamma_2 e^{\eta_2 x} + \gamma_3 e^{\eta_3 x} \mbox{ in } [0, L]. 
$$
Using the boundary conditions of $\varphi$, we obtain 
$$
\left\{\begin{array}{c}
\gamma_1 + \gamma_2 + \gamma_3 = 0, \\[6pt]
\gamma_1 \eta_1 + \gamma_2 \eta_2 + \gamma_3 \eta_3 = 0, \\[6pt]
\gamma_1 e^{\eta_1 L} + \gamma_2 e^{\eta_2 L} + \gamma_3 e^{\eta_3 L} =0,  \\[6pt]
\gamma_1 \eta_1^2 e^{\eta_1 L} + \gamma_2 \eta_2^2 e^{\eta_2 L} + \gamma_3 \eta_3^2 e^{\eta_3 L} = 0.  
\end{array} \right. 
$$
As in \cite[(30)]{GG10}, we then derive that 
$$
\eta_1 e^{-\eta_1 L} = \eta_2 e^{-\eta_2 L} = \eta_3 e^{-\eta_3 L}.  
$$
Without loss of generality, one can assume that  
$$
- \eta_1 L  = z_1 \quad \mbox{ and } \quad - \eta_2 L = z_2, 
$$
(here $z_1$ and $z_2$ are the complex numbers in the definition of $\cN_D$) which yields 
$$
\varphi = c \phi
$$
for some constant $c \in \mC$.  In other words, 
$$
\varphi \in \cM_D. 
$$
We have a contradiction since $\varphi \neq 0$ and $\varphi \in \cM_D^\perp$. 

\medskip
The proof is complete. 
\end{proof}

By the same arguments used in the proof of \Cref{lem-obs}, we also have the following observability inequality in the case $L \not \in \cN_D$, which is the key point of the proof of \Cref{thm3}.  

\begin{lemma}\label{lem-obs2} Let $L \not \in \cN_D$, $T>0$,  and  let  $y \in X_T$ be a solution of the linearized KdV equation with $y(0, \cdot) \in L^2(0, L)$ and $y(\cdot, 0) = y(\cdot, L) = y_x( \cdot, 0) = 0$. Then, for some $\Lambda \ge 1$,  
\begin{equation}\label{lem-obs2-st}
\Lambda^{-1} \| y(0, \cdot) \|_{L^2(0, L)} \le \| y_{xx}(\cdot, 0) \|_{[H^{1/3}(0, T)]^*} \le \Lambda \| y(0, \cdot) \|_{L^2(0, L)}. 
\end{equation}
\end{lemma}

Here is a variant of \Cref{pro-C} in the case $L \not \in \cN_D$. 
\begin{proposition}\label{pro-C-n} Let $L \not \in \cN_D$ and $T>0$. Then 

a)  There exists a linear continuous operator $\cL:  L^2(0, L) \to H^{1/3}(0, T)$ such that  $y(T, \cdot) = \varphi$ where $y \in X_T$ is the unique solution of \eqref{LKdV1}
with $y(0, \cdot) = 0$ and $u = \cL(\varphi)$. 

b)  There exists a linear continuous operator $\hat \cL:  L^2(0, L) \to H^{1/3}(0, T)$ such that  $y(T, \cdot) = 0$ where $y \in X_T$ is the unique solution of \eqref{LKdV1} with $y(0, \cdot) = \varphi$ and $u = \hat  \cL(\varphi)$. 

\end{proposition}

We are ready to give 

\begin{proof}[Proof of \Cref{thm3}] \Cref{thm3} follows from \Cref{pro-WP} and \Cref{lem-obs2} as usual. The details are omitted. 
\end{proof}


\section{Properties of  controls which steer 0 at time $0$ to $0$ at time $T$} \label{sect-0-0}

In this section, we study controls that steer 0 at time $0$ to $0$ at time $T$ for the linearized KdV system of \eqref{sys-KdV}.  To this end, it is convenient to introduce the following quantities. 

\begin{definition}\label{def:Q-P}
For $z \in \mC$,  let $(\lambda_j)_{1\leq j \leq 3} =  \big(\lambda_j(z) \big)_{1\leq j \leq 3}$ be the three solutions  of
 \begin{equation}\label{eq-lambda}
  \lambda^3 + \lambda + i z = 0.
 \end{equation}
Set
 \begin{equation}\label{eq-defQ}
  Q = Q (z) : = 
 \begin{pmatrix}
 1 & 1 & 1\\
 e^{\lambda_1 L } & e^{\lambda_2 L } & e^{\lambda_3 L }\\
 \lambda_1 e^{\lambda_1 L } & \lambda_2 e^{\lambda_2 L } & \lambda_3 e^{\lambda_3 L }
 \end{pmatrix},
\end{equation}
\begin{equation}\label{def-P}
P_D =  P_D(z) : =  \sum_{j=1}^3  \lambda_j^2 \big(\lambda_{j+1}e^{\lambda_{j+1} L } - \lambda_{j+2} e^{\lambda_{j+2} L } \big), 
\end{equation}
 and
\begin{equation}\label{def-Xi}
\Xi = \Xi(z) :=  \det \begin{pmatrix}1&1&1\\ \lambda_1&\lambda_2&\lambda_3\\
\lambda_1^2&\lambda_2^2&\lambda_3^2\end{pmatrix}, 
\end{equation}with the convention $\lambda_{j+3} = \lambda_{j}$ for $j \ge 1$. 
\end{definition}

\begin{remark}\rm
The matrix $Q$ and the quantities $P_D$ and $\Xi$  are antisymmetric with respect
to $\lambda_j$ ($j=1, 2, 3$), and their definitions depend on a choice of the order of  $(\lambda_1, \lambda_2, \lambda_3)$. Nevertheless, we later consider a product  of either $P_D$, $\Xi$,  or $\det Q$ with another antisymmetric function of
$(\lambda_j)$,  or deal with $|\det Q|$,  and these quantities therefore make sense. The definitions of $P$, $\Xi$,  and $Q$ are only understood in these contexts.
\end{remark}

Given $u \in H^{1/3}(0, + \infty)$, let $y \in C([0, + \infty); L^2(0, L)) \cap L^2_{\loc}([0, + \infty); H^1(0, L))$ be the unique solution of the system 
\begin{equation}\left\{
\begin{array}{cl}
y_t + y_x  + y_{xxx} = 0 &  \mbox{ in } (0, +\infty) \times (0, L), \\[6pt]
y(\cdot, 0) = y_x(\cdot, L) = 0 & \mbox{ in } (0, +\infty), \\[6pt]
y(\cdot ,L) = u & \mbox{ in } (0, +\infty), \\[6pt]
y(0, \cdot)  = 0 &  \mbox{ in } (0, L). 
\end{array}\right.
\end{equation}
In what follows in this section, we extend $y$  and $u$ by $0$ for $t < 0$  and still denote these extensions by $y$ and $u$, respectively.  For an appropriate function $v$ defined on $\mR \times (0, L)$, let $\hat v$ denote its Fourier transform with respect to $t$, i.e., 
\begin{equation*}
\hat v(z, x) = \frac{1}{\sqrt{2 \pi} }\int_{\mR} v(t, x) e^{- i z t} \, dt. 
\end{equation*}
From the system of $y$, we have 
\begin{equation}
\left\{\begin{array}{cl}
i z \hy(z, x) + \hy_x (z, x) + \hy_{xxx} (z, x) = 0 &  \mbox{ in } \mR \times  (0, L), \\[6pt]
\hy(z, 0) = \hy_x(z, L) = 0 & \mbox{ in }  \mR, \\[6pt]
\hy(z, L) = \hu(z) & \mbox{ in } \mR.  
\end{array}\right.
\end{equation}
This system has a unique solution outside a discrete set of $z$ in $\mC$, see \cite[Lemma 2.1]{CKN20}. 

Taking into account the equation of $\hy$, we search for the solution of the form 
\begin{equation*}
\hy(z, \cdot) = \sum_{j=1}^3 a_j e^{\lambda_j x}, 
\end{equation*}
where $\lambda_j = \lambda_j(z)$ with $j=1, 2, 3$ determined by \eqref{eq-lambda}. 

Using the boundary conditions for $\hy$,  we require that 
\begin{equation*}
\left\{\begin{array}{cl}
\sum_{j=1}^3 a_j = 0, \\[6pt]
\sum_{j=1}^3 e^{\lambda_j L } a_j = \hu, \\[6pt]
\sum_{j=1}^3 \lambda_j e^{\lambda_j L } a_j = 0. 
\end{array}\right.
\end{equation*}
This implies 
\begin{equation*}
Q (a_1, a_2, a_3)\tr = (0,  \hu, 0)\tr, 
\end{equation*}
where $Q = Q(z)$ is given in \Cref{def:Q-P}. We thus obtain 
\begin{equation*}
a_j = \frac{\hu }{\det Q} \big(\lambda_{j+1}e^{\lambda_{j+1} L } - \lambda_{j+2} e^{\lambda_{j+2} L } \big). 
\end{equation*}
This yields 
\begin{equation}\label{form-hy}
\hy(z, x) =  \frac{\hu }{\det Q}  \sum_{j=1}^3 \big(\lambda_{j+1}e^{\lambda_{j+1} L } - \lambda_{j+2} e^{\lambda_{j+2} L } \big) e^{\lambda_j x}. 
\end{equation}
From \eqref{form-hy}, we derive that   
\begin{equation}
\partial_{xx} \hy(z, 0) =  \frac{\hu (z) P_D(z) }{\det Q (z)}, 
\end{equation}
where $P_D(z)$ is given in \Cref{def:Q-P}. 

Set 
\begin{equation}\label{def-GH}
G(z) = P_D(z) /  \Xi (z)  \quad \mbox{ and } \quad H(z) = \det Q(z) /  \Xi (z) , 
\end{equation}
where $\Xi(z)$ is given in \Cref{def:Q-P}.  It  is convenient to consider $\partial_{xx} \hy(z, 0)$ under the form 
\begin{equation}
\partial_{xx} \hy(z, 0) = \frac{\hu (z) G(z) }{H(z)}, 
\end{equation}

By \cite[Lemmas A1 and B1]{CKN20}, $\partial_{xx} \hy(z, 0) $ is a meromorphic function,  and  
\begin{equation}\label{GH}
\mbox{$G(z)$ and $H(z)$ are entire functions. }
\end{equation}

\medskip 
We thus have just established the following result. 

\begin{lemma} \label{lem-form-sol} Let $u \in H^{1/3}(\mR_+)$  and let $y \in C \big([0, + \infty); L^2(0, L) \big) \cap L^2_{\loc}\big( [0, + \infty); H^1(0, L) \big)$ be the unique solution of 
\begin{equation}\label{sys-y}\left\{
\begin{array}{cl}
y_t  + y_x  + y_{xxx}  = 0 &  \mbox{ in } (0, +\infty) \times (0, L), \\[6pt]
y(\cdot, 0) = y_x(\cdot, L) = 0 & \mbox{ in }  (0, +\infty), \\[6pt]
y(\cdot,  L) = u & \mbox{ in } (0, +\infty), 
\end{array}\right.
\end{equation}
with  
\begin{equation}\label{IC-y}
y(0, \cdot) =  0 \mbox{ in } (0, L). 
\end{equation}
Outside a discrete set $z \in \mR$, we have 
\begin{equation}
\hy (z, x) =  \frac{\hu }{\det Q}  \sum_{j=1}^3 \big(\lambda_{j+1}e^{\lambda_{j+1} L } - \lambda_{j+2} e^{\lambda_{j+2} L } \big)e^{\lambda_j x} \mbox{ for a.e. } x \in (0, L).
\end{equation}
\end{lemma}


\begin{remark}\label{rem-form-sol} \rm Assume that $\hu(z, \cdot)$ is well-defined for $z
\in \mC$ (e.g. when $u$ has a compact support). Then the conclusions of
\Cref{lem-form-sol} hold  outside of a discrete set $z\in \mC$.
\end{remark}

We end this section with the following result, which is the starting point of our approach, and follows from \Cref{lem-form-sol} and Paley-Wiener's theorem, see e.g., \cite{Rudin-RC}. 

\begin{proposition} \label{pro-Gen} Let $L>0$, $T>0$, and $u \in H^{1/3}(\mR_+)$. Assume that   $u$ has a compact support in $[0, T]$,  and $u$  steers $0$ at the time $0$ to $0$ at the time $T$, i.e., the unique solution $y$ of \eqref{sys-y} and \eqref{IC-y} satisfies $y(T, \cdot) = 0$ in $(0, L)$.  Then $\hu$ and  $\hu G/ H$  satisfy the assumption of Paley-Wiener's theorem concerning the support in $[-T, T]$, i.e., 
\be\label{pro-Gen-cl1}
\hu \mbox{ and } \hu G/H \mbox{ are entire functions}, 
\ee
and 
\be\label{pro-Gen-cl2}
|\hu(z)|  + \left|\frac{\hu G(z)}{H(z)} \right| \le C e^{T|z|}, 
\ee
for some positive constant $C$. 
\end{proposition}

\begin{remark} \label{rem-Gen} \rm The computations in this section are in the spirit of the ones \cite{CKN20}. Nevertheless, in the conclusions of \Cref{pro-Gen}, we have/require that
$$
\hy(z, L) \mbox{ and } \partial_{xx} \hy(z, 0)  \mbox{ are entire functions satisfying \eqref{pro-Gen-cl2}}. 
$$
This is different with the one used in \cite[Proposition 3.1]{CKN20} where one obtains that  
$$
\partial_x \hy(z, L)  \mbox{ and } \partial_{x} \hy(z, 0)  \mbox{ are entire functions satisfying a variant of \eqref{pro-Gen-cl2}}. 
$$
These differences are important to take into account different boundary conditions, see the proof of Assertion $ii)$ of \Cref{thm2} in \Cref{sec-Assertion-ii}. 
\end{remark}

\section{Attainable directions in the unreachable space in  small time}\label{sect-dir}

In this section, we investigate whether or not directions in $\cM_D$, defined in \eqref{def-MD},  can be reached in small time. The starting point comes from the power series expansion approach.  Let $y_1$ and $y_2$ be the solutions of 
\begin{equation}\label{eq:first_order}\left\{
\begin{array}{cl}
y_{1, t}  + y_{1, x}  + y_{1, xxx}  = 0 &  \mbox{ in } 
(0, T) \times (0, L), \\[6pt]
y_1(\cdot, 0) = y_{1, x} (\cdot, L) = 0 & \mbox{ in }  (0, T), \\[6pt]
y_{1}(\cdot, L) = u_1 & \mbox{ in } (0, T), \\[6pt]
y_1(0, \cdot) = 0 & \mbox{ in } (0, L), 
\end{array}\right.
\end{equation}
\begin{equation} \label{eq:second_order}\left\{
\begin{array}{cl}
y_{2, t} + y_{2, x}  + y_{2, xxx}  + y_1  y_{1, x}   = 0 &
\mbox{ in } (0, T) \times (0, L), \\[6pt]
y_2(\cdot, 0) = y_{2}(\cdot, L) =  y_{2, x}(\cdot, L)  = 0 & \mbox{ in }  (0, T), \\[6pt]
y_2(0, \cdot) = 0 & \mbox{ in } (0, L), 
\end{array}\right.
\end{equation}
for some control $u_1$. The key point of this approach is to first understand how one can choose $u_1$ so that 
$$
y_1(T, \cdot) = 0
$$
and then analyse what the behavior of $ \proj_{\cM_D} y_2 (T, \cdot)$ is. 
To this end, we compute the quantity 
\be
\int_0^L y_2(T, x) \Phi(t, x) \, dx, 
\ee
where $\Phi$ is defined in \eqref{thm-MD-Phi}. Multiplying the equation of $y_2$ by $\Phi$, integrating by parts in $(0, T) \times (0, L)$, we obtain, after using the boundary conditions and the initial conditions,   
\be\label{motivation1}
\int_0^L y_2(T, x) \Phi(t, x) \, dx =  \frac{1}{2} \int_0^T \int_0^L y_1^2(t, x)  \Phi_{x}(t, x) \, dx. 
\ee
The goal is then to understand the value of the RHS of \eqref{motivation1}. 

We will study the value of the RHS of \eqref{motivation1} in a more general setting.  Motivated by the definition of $y_1$,  we consider the unique solution  $y \in X_T$ of the system,  for $u \in H^{1/3}(0, T)$,  
\begin{equation}\left\{
\begin{array}{cl}
y_{t}  + y_{x}  + y_{xxx}  = 0 &  \mbox{ in } 
(0, T) \times (0, L), \\[6pt]
y(\cdot, 0) = y_{x} (\cdot, L) = 0 & \mbox{ in }  (0, T), \\[6pt]
y(\cdot, L) = u & \mbox{ in } (0, T),  \\[6pt]
y(0, \cdot) = 0 & \mbox{ in } (0, L). 
\end{array}\right.
\end{equation}
Guided  by the definition of $\phi$, as suggested in \cite{CKN20,Ng-Decay21},  for $\eta_1, \, \eta_2, \,  \eta_3 \in \mC \setminus \{0 \}$,  we set
\begin{equation}\label{def-vp}
\varphi(x) = \sum_{j=1}^3 (\eta_{j+1} - \eta_{j}) e^{\eta_{j+2} x} \mbox{ for } x \in [0, L],
\end{equation}
with the convention $\eta_{j+3} = \eta_j$ for $j \ge 1$. The following assumptions on $\eta_j$ are used repeatedly  throughout this section:
\begin{equation}\label{pro-eta0}
\eta_{1} + \eta_{2} + \eta_{3} = 0, \quad \eta_{1} \eta_{2} + \eta_{1} \eta_{3} + \eta_{2} \eta_{3} = 1, 
\end{equation}
and
\begin{equation}\label{pro-eta1}
\eta_1 e^{-\eta_1  L} = \eta_2 e^{-\eta_2 L} = \eta_3 e^{-\eta_3 L}. 
\end{equation}

Extend $y$ and $u$ by $0$ for $t > T$ and still denote these extensions by $y$ and $u$, respectively. Then, by \Cref{lemH1/3} in the appendix, 
$$
\| u\|_{H^{1/3}(\mR_+)} \le C \| u \|_{H^{1/3}(0, T)}. 
$$
Assume that 
$$
y(T, \cdot) =0.
$$
Then the extension $y \in C([0, + \infty); L^2(0, L)) \cap L^2((0, + \infty); H^1(0, L))$ is also a solution of the linearized KdV system in $[0, +\infty) \times (0, L)$ using the control which is  the extension of $u$ (by $0$ outside $(0, T)$), i.e.,  
\begin{equation}\label{sys-y-dir}\left\{
\begin{array}{cl}
y_t  + y_x  + y_{xxx}  = 0 &  \mbox{ in } (0, +\infty) \times (0, L), \\[6pt]
y(\cdot, 0) = y_x(\cdot, L) = 0 & \mbox{ in }  (0, +\infty), \\[6pt]
y(\cdot, L) = u & \mbox{ in } (0, +\infty), \\[6pt]
y(0, \cdot) =  0 &  \mbox{ in } (0, L). 
\end{array}\right.
\end{equation}
In what follows in this section, we study this quantity, for $p \in \mR$: 
\be\label{motivation1-1}
\int_0^T \int_0^L y^2(t, x)  \varphi_x(x) e^{p t} \, dt \, dx \quad \left(=  \int_0^{+\infty} \int_0^L y^2(t, x)  \varphi_x(x) e^{p t} \, dt \, dx \right). 
\ee

We have, by \Cref{lem-form-sol} (see also \Cref{rem-form-sol}), for $z \in \mC$ outside  a discrete
set,
\begin{equation}\label{def-y}
\hat y(z, x) = \hat u(z)\frac{\sum_{j=1}^3 \big(\lambda_{j}e^{\lambda_{j} L } - \lambda_{j+1} e^{\lambda_{j+1} L } \big) e^{\lambda_{j+2} x} }{\sum_{j=1}^3 (\lambda_{j+1} - \lambda_{j}) e^{-\lambda_{j+2} L }}, 
\end{equation}
where   $\lambda_j = \lambda_j(z)$ for $j=1, \, 2, \, 3$  are determined by \eqref{eq-lambda}.

\medskip 
We begin with

\begin{lemma}\label{lem-1} Let $p \in \mR$,  $\eta_1, \, \eta_2, \,  \eta_3 \in \mC \setminus \{0 \}$, and let $\varphi$ be defined by \eqref{def-vp}.  Set, for $(z, x) \in \mR \times (0, L)$,  
\begin{equation*}
B_D(z, x) = \left| \frac{\sum_{j=1}^3 \big(\lambda_{j}e^{\lambda_{j} L } - \lambda_{j+1} e^{\lambda_{j+1} L } \big) e^{\lambda_{j+2} x} }{\sum_{j=1}^3 (\lambda_{j+1} - \lambda_{j}) e^{-\lambda_{j+2} L }} \right|^2 \varphi_x(x), 
\end{equation*}
where $\lambda_j = \lambda_j(z)$ with $j=1, 2, 3$ are the three solutions of $\lambda^3 + \lambda + i z = 0$. 
Let $u \in H^{1/3}(0, + \infty)$ with compact support in $[0, + \infty)$ and let $y \in C \big([0, + \infty); L^2(0, L) \big) \cap L^2_{\loc}\big( [0, + \infty); H^1(0, L) \big)$ be the unique solution of \eqref{sys-y-dir}.  Assume that $y(t, \cdot) =0$ for large $t$.  We have 
\begin{equation}\label{identity-1}
\int_0^L \int_{0}^{+ \infty} |y(t, x)|^2 \varphi_x(x) e^{p t} \, dt \, dx  = \int_{\mR}  |\hu(z + i p/ 2)|^2 \int_0^L B_D(z + i p/2, x) \, dx \, dz . 
\end{equation}
\end{lemma}

\begin{proof}  The conclusion is a direct consequence of Parseval's theorem and \eqref{def-y}. 
\end{proof}

We next investigate the behavior of  
$$
\int_0^L B_D(zn + i p/2, x) \, dx 
$$
for $z \in \mR$ with  large $|z|$.  We have

\begin{lemma} \label{lem-B}  Let   $p \in \mR$  and $\eta_1, \eta_2, \eta_3 \in \mC \setminus \{0 \}$.  Assume \eqref{pro-eta1}.    We have  
\begin{equation}
 \int_0^L B_D(z + i p /2, x) \, dx  = E_D |z|^{-1/3} + O(|z|^{-2/3}) \mbox{ for $z\in \mR$ with large $|z|$}, 
\end{equation}
where $E_D = E_D (\eta_1, \eta_2, \eta_3)$ is  defined by 
\begin{equation}\label{def-D}
E_D = \frac{1}{\sqrt{3} A}   \sum_{j=1}^3 \eta_{j+2}^2(\eta_{j+1} - \eta_{j}), 
\end{equation}
with  \footnote{$A$ does not depend on $j$ by \eqref{pro-eta1}.}
$$
A = A(\eta_1, \eta_2, \eta_3) : =  \eta_j e^{-\eta_j L}.  
$$
\end{lemma}

Here and in what follows, for $s \in \mR$, $O(z^s)$ denotes a quantity bounded by $C z^s$ for large positive $z$. Similar convention is used for $O(|z|^s)$ for $z \in \mC$.  

\begin{proof}  We first consider the case where $z$ is positive and large.  
We use the following convention  $\Re(\lambda_1) < \Re(\lambda_2) < \Re(\lambda_3)$.

We first look at the denominator of $B_D(z + i p/2, x)$. We have, by \Cref{lem-lambda}, at $(z + ip/2, x)$, 
\begin{multline}\label{den-B}
\frac{1}{\sum_{j=1}^3 (\lambda_{j+1} - \lambda_{j}) e^{-\lambda_{j+2} L }} \cdot  \frac{1}{ \sum_{j=1}^3 (\tlambda_{j+1} - \tlambda_{j}) e^{-\tlambda_{j+2} L }} \\[6pt]
= \frac{e^{\lambda_1 L } e^{\tlambda_1 L}}{(\lambda_3 - \lambda_2)  (\tlambda_3 - \tlambda_2)  } \Big( 1 + O \big(e^{-C |z|^{1/3}} \big) \Big). 
\end{multline}

We next deal with the numerator of $B_D(z + i p/2, x)$. Set, for $(z, x) \in \mR \times (0, L)$,  
\begin{equation}\label{def-fg}
f(z, x) = \sum_{j=1}^3(\lambda_j e^{\lambda_{j} L } - \lambda_{j+1}e^{\lambda_{j+1} L })e^{\lambda_{j+2} x}, \quad g (z, x) = \sum_{j=1}^3(\tlambda_j e^{\tlambda_{j} L } - \tlambda_{j+1} e^{\tlambda_{j+1} L })e^{\tlambda_{j+2} x}, 
\end{equation}
and \footnote{The index $m$ stands the main part.} 
$$
f_m(z, x) =   \lambda_3 e^{\lambda_3 L} e^{\lambda_2 x} - \lambda_2 e^{\lambda_2 L}  e^{\lambda_3 x} - \lambda_3 e^{\lambda_3 L } e^{\lambda_1 x}, \quad 
g_m (z, x) =   \tlambda_3 e^{\tlambda_3 L} e^{\tlambda_2 x} - \tlambda_2 e^{\tlambda_2 L}  e^{\tlambda_3 x} - \tlambda_3 e^{\tlambda_3 L } e^{\tlambda_1 x}. 
$$
We have 
\begin{multline*}
 \int_0^L f(z + i p/2, x) g(z + i p/2, x) \varphi_x(x) \, dx \\[6pt]
  = \int_0^L f_m(z + i p/2, x) g_m(z + i p/2, x) \varphi_x(x) \, dx + \int_0^L (f- f_m)(z + i p/2, x) g_m(z + i p/2, x) \varphi_x(x) \, dx \\[6pt]  + \int_0^L f_m(z + i p/2, x) (g -g_m) (z + i p/2, x) \varphi_x(x) \, dx +  \int_0^L (f- f_m)(z + i p/2, x) (g -g_m) (z + i p/2, x) \varphi_x(x) \, dx. 
\end{multline*}

By \Cref{lem-lambda}, we have  
\begin{multline}\label{B-p0}
 \int_0^L |(f- f_m)(z + i p/2, x) g_m(z + i p/2, x) \varphi_x(x)| \, dx \\[6pt]
 + \int_0^L |(f- f_m)(z+ i p/2, x) (g -g_m) (z+ i p/2, x) \varphi_x(x)| \, dx  \\[6pt]  + \int_0^L |f_m(z+ i p/2, x) (g -g_m) (z+ i p/2, x) \varphi_x(x)| \, dx \le C |e^{(\lambda_3 + \tlambda_3)L}| e^{- C |z|^{1/3}}. 
\end{multline}

We next estimate
\begin{multline}
\int_0^L f_m(z + i p/2, x) g_m(z + i p/2, x) \varphi_x(x) \\[6pt]
= \int_0^L f_m(z+ i p/2, x) g_m(z+ i p/2, x) \left(  \sum_{j=1}^3 \eta_{j+2} (\eta_{j+1} - \eta_{j}) e^{\eta_{j+2} x}\right) \, dx. 
\end{multline}

We first have,  by \Cref{lem-lambda}, at $z + i p/2$, 
\begin{multline}\label{B-p1}
\int_0^L \Big(  - \lambda_3 e^{\lambda_3 L} e^{\lambda_2 x}  \tlambda_2 e^{\tlambda_2 L}  e^{\tlambda_3 x} - \lambda_2 e^{\lambda_2 L}  e^{\lambda_3 x} \tlambda_3 e^{\tlambda_3 L} e^{\tlambda_2 x}
+ \lambda_2 e^{\lambda_2 L}  e^{\lambda_3 x} \tlambda_2 e^{\tlambda_2 L}  e^{\tlambda_3 x} \Big)   \\[6pt]
\times  \left( \sum_{j=1}^3 \eta_{j+2} (\eta_{j+1} - \eta_{j}) e^{\eta_{j+2} x} \right) \, dx
\mathop{=}^{\eqref{pro-eta1}}  A^{-1} e^{ (\lambda_3 +   \tlambda_3 + \lambda_2 +  \tlambda_2) L } \Big(S_1(z + i p/2) + O\big(e^{-C|z|^{1/3}} \big) \Big), 
\end{multline}
where 
\begin{equation}\label{def-T1}
S_1(z) : =  \sum_{j=1}^3  \eta_{j+2}^2 (\eta_{j+1} - \eta_{j}) \left( \frac{\lambda_2 \tlambda_2}{  \lambda_3   +   \tlambda_3   + \eta_{j+2}} - \frac{\lambda_2 \tlambda_3}{ \lambda_3  +   \tlambda_2   + \eta_{j+2}} - \frac{\lambda_3 \tlambda_2}{ \lambda_2  + \tlambda_3  + \eta_{j+2}}  \right)(z). 
\end{equation}

 We next obtain, by \Cref{lem-lambda}, at $z + i p/2$, 
\begin{multline}\label{B-p2}
\int_0^L \Big( \lambda_3 e^{\lambda_3 L } e^{\lambda_1 x}\tlambda_3 e^{\tlambda_3 L } e^{\tlambda_1 x}  
- \lambda_3 e^{\lambda_3 L } e^{\lambda_1 x} \tlambda_3 e^{\tlambda_3 L} e^{\tlambda_2 x} - \lambda_3 e^{\lambda_3 L} e^{\lambda_2 x} \tlambda_3 e^{\tlambda_3 L } e^{\tlambda_1 x} \Big)\\[6pt]
\times   \left(  \sum_{j=1}^3 \eta_{j+2} (\eta_{j+1} - \eta_{j}) e^{\eta_{j+2} x} \right) \, dx  =  e^{(\lambda_3 + \tlambda_3) L }\Big( S_2(z + i p/2) + O(e^{-C|z|^{1/3}}) \Big), 
\end{multline}
where 
\begin{equation}\label{def-T2}
S_2 (z): = \sum_{j=1}^3  \eta_{j+2} (\eta_{j+1} - \eta_{j}) \left( - \frac{\lambda_3 \tlambda_3}{\lambda_1   +  \tlambda_1  + \eta_{j+2}} + \frac{\lambda_3 \tlambda_3}{\lambda_1   +  \tlambda_2   + \eta_{j+2}} + \frac{\lambda_3 \tlambda_3}{ \lambda_2  +   \tlambda_1   + \eta_{j+2}}\right)(z). 
\end{equation}

We  have, at $z + i p/2$, 
\begin{equation}\label{B-p3}
\int_0^L \lambda_3 e^{\lambda_3 L} e^{\lambda_2 x} \tlambda_3 e^{\tlambda_3 L} e^{\tlambda_2 x}  \left( \sum_{j=1}^3  \eta_{j+2} (\eta_{j+1} - \eta_{j}) e^{\eta_{j+2} x} \right) \, dx 
=  e^{(\lambda_3 + \tlambda_3) L } S_3(z + ip/2), 
\end{equation}
where 
\begin{equation}\label{def-T3}
S_3 (z): = \sum_{j=1}^3 \frac{ \eta_{j+2} (\eta_{j+1} - \eta_{j}) \lambda_3 \tlambda_3 \Big( e^{\lambda_2 L  + \tlambda_2 L  + \eta_{j+2} L } - 1\Big) }{\lambda_2   + \tlambda_2   + \eta_{j+2}}(z).  
\end{equation}

We finally get, by \Cref{lem-lambda},  at $z + ip/2$, 
\begin{multline}\label{B-p3-1}
\left| \int_0^L \Big( \lambda_3  e^{\lambda_3 L } e^{\lambda_1 x} \tlambda_2 e^{\tlambda_2 L} e^{\tlambda_3 x} + \lambda_2 e^{\lambda_2 L} e^{\lambda_3 x} \tlambda_3 e^{\tlambda_3 L} e^{\tlambda_1 x}  \Big) \Big(\sum_{j=1}^3  \eta_{j+2} (\eta_{j+1} - \eta_{j}) e^{\eta_{j+2} x}  \Big) \, dx \right| \\[6pt]
= |e^{(\lambda_3 + \tlambda_3)L} | O(e^{-C z^{1/3}}). 
\end{multline}

By \Cref{lem-lambda}, we have, at $z + ip/2$,  
\begin{equation}\label{B-p3-2}
\left\{\begin{array}{c}
\lambda_1 + \tlambda_1 + \lambda_2 + \tlambda_2 + \lambda_3 + \tlambda_3 = O(z^{-1/3}),  \\[6pt]
\lambda_1 + \tlambda_1 + \lambda_3 + \tlambda_3 = O(z^{-1/3}), \\[6pt]
(\lambda_3 - \lambda_2)(\tlambda_3 - \tlambda_2) = 3 z^{2/3} ( 1 + O(z^{-1/3}) ). 
\end{array}\right.
\end{equation}

We claim that 
\begin{equation}\label{claim-T}
|z|^{-1/3} |S_1(z + i p/2)| + |S_2(z+ i p/2)|  + 
 |S_3(z+ i p/2)|  = O(1) \mbox{ for large positive $z$}. 
\end{equation}
Admitting this, by  combining \eqref{den-B}, \eqref{B-p1}, \eqref{B-p2}, \eqref{B-p3}, \eqref{B-p3-1}, and \eqref{B-p3-2}, and using \eqref{claim-T}, we obtain 
\begin{equation}\label{B-p4}
\int_{0}^L B_D(z + i p/2, x) \, d x  =  \frac{A^{-1} S_1(z+ i p/2)}{3 z^{2/3}} + O(|z|^{-2/3}). 
\end{equation}

We first derive the  the asymptotic behavior of $S_1(z+ i p/2)$. We have, by \Cref{lem-lambda},  at $z + i p/2$, 
\be
\lambda_2 \tlambda_2 = z^{2/3} + O(1), \quad \lambda_2 \tlambda_3 = z^{2/3} e^{i \pi/3} + O(1), \quad \lambda_3 \tlambda_2 = z^{2/3} e^{-i \pi/3} + O(1), 
\ee
and 
\begin{multline}
\frac{1}{\lambda_3 + \tlambda_3 + \eta_{j+2}} = \frac{1 + O(z^{-1/3})}{\sqrt{3}z^{1/3}}, \quad \frac{1}{\lambda_2 + \tlambda_3 + \eta_{j+2}} = \frac{1 + O(z^{-1/3})}{(e^{i \pi/6} + i)z^{1/3}}, \\[6pt]
\quad \frac{1}{\lambda_3 + \tlambda_2 + \eta_{j+2}} = \frac{1 + O(z^{-1/3})}{(e^{- i \pi/6} - i)z^{1/3}}. 
\end{multline}
It follows that 
\begin{multline}\label{B-S1}
S_1(z + i p/ 2) 
= \left(\frac{1}{\sqrt{3}} + 2 \Re \frac{e^{i \pi /3}}{e^{ i \pi /6} + i} \right)\sum_{j=1}^3 \eta_{j+2}^2 (\eta_{j+1} - \eta_j) |z|^{1/3} + O(1) \\[6pt]
= \sqrt{3} \sum_{j=1}^3 \eta_{j+2}^2 (\eta_{j+1} - \eta_j) |z|^{1/3} + O(1). 
\end{multline}

We next deal with $S_2(z + i p/2)$. Since 
$$
 \sum_{j=1}^3  \eta_{j+2} (\eta_{j+1} - \eta_{j})  = 0, 
$$
it follows that 
\begin{multline*}
S_2(z+ i p/2) = \sum_{j=1}^3  \eta_{j+2} (\eta_{j+1} - \eta_{j}) \left( - \frac{\lambda_3 \tlambda_3}{\lambda_1   +  \tlambda_1  + \eta_{j+2}} + \frac{\lambda_3 \tlambda_3}{\lambda_1   +  \tlambda_1} \right)(z+ i p/2) \\[6pt]
+ \sum_{j=1}^3  \eta_{j+2} (\eta_{j+1} - \eta_{j}) \left(  \frac{\lambda_3 \tlambda_3}{\lambda_1   +  \tlambda_2   + \eta_{j+2}} -  \frac{\lambda_3 \tlambda_3}{\lambda_1   +  \tlambda_2  }\right)(z+ i p/2) \\[6pt]
+ \sum_{j=1}^3  \eta_{j+2} (\eta_{j+1} - \eta_{j}) \left( \frac{\lambda_3 \tlambda_3}{ \lambda_2  +   \tlambda_1   + \eta_{j+2}} - \frac{\lambda_3 \tlambda_3}{ \lambda_2  +   \tlambda_1}\right)(z+ i p/2). 
\end{multline*}
Using \Cref{lem-lambda}, we derive that  
\begin{equation}
S_2(z + i p/2) = O(1). 
\end{equation}

We next derive the asymptotic behavior of $S_3(z + i p/2)$. From \Cref{lem-lambda}, we have, at $z + i p/2$, 
\be\label{T3-cc}
\lambda_2 + \tlambda_2 = O( z^{-2/3}).
\ee
Using  \eqref{pro-eta1}, we derive from \eqref{def-T3} that 
\begin{multline*}
S_3 (z + i p/2) = z^{2/3} \sum_{j=1}^3 \frac{ (\eta_{j+1} - \eta_{j})  \eta_{j+2}  }{\eta_{j+2}} \Big(A^{-1} \eta_{j+2}  - 1\Big) + O(1) \\[6pt]
=  z^{2/3} \sum_{j=1}^3 (\eta_{j+1} - \eta_{j})   \Big(A^{-1} \eta_{j+2}  - 1\Big) + O(1).  
\end{multline*}
Since  $ \sum_{j=1}^3 (\eta_{j+1} - \eta_{j})= 0 $ and $\sum_{j=1}^3 (\eta_{j+1} - \eta_{j})  \eta_{j+2}  =0 $,  it follows  that  
\begin{equation}\label{T3-final}
S_3 (z + i p/2)=   O(1). 
\end{equation}
 
Combing \eqref{B-p4} and  \eqref{B-S1} yields  
\begin{equation*}
\int_{0}^L B_D(z + i p/2, x) \, dz  = E_D |z|^{-1/3}   + O(z^{-2/3}), 
\end{equation*}
where 
\begin{equation}
E_D = \frac{1}{\sqrt{3}A } \sum_{j=1}^3 \eta_{j+2}^2(\eta_{j+1} - \eta_{j}). 
\end{equation}

\medskip 
The conclusion in the case where $z$ is large and negative can be derived from the case where $z$ is positive and large as follows. 
Define, for $(z, x) \in \mR \times (0, L)$, with large $|z|$,  
$$
M_D(z, x) =  \frac{\sum_{j=1}^3(\lambda_j e^{\lambda_{j} L } - \lambda_{j+1} e^{\lambda_{j+1} L })e^{\lambda_{j+2} x} }{\sum_{j=1}^3 (\lambda_{j+1} - \lambda_{j}) e^{-\lambda_{j+2} L }}.   
$$
Then 
$$
B_D(z + ip/2, x) = |M_D(z + i p/2, x)|^2 \varphi_x(x). 
$$
It is clear from the definition of $M_D$ that
$$
M_D(-z, x) = \overline{M_D(\bar z, x)}.  
$$
We then have 
$$
B_D(-z + ip/2, x) = |M_D(-z + i p/2, x)|^2 \varphi_x(x)= |M_D(z - ip/2, x)|^2 \varphi_x(x). 
$$
We thus obtain the result in the case where $z$ is negative and large by the corresponding expression for large positive $z$ in which $p$  is replaced by $-p$.  The conclusion follows. 
\end{proof}

As a consequence of  \Cref{lem-1,lem-B}, we obtain

\begin{proposition}\label{pro-dir} Let  $L \in \cN_D$.  Let $u \in H^{1/3}(0, + \infty)$ with compact support in $[0, +\infty)$,  and let $y \in C([0, + \infty); L^2(0, L))
\cap L^2_{\loc}\big([0, +\infty); H^1(0, L) \big)$ be the unique solution of
\eqref{sys-y-dir}. Assume that $y(t, \cdot) =0$ for large $t$. We have
\begin{equation}
\int_{0}^{+\infty} \int_0^L  |y(t, x)|^2 \Phi_x(t, x) \diff x \diff t  =
\int_{\mR}  |\hu(z + i q/2)|^2  \int_0^L B(z + i q/2, x) \, dx \diff z, 
\end{equation}
where 
$$
B(z, x) = \left| \frac{\sum_{j=1}^3 \big(\lambda_{j}e^{\lambda_{j} L } - \lambda_{j+1} e^{\lambda_{j+1} L } \big) e^{\lambda_{j+2} x} }{\sum_{j=1}^3 (\lambda_{j+1} - \lambda_{j}) e^{-\lambda_{j+2} L }} \right|^2 \phi_x(x), 
$$
with $\lambda_j = \lambda_j(z)$ ($j=1, 2, 3$) being the three solutions of $\lambda^3 + \lambda + i z = 0$. Moreover, 
\be\label{pro-dir-B}
\int_0^L B(z + i q/2, x) \, dx = E |z|^{-1/3}  + O(|z|^{-2/3}), 
\ee
where 
$$
E = - \frac{b e^{2 a}}{2 \sqrt{3} a L^2} (b^2 + 9a^2) > 0. 
$$
\end{proposition}

Recall that $q$, $a$, $b$, $\phi$, and $\Phi$ are defined in \Cref{thm-MD}. 

\begin{proof} We will apply \Cref{lem-1,lem-B} with  $p = q$, 
$$ 
\eta_1 = \alpha + i \beta, \quad \eta_2 = \alpha - i \beta, \quad \mbox{ and } \quad \eta_3 = - 2 \alpha, 
$$
where $\alpha$ and $\beta$ are defined in \eqref{thm-MD-alphabeta}.  Thus 
\be \label{def-AAA}
A  = \eta_3 e^{- \eta_3 L} = - 2 \alpha e^{2 \alpha L} = \frac{2 a}{L} e^{-2 a}. 
\ee

We have 
\begin{align*}
\sqrt{3} A E_D = &  \eta_1^2 (\eta_3 - \eta_2) + \eta_2^2 (\eta_1 - \eta_3) + \eta_3^2 (\eta_2 - \eta_1) \\[6pt]
= &  (\alpha + i \beta)^2 (- 3 \alpha + i \beta) + (\alpha - i \beta)^2 (3 \alpha + i \beta) - 8 i  \alpha^2  \beta \\[6pt]
= & - 12 i  \alpha^2 \beta + 2 i \beta (\alpha^2 - \beta^2) - 8 i \alpha^2 \beta. 
\end{align*}
This implies 
\be
\sqrt{3} A E_D = -2 i (\beta^2 + 9 \alpha^2) \beta =  \frac{2 i b}{L^3} (b^2 + 9 a^2). 
\ee
We  obtain 
$$
E_D =  \frac{ i b e^{2 a}}{\sqrt{3} a L^2} (b^2 + 9 a^2). 
$$
Since, by \Cref{lem-Phi-Psi},  
$$
\varphi = - 2 i \phi, 
$$
the conclusion now follows  from \Cref{lem-1,lem-B}.  
\end{proof}

Using \Cref{pro-dir}, we derive the following result which is the key
ingredient for the analysis of the local controllability of the KdV system \eqref{sys-KdV} in small time.  

\begin{proposition} \label{pro-monotone}  Let $L \in \cN_D$ and $T < 1$.   Let $u \in H^{1/3}(0, + \infty)$ and let $y \in C([0, + \infty);
L^2(0, L)) \cap L^2_{\loc}\big([0, +\infty); H^1(0, L) \big)$ be the unique solution
of \eqref{sys-y-dir}.  Assume that $u(t) = 0$ for $t
> T$, and  $y(t, \cdot) = 0$ for large $t$. Then
\begin{equation}\label{pro-monotone-dir}
\int_{0}^{\infty} \int_0^L |y (t, x)|^2 \Phi_x(t, x) \diff x \diff t =
E \| u e^{q \cdot /2}\|_{[H^{1/6}(0, + \infty)]^*}^2  \Big(1 + O (T^{2/9}) \Big). 
\end{equation}
\end{proposition}

\begin{proof}
We have, by \Cref{pro-dir} and \Cref{lem-Bp},   
\begin{equation*}
|\hu (z + i q/2)|^2  \left| \int_0^L B(z + i q/2, x) \, d x \right|    \le   \frac{C |\hat u (z + i q/2)|^2}{1 + |z|^{1/3}}. 
\end{equation*}
Applying \Cref{pro-dir} again and using  the fact $\int_{\mR} \, dz = \int_{z \in \mR; |z| < m} \, dz + \int_{z \in \mR; |z| \ge  m} \, dz$,  we
derive  that
\begin{multline}\label{pro-monotone-p0}
\left|\int_{\mR} |\hu (z + i q/ 2)|^2 \int_0^L B(z + i q/ 2, x) \diff x \diff z  - E 
\int_{\mR} \frac{|\hu (z + i q/ 2)|^2}{(1 + |z|^{2})^{1/6}} \diff z  \right| \\[6pt]
\le  C \int_{|z| <  m} \frac{|\hu (z + i q/ 2)|^2}{(1 + |z|^{2})^{1/6}}  + C m^{-1/3}
\int_{|z| > m} \frac{|\hu (z + i q/ 2)|^2}{(1 + |z|^{2})^{1/6}} \diff z.  
\end{multline}
Applying \Cref{lem-Fourier1} below with $s = 1/6$, we have
\be \label{pro-monotone-p1}
|\hu (z + i q/ 2)| \le C ( |z| T^{4/3} + T^{1/3} ) \| u e^{q \cdot/2}\|_{[H^{1/6}(0, + \infty)]^*}. 
\ee
Combining \eqref{pro-monotone-p0} and  \eqref{pro-monotone-p1} yields 
\begin{multline}
\left|\int_{\mR} |\hu (z + i q/ 2)|^2 \int_0^L B(z + i q/ 2, x) \diff x \diff z  - E 
\int_{\mR} \frac{|\hu (z + i q/ 2)|^2}{(1 + |z|^{2})^{1/6}} \diff z  \right| \\[6pt]
 \le C \| u e^{q \cdot/2} \|_{[H^{1/6}(0, + \infty)]^*}^2 (T^{8/3} m^{8/3} + T^{2/3} m^{2/3}+ m^{- \frac{1}{3}}) . 
\end{multline}
By choosing $m = T^{- 2/3}$, we obtain 
$$
\left|\int_{\mR} |\hu (z + i q/ 2)|^2 \int_0^L B(z + i q/ 2, x) \diff x \diff z  - E 
\int_{\mR} \frac{|\hu (z + i q/ 2)|^2}{(1 + |z|^{2})^{1/6}} \diff z  \right| \le C \| u e^{q \cdot/2} \|_{H^{-1/6}(0, + \infty)}^2 T^{2/9}. 
$$
The conclusion follows. 
\end{proof}

In the proof of \Cref{pro-monotone}, we have used the following lemma.

\begin{lemma} \label{lem-Fourier1} Let $0< T < 1$, $0 < s < 1/2$ and let $v \in [H^{s} (0, + \infty)]^*$ with $\supp v \subset [0, T]$. Extend $v$ by $0$ for $t < 0$ and still denote this extension by $v$. There exists a positive constant $C$ independent of $T$ and $v$ such that, for $z \in \mR$,
\be
| \hv (z) | \le C (|z| T^{\frac{3}{2} - s} + T^{\frac{1}{2} - s} ) \| v\|_{[H^{s}(0, +\infty)]^*}. 
\ee
\end{lemma}

\begin{proof} Fix $\chi \in C^1(\mR)$ such that $\chi = 1$ in $[0, T]$, $\chi = 0$ for $t \ge 2 T$,  and $0 \le \chi \le 1$ and  $\chi' \le C/T$ in $\mR$.  Then 
\be \label{lem-Fourier-p1}
|\hv (z)| \le \| v \|_{[H^{s}(0, + \infty)]^*} \|e^{ - i z \cdot } \chi \|_{H^{s}(0, + \infty)}. 
\ee
We have 
\begin{multline} \label{lem-Fourier1-p1}
C \|e^{ - i z \cdot } \chi \|_{H^{s}(0, + \infty)}^2 
\le \int_{0}^{4T} \int_{0}^{4T} 
\frac{|\chi(x)|^2 |e^{ - i z x } - e^{- i z y}|^2}{|x-y|^{1 + 2 s}} \, dx \, dy \\[6pt]
+ \int_{0}^{4T}\int_{0}^{4T}  \frac{|\chi(x) - \chi(y)|^2}{|x-y|^{1 + 2 s}} \, dx \, dy + \int_{0}^{2T} \int_{4T}^\infty \frac{1}{|x-y|^{1 + 2s}} \, dx \, dy. 
\end{multline}
Using the fact, for $x \in \mR$,  
$$
|e^{ i x} - 1|^2 =  4 \sin^2 (x/2) \le x^2, 
$$
we derive that the RHS of \eqref{lem-Fourier1-p1} is bounded by 
\be \label{lem-Fourier-p2}
C z^2 T^{3 - 2s} + C T^{1-2s} + C T^{1- 2 s}. 
\ee
The conclusion now follows from \eqref{lem-Fourier-p1} and \eqref{lem-Fourier-p2}. 
\end{proof}

\section{Small time local controllability properties  of the KdV system - Proof of \Cref{thm1}}\label{sect-thm1}

The main result of this section is the following, which implies, in particular, \Cref{thm1} after applying \Cref{pro-monotone} and \Cref{lem-multiplier} at the end of this section. This result is also the key ingredient in the proof of Assertion i) of \Cref{thm2}. 

\begin{theorem}\label{thm-NL} Let $L \in \cN_D$, $T>0$,  and $\alpha > 0$. Assume that 
\be \label{thm-NL-coercivity}
\int_0^{T} \int_0^L |\xi(t, x)|^2 \Phi(t, x) \, dx \, dt  \ge \alpha \| v \|_{[H^{1/6}(0, T)]^*}^2 
\ee
for all $v \in H^{1/3}(0, T)$ such that $\xi(T, \cdot) = 0$ in $(0, L)$ where $\xi \in X_T$ is the unique solution of the system 
\begin{equation}\left\{
\begin{array}{cl}
\xi_{t}  + \xi_{x}  + \xi_{xxx}  = 0 &  \mbox{ in } 
(0, T) \times (0, L), \\[6pt]
\xi(\cdot, 0) = \xi_{x} (\cdot, L) = 0 & \mbox{ in }  (0, T), \\[6pt]
\xi(\cdot, L) = v & \mbox{ in } (0, T),  \\[6pt]
\xi(0, \cdot) = 0 & \mbox{ in } (0, L). 
\end{array}\right.
\end{equation}
There exists $\eps_0 > 0$ depending only on $\alpha$, $T$, and $L$ such that for all $0< \eps <
\eps_0$,   and for all  solutions $y \in X_{T/2}$  of the system 
\begin{equation}\label{thm-NL-y}\left\{
\begin{array}{cl}
y_t  + y_x  + y_{xxx}  + y y_x  = 0 &  \mbox{ in } (0, T/2)
\times  (0, L), \\[6pt]
y(\cdot, 0) = y_x(\cdot, L) = 0 & \mbox{ in }  (0, T/2), \\[6pt]
y(\cdot,L) = u & \mbox{ in } (0, T/2),  \\[6pt]
y(0, \cdot)  = \eps \phi  & \mbox{ in } (0, L),
\end{array}\right.
\end{equation}
with $\| u \|_{H^{1/2}(0, T/2)} < \eps_0$, we have
\[
y(T/2, \cdot) \neq 0.
\]
\end{theorem}

Recall that $\phi$ is defined in \eqref{thm-MD-phi}. 

\begin{proof}    Let
$\eps_0$ be a small positive constant, which depends only on $\alpha$, $T$,  and $L$,   and is determined
later.  We prove \Cref{thm-NL} by contradiction.    Assume that there exists a solution $y
\in C\big([0, + \infty); L^2(0, L) \big) \cap L^2_{\loc}\big([0, + \infty); H^1(0, L)
\big) $ of
\eqref{thm-NL-y}  with $y(t, \cdot) =0$ for $t \ge T/2$,  for some $0< \eps < \eps_0$, for some $u \in H^{1/3}(0, +
\infty)$  with   $\supp u \subset [0, T/2]$ and $\| u
\|_{H^{1/2}(0, T)} < \eps_0$.

Using the fact $y(t, \cdot) = 0$ for $t \ge T/2$,  from \Cref{pro-WP}, we have, for $\eps_0$ small, 
\begin{equation}\label{thm-NL-yyy}
\| y\|_{L^2 \big(\mR_+; H^1(0, L) \big)} \le C \Big( \| y_0 \|_{L^2(0, L)} + \| u
\|_{H^{1/3}(\mR_+)} \Big), 
\end{equation}
which in turn implies,  by \Cref{pro-kdv1} and \Cref{lem-interpolationL1L2} below, 
\begin{equation}\label{thm-NL-yyy-1}
\| y\|_{L^2 \big( \mR_+ \times (0, L) \big)} \le C \Big( \| y_0 \|_{L^2(0, L)} + \| u
\|_{L^2(\mR_+)} \Big).
\end{equation}
Here and in what follows, $C$ denotes a positive constant depending only on $T$ and $L$ ($C$ thus does not depend on $\alpha$).

Let $y_1$ and  $y_2$ be the solution of the following systems
\begin{equation} \label{thm-NL-y1}\left\{
\begin{array}{cl}
y_{1, t}  + y_{1, x}  + y_{1, xxx}  = 0 &  \mbox{ in } 
\mR_+ \times (0, L), \\[6pt]
y_1( \cdot, 0) =   y_{1, x}(\cdot, L) = 0 & \mbox{ in }  \mR_+,\\[6pt]
y_{1}(\cdot, L) = u& \mbox{ in }  \mR_+,\\[6pt]
y_{1}(0 , \cdot) = 0  & \mbox{ in }  (0, L), 
\end{array}\right.
\end{equation}
\begin{equation}\label{thm-NL-y2}\left\{
\begin{array}{cl}
y_{2, t}  + y_{2, x}  + y_{2, xxx}  + y_{1} y_{1, x}= 0 &  \mbox{ in } 
\mR_+ \times (0, L), \\[6pt]
y_2( \cdot, 0) =   y_{2, x}(\cdot, L) = y_{2}(\cdot, L)= 0 & \mbox{ in }  \mR_+,\\[6pt]
y_{2}(0 , \cdot) = 0 & \mbox{ in }  (0, L). 
\end{array}\right.
\end{equation}

Then, by \Cref{pro-kdv1},   
\be \label{thm-NL-y1-1}
\|y_1\|_{X_T} \le C \| u\|_{H^{1/3}(\mR_+)} 
\ee
and
\be\label{thm-NL-y1-2}
\| y_1\|_{L^2 \big((0, T) \times (0, L) \big)} \le C    \| u\|_{L^2(\mR_+)},  
\ee
which in turn imply, by \Cref{pro-kdv1} and \Cref{lem-interpolationL1L2} below, 
\be \label{thm-NL-y2-1}
\|y_2\|_{X_T} \le C \| y_{1} y_{1, x} \|_{L^1((0, T); L^2(0, L))} \le C \| u\|_{L^2(\mR_+)}^{1/2} \| u\|_{H^{1/3}(\mR_+)}^{3/2}, 
\ee
and
\be\label{thm-NL-y2-2}
\| y_2\|_{L^2 \big((0, T) \times (0, L) \big)} \le C \| y_1^2 \|_{L^1((0, T); L^2(0, L))}   \le C   \| u\|_{L^2(\mR_+)}^{3/2} \| u\|_{H^{1/3}(\mR_+)}^{1/2}. 
\ee

Set 
$$
\dy = y - y_1 - y_2 \mbox{ in } \mR_+ \times (0, L).  
$$
We have 
\begin{equation*}\left\{
\begin{array}{cl}
\dy_t  + \dy_x  + \dy_{xxx}  + y \dy_x + \dy (y_1 + y_2)_x  + (y_1 y_2)_x + y_2 y_{2, x}  = 0&  \mbox{ in } 
\mR_+ \times (0, L), \\[6pt]
\dy(\cdot, 0) =  \dy_{x} (\cdot, L) = \dy(\cdot, L)  = 0 & \mbox{ in } \mR_+, \\[6pt]
\dy(0, \cdot) = \eps \phi & \mbox{ in } (0, L).
\end{array}\right.
\end{equation*}
Applying   \Cref{pro-kdv1} and using \Cref{lem-interpolationL1L2} below,  we derive  that
\begin{multline}\label{thm-NL-dy-1}
\|\dy\|_{X_T} \le C \Big( \| y_{1} y_{2, x} \|_{L^1((0, T); L^2(0, L))} + \| y_{2} y_{1, x} \|_{L^1((0, T); L^2(0, L))}  + \| y_{2} y_{2, x} \|_{L^1((0, T); L^2(0, L))}  \Big)\\[6pt]
  \mathop{\le}^{\eqref{thm-NL-y1-1}-\eqref{thm-NL-y2-2}} C  \| u\|_{L^2(\mR_+)} \| u\|_{H^{1/3}(\mR_+)}^{2} + C \eps,
\end{multline}
and 
\begin{multline}\label{thm-NL-dy-2}
\|\dy\|_{L^2\big( (0, T) \times (0, L) \big)} \le C \Big( \|y \dy_x \|_{L^1((0, T); L^2(0, L))} + \|(y_1 + y_2) \dy_x \|_{L^1((0, T); L^2(0, L))}  \Big) \\[6pt]
+ C \Big(\|(y_1 + y_2) \dy \|_{L^1((0, T); L^2(0, L))} +  \| y_1 y_2 \|_{L^1((0, T); L^2(0, L))} +  \|y_2^2\|_{L^1((0, T); L^2(0, L))} \Big) \\[6pt]\mathop{\le}^{\eqref{thm-NL-yyy}-\eqref{thm-NL-yyy-1}, \eqref{thm-NL-y1-1}-\eqref{thm-NL-dy-1}}  C  \| u\|_{L^2(\mR_+)}^{3/2} \| u\|_{H^{1/3}(\mR_+)}^{5/2} + C \| u\|_{L^2(\mR_+)}^{2} \| u\|_{H^{1/3}(\mR_+)} 
+ C \eps 
 \\[6pt]
\le C    \| u\|_{L^2(\mR_+)}^{3/2} \| u\|_{H^{1/3}(\mR_+)}^{3/2} + C \eps. 
\end{multline}
In \eqref{thm-NL-dy-1}, we absorbed the contribution in the RHS of 
$$
\| y \dy_x\|_{L^1((0, T); L^2(0, L))} + \|\dy (y_1 + y_2)_x \|_{L^1((0, T); L^2(0, L))},
$$ 
which is bounded by $C \eps_0 \| \dy \|_{X_T}$. Combining  \eqref{thm-NL-y2-2} and \eqref{thm-NL-dy-2} yields
\be\label{thm-NL-y1-0}
\|y - y_1 \|_{L^2\big((0, T) \times (0, L) \big)} \le C    \| u\|_{L^2(\mR_+)}^{3/2} \| u\|_{H^{1/3}(\mR_+)}^{1/2} + C \eps. 
\ee
Since $y =0$ for $t \ge T/2$ and $u = 0$ for $t \ge T/2$, after using the regularizing effect of the linear KdV equation, 
and considering the projection into $\cM_D^\perp$,  we derive that
\be\label{thm-NL-y1}
\|y_1 (T, \cdot) \|_{L^2(0, L)} \le C    \| u\|_{L^2(\mR_+)}^{3/2} \| u\|_{H^{1/3}(\mR_+)}^{1/2} + C \eps. 
\ee

Since $y_1(T, \cdot) \in \cM_{D}^\perp$ by \Cref{lem-projection}, it follows from  \Cref{pro-C} that  there exists  $u_1 \in H^{1/3}(0, T)$ such that 
\be\label{thm-NL-yyy1}
\| u_1 \|_{H^{1/3}(0, T)} \le C \| y_1(T, \cdot) \|_{L^2(0, L)}
\ee
and  the solution $\ty_1 \in X_{T} $ of the  system 
\begin{equation*}\left\{
\begin{array}{cl}
\ty_{1, t}  + \ty_{1, x}  + \ty_{1, xxx}   = 0 &
\mbox{ in } (0, T) \times (0, L), \\[6pt]
\ty_1(\cdot, 0) = \ty_{1, x}(\cdot, L)  = 0 & \mbox{ in }  (0, T), \\[6pt]
\ty_{1}(\cdot, L) =u_1 & \mbox{ in } (0, T), \\[6pt]
\ty_1(0 , \cdot) = 0 & \mbox{ in } (0, L),
\end{array}\right.
\end{equation*}
satisfies 
$$
\ty_1(T, \cdot) = - y_1(T, \cdot).
$$
Using  \eqref{thm-NL-y1}, we derive from \eqref{thm-NL-yyy1} that 
\be \label{thm-NL-u1}
\| u_1 \|_{H^{1/3}(0, T)} \le C   \| u\|_{L^2(\mR_+)}^{3/2} \| u\|_{H^{1/3}(\mR_+)}^{1/2} + C \eps, 
\ee
which in turn implies, by \Cref{pro-kdv1}, 
\be \label{thm-NL-ty1}
\| \ty_1 \|_{X_T} \le C   \| u\|_{L^2(\mR_+)}^{3/2} \| u\|_{H^{1/3}(\mR_+)}^{1/2} + C \eps. 
\ee

Set 
$$
\ty = y_1 + \ty_1 \mbox{ in } (0, T) \times (0, L). 
$$
Then 
\begin{equation}\label{thm-NL-ty}\left\{
\begin{array}{cl}
\ty_{t}  + \ty_{x}  + \ty_{xxx}   = 0 &
\mbox{ in } (0, T) \times (0, L), \\[6pt]
\ty(\cdot, 0) = \ty_{x}(\cdot, L)  = 0 & \mbox{ in }  (0, T), \\[6pt]
\ty(\cdot, L) = u + u_1 & \mbox{ in } (0, T), \\[6pt]
\ty(0 , \cdot) =  \ty(T, \cdot) =0  & \mbox{ in } (0, L). 
\end{array}\right.
\end{equation}
We have 
\begin{multline}\label{thm-NL-yhy}
\| y - \ty \|_{L^2\big( (0, T) \times (0, L)\big)} \le \| y - y_1 \|_{L^2\big( (0, T) \times (0, L)\big)} + \| \ty_1 \|_{L^2\big( (0, T) \times (0, L)\big)} \\[6pt]
\mathop{\le}^{\eqref{thm-NL-y1-0}, \eqref{thm-NL-ty1}} C   \| u\|_{L^2(\mR_+)}^{3/2} \| u\|_{H^{1/3}(\mR_+)}^{1/2} + C \eps.
\end{multline}

Multiplying the equation of $y$ with $\Phi(t, x)$, integrating by parts on $[0, L]$, and
using \eqref{sys-Phi}, we have
\begin{equation}\label{thm-NL-B}
\frac{d}{dt} \int_{0}^L y (t, x) \Phi(t, x)\diff x - \frac{1}{2} \int_0^L y^2 (t, x)
\Phi_x (t, x) \diff x = 0.
\end{equation}
Integrating \eqref{thm-NL-B} from 0 to $T$ and using the fact $y(T, \cdot) = 0$ yield
\begin{equation}\label{thm-NL-id}
\int_{0}^L y_0 (x)  \Phi(0, x) \diff x  + \frac{1}{2} \int_0^{T} \int_0^L y^2 (t, x)
\Phi_x(t, x) \diff x \diff t  = 0.
\end{equation}
We have
\begin{multline}\label{thm-NL-id1}
 \|  y^2 - \ty^2 \|_{L^1\big( (0, T) \times (0, L) \big)} \le  \|  y - \ty\|_{L^2\big( (0, T) \times (0, L) \big)}\|  y + \ty\|_{L^2\big( (0, T) \times (0, L) \big)}\\[6pt] \mathop{\le}^{\eqref{thm-NL-yyy-1}, \eqref{thm-NL-yhy}} C   \| u\|_{L^2(\mR_+)}^{5/2} \| u\|_{H^{1/3}(\mR_+)}^{1/2} + C \eps_0 \eps.
\end{multline}
Combining   \eqref{thm-NL-id} and \eqref{thm-NL-id1} yields 
\be\label{thm-NL-key} 
\int_{0}^L y_0 (x)  \Phi(0, x) \diff x   + \frac{1}{2} \int_0^{T} \int_0^L \ty^2 (t, x)
\Phi_x(t, x) \diff x \diff t \\[6pt]  \le C   \| u\|_{L^2(\mR_+)}^{5/2} \| u\|_{H^{1/3}(\mR_+)}^{1/2} + C \eps_0 \eps. 
\ee
Applying \eqref{thm-NL-coercivity} to $\ty$  after noting \eqref{thm-NL-ty}, we derive from \eqref{thm-NL-key} that, for $\eps_0$ sufficiently small,   
$$
\eps \|\phi \|_{L^2(0, L)}^2 + \frac{1}{2} \alpha \| u + u_1 \|_{[H^{1/6}(\mR_+)]^*}^2 \le C   \| u\|_{L^2(\mR_+)}^{5/2} \| u\|_{H^{1/3}(\mR_+)}^{1/2} + C \eps_0 \eps. 
$$
Using \eqref{thm-NL-u1}, it follows that, for  sufficiently small $\eps_0$,  
$$
\alpha \| u  \|_{[H^{1/6}(\mR_+)]^*}^2 \le C \| u
\|_{L^2(\mR_+)}^{5/2}  \| u
\|_{H^{1/3}(\mR_+)}^{1/2}.
$$
We derive from \Cref{lemH1/3} that   
$$
\alpha \| u  \|_{[H^{1/6}(0, T/2)]^*}^2 \le C\| u
\|_{L^2(0, T/2)}^{5/2}  \| u
\|_{H^{1/3}(0, T/2)}^{1/2}.
$$

Note that 
$$
\| u\|_{L^2(0, T/2)} \le C \| u\|_{[H^{1/6}(0, T/2)]^*}^{3/4} \| u\|_{H^{1/2}(0, T/2)}^{1/4}
$$
and 
$$
\| u\|_{H^{1/3}(0, T/2)} \le C \| u\|_{[H^{1/6}(0, T/2)]^*}^{1/4} \| u\|_{H^{1/2}(0, T/2)}^{3/4}.
$$
This yields 
$$
\| u\|_{L^2(0, T/2)}^{5/2}\| u\|_{H^{1/3}(0, T/2)}^{1/2} \le C  \| u\|_{[H^{1/6}(0, T/2)]^*}^2 \| u \|_{H^{1/2}(0, T/2)}.  
$$

So, for fixed sufficiently small $\eps_0$,
\[
u =0. 
\]
Hence $y(t, \cdot) = \eps \Psi(t, \cdot) \not \equiv 0$ for all $t > 0$. We have a contradiction.

\medskip 
The proof is complete.
\end{proof}

In the proof of \Cref{thm-NL}, we repeatedly used the following useful result. 
\begin{lemma} \label{lem-interpolationL1L2} Let $L > 0$ and $T > 0$ and let  $f \in L^2 \big( (0, T); H^1(0, L) \big)$ and $g \in L^2\big((0, T) \times (0, L) \big)$. Then 
\be\label{lem-interpolationL1L2-st1}
\| f g \|_{L^1\big((0, T); L^2(0, L)\big)} \le C \| f\|_{L^2 \big((0, T) \times (0, L) \big)}^{\frac{1}{2}}  \| f\|_{L^2 \big((0, T); H^1 (0, L) \big)}^{\frac{1}{2}}  \| g\|_{L^2 \big((0, T) \times (0, L) \big)} 
\ee
and 
\begin{multline}\label{lem-interpolationL1L2-st11}
\| (fg)_x\|_{L^1((0, T); [H^1(0, L)]^*)} \\[6pt]
\le C \Big(\| f g \|_{L^1((0, T); L^2(0, T))} + \| (f g)(\cdot, 0) \|_{L^1(0, T)} + \| (f g)(\cdot, L) \|_{L^1(0, T)} \Big),  
\end{multline}
for some positive constant $C$ depending only on $L$. Consequently, for $f, g \in L^2 \big( (0, T); H^1(0, L) \big)$, 
\be\label{lem-interpolationL1L2-st2}
\| f g_x \|_{L^1\big((0, T); L^2(0, L)\big)} \le C \| f\|_{L^2 \big((0, T) \times (0, L) \big)}^{\frac{1}{2}}  \| f\|_{L^2 \big((0, T); H^1 (0, L) \big)}^{\frac{1}{2}}  \| g_x\|_{L^2 \big((0, T) \times (0, L) \big)}.
\ee
\end{lemma}

\begin{proof}  We first establish \eqref{lem-interpolationL1L2-st1}. We have  
\begin{multline} \label{lem-interpolationL1L2-p1}
\| f g \|_{L^1 \big( (0, T);  L_2 (0, L) \big)}^2 = \int_{0}^T \left(\int_0^L |f g|^2 (s, x) \, dx  \right)^{\frac{1}{2}} \, ds   \\[6pt] \le   \int_{0}^T \sup_{x \in [0, L]} |f(s, x)| \left(\int_0^L |g|^2 (s, x) \, dx  \right)^{\frac{1}{2}} \, ds. 
\end{multline}
Using the fact 
\begin{multline}\label{lem-interpolationL1L2-coucou}
\sup_{x \in [0, L]} |f(s, x)| \le C \left( \int_0^L |f(s, x)|^2 \, dx + \int_0^L |f(s, x)| |f_{x}(s, x)|^2 \, d x  \right)^{\frac{1}{2}} \\[6pt]
\le C \| f(s, \cdot)\|_{L^2(0, L)}^{\frac{1}{2}} \| f(s, \cdot)\|_{H^1(0, L)}^{\frac{1}{2}} , 
\end{multline}
we derive from \eqref{lem-interpolationL1L2-p1} that 
\begin{multline*}
\| f g_x \|_{L^1 \big( (0, T);  L_2 (0, L) \big)}^2 \le C \int_0^T \| f(s, \cdot)\|_{L^2(0, L)}^{\frac{1}{2}}  
\| f(s, \cdot)\|_{H^1(0, L)}^{\frac{1}{2}} \| g(s, \cdot)\|_{L^2(0, L)} \, ds 
\\[6pt]
\le C \| f\|_{L^2 \big((0, T) \times (0, L) \big)}^{\frac{1}{2}}  \| f\|_{L^2 \big((0, T); H^1 (0, L) \big)}^{\frac{1}{2}}  \| g\|_{L^2 \big((0, T) \times (0, L) \big)} \quad (\mbox{by H\"older's inequality}), 
\end{multline*}
which is \eqref{lem-interpolationL1L2-st1}.

We next derive \eqref{lem-interpolationL1L2-st11}.  We have, for $0 \le t \le T$,  
$$
\int_{0}^L (fg)_x(t, x) \varphi(x)\, dx = -  \int_{0}^L f g(t, x) \varphi_x(x)\, dx + (fg)(t, L) \varphi(L) - (fg)(t, 0) \varphi(0) \mbox{ for } \varphi \in H^1(0, L). 
$$
This implies, by the trace theory applied to $\varphi \in H^1(0, L)$, for $0 \le t \le T$, 
$$
\| (fg)_x(t, \cdot) \|_{[H^1(0, L)]^*} \le C \Big( \| (fg) (t, \cdot) \|_{L^2(0, L)} +   \| (f g)(t, 0) \|_{L^1(0, T)} + \| (f g)(t, L) \|_{L^1(0, T)} \Big). 
$$
Estimate \eqref{lem-interpolationL1L2-st11} follows. 

\medskip
The proof is complete.
\end{proof}

To derive \Cref{thm1} from \Cref{thm-NL}, we also need to use the following simple lemma. 
\begin{lemma} \label{lem-multiplier} Let $T> 0$, $0 < s <1$,  $\varphi \in C^1([0, +\infty])$. Then 
$$
\| \varphi v\|_{[H^{s}(\mR_+)]^*} \le C \| v\|_{[H^{s}(\mR_+)]^*} \mbox{ for } v \in [H^{s} (\mR_+)]^* \mbox{ with } \supp v \subset [0, T],  
$$
where $C$ is a positive constant independent of $v$.
\end{lemma}

\begin{proof} Without loss of generality, one might assume that $\supp \varphi \subset [0, 2T]$. Using the duality between $[H^{s}(\mR_+)]^*$ and $H^s(\mR_+)$, it suffices  to prove that 
$$
\| \varphi V\|_{H^{s}(\mR_+)} \le C \| V\|_{H^{s}(\mR_+)} \mbox{ for } V \in H^{s} (\mR_+),  
$$
for some positive constant $C$ independent of $V$. This follows immediately from the Gagliardo-Nirenberg characterization of the norm of $H^s(\mR_+)$. 
\end{proof}

\section{Local, exact controllability of the KdV equation in a positive time - Proof of \Cref{thm2}} \label{sect-thm2}

This section consists of two subsections. In the first one, we give the proof of Assertion $i)$. The proof of Assertion $ii)$ is given in the second subsection. 

\subsection{Proof of Assertion $i)$ of  \Cref{thm2}}  \label{proof-thm2-1}

We first consider the case $\omega =0$. We have
\begin{multline*}
\left|\sum_{j=1}^3 \big(\lambda_{j}e^{\lambda_{j} L } - \lambda_{j+1} e^{\lambda_{j+1} L } \big) e^{\lambda_{j+2} x} \right|^2 (z) \\[6pt]
=  \sum_{j=1}^3 \big(\lambda_{j}e^{\lambda_{j} L } - \lambda_{j+1} e^{\lambda_{j+1} L } e^{\lambda_{j+2} x} \big) (z)  \sum_{j=1}^3 \big(\lambda_{j}e^{\lambda_{j} L } - \lambda_{j+1} e^{\lambda_{j+1} L } e^{\lambda_{j+2} x} \big) (- \bar z)
\end{multline*}
and similarly, 
$$
|\Xi (z)|^2 = \Xi (z) \Xi (- \bar z), 
$$
where $\Xi$ is defined in \eqref{def-Xi}. 
It follows that 
$$
|\Xi(\cdot + i q/2)|^2|\Omega (\cdot + i q/2)|^2 \mbox{ is an analytic function in $z \in \mR$}.  
$$
 There thus exists a finite set $\{z_j; 1 \le j \le  m \} \subset \mR$ such that 
$$
\Omega(z_j + i q / 2) = 0 \mbox{ for } 1 \le j \le m \quad \mbox{ and } \quad \Omega(z + i q/2) \neq 0 \mbox{ for } z \in \mR \setminus \{z_j; 1 \le j \le  m \}. 
$$
Since $\Omega (\cdot + i q/2)$ is continuous in  $\mR$ and $\Omega (z + i q/2) \to + \infty$ as $|z| \to + \infty$ by \eqref{pro-dir-B} of \Cref{pro-dir} and \Cref{lem-lambda}, we derive that for each $\delta > 0$ (small),  there exists $\rho > 0$ (depending  on $\delta$) such that 
$$
\Omega (z + i q /2) \ge \rho \mbox{ in } J_\delta  : =  \mR \setminus \cup_{j=1}^m (z_j - \delta, z_j + \delta). 
$$

Let $T>0$ be arbitrary and let $y$ be a solution of \eqref{sys-y-dir} such that 
$y(t, \cdot) = 0$ for $t \ge T$.  Applying \Cref{pro-dir} and taking into account the fact $\omega \ge 0$, we get
\be \int_0^{+\infty}\int_{0}^L y^2(t, x) \Phi_x(t, x) \diff x \diff t  = 
\int_{\mR} \int_{0}^L  |\hu(z + i q / 2)|^2 B(z + i q/2, x) \, dx \, dz, 
\ee
where we extend $u$ by $0$ for $t < 0$ and $t > T$. This implies 
\begin{multline} \int_0^{+\infty}\int_{0}^L y^2(t, x) \Phi_x(t, x) \diff x \diff t \\[6pt]
=  \int_{J_\delta } \int_{0}^L  |\hu(z + i q / 2)|^2 B(z + i q/2, x) \, dx \, dz + \int_{\mR \setminus J_\delta} \int_{0}^L  |\hu(z + i q / 2)|^2 B(z + i q/2, x) \, dx \, dz \\[6pt]
\ge C_1 \rho \| u e^{q \cdot /2} \|_{[H^{1/6}(0, T)]^*}^2 - C_2   \rho \int_{J_\delta} |\hu(z + i q/2)|^2  \, dz \\[6pt]
\mathop{\ge}^{\Cref{lem-multiplier}}  C_1 \rho  \| u
\|_{[H^{1/6}(0, T)]^*}^2 - C_T \rho  \delta \| u
\|_{[H^{1/6}(0, T)]^*}^2. 
\end{multline}

By choosing $\delta$ sufficiently small, we arrive that 
\be \label{thm2-key-i}
\int_{0}^\infty \int_0^{+\infty} y^2(t, x) \Phi_x(t, x) \diff x \diff t  \ge C \rho  \| u
\|_{[H^{1/6}(0, T)]^*}^2. 
\ee

Applying \Cref{thm-NL}, we derive that system \eqref{sys-KdV} is not locally null controllable at time $T/2$. Since $T$ is arbitrary, the conclusion follows. 

We next deal with the case $\omega > 0$. It is clear that \eqref{thm2-key-i} holds with $\rho = \omega$. The conclusion follows by  \Cref{thm-NL} as before since $T$ is arbitrary. \qed

\subsection{Proof of Assertion $ii)$ of \Cref{thm2}} \label{sec-Assertion-ii}

Set 
\be
\sigma(t) = e^{-\frac{1}{1 - t^2}} \mathds{1}_{(-1, 1)}, 
\ee
where $\mathds{1}_{(-1, 1)}$ denotes the characteristic function of the interval $(-1, 1)$.  Then $\sigma \in C^\infty_c(\mR)$ and 
\be \label{thm2-p2-sigma}
|\hat \sigma (z)| \le c_1 e^{- c_2 \sqrt{|z|}} \mbox{ for } z \in \mR, 
\ee
see e.g., \cite{TT07}.  Let $m \in \N$ be large and denote 
$$
\sigma_m (t) = \frac{1}{m}\sigma(t/m) \mbox{ for } t \in \mR. 
$$ 

Let $z_0 \in \mR$ be such that 
$$
\Omega(z_0 + i q/2) = \omega < 0. 
$$
Consider $u_m$ defined by   
\be \label{thm2-p2-um}
\hat u_m(z) = e^{i m z} \hat \sigma_m(z - i q / 2 - z_0) \det Q(z) \Xi (z). 
\ee
Applying Paley-Wiener's theorem, see e.g., \cite{Rudin-RC},  and using \eqref{thm2-p2-sigma}, one can check that the function $v_m$ which is defined by  $\hat v_m (z) = \hat \sigma_m(z - i q / 2 - z_0) \det Q(z) \Xi (z)$ belongs to  $L^2(\mR)$ with support in $[-m, m]$. We then derive that $u_m \in L^2(\mR)$ with support in $[0, 2m]$ since $u_m(t) = v_m(t + m)$ for $t \in \mR$.  Using \eqref{thm2-p2-sigma}, one can also derive that $u_m \in C^\infty(\mR)$.

Let $y = y_m$ be the unique solution of \eqref{sys-y} with $u = u_m$. From the definition of $u_m$, we deduce that 
$$
\supp y \subset [0, m] \times [0, L], 
$$
since, by \Cref{lem-form-sol}, 
$$
\hat y_{xx} (z, 0) = \frac{\hat u_m (z) G(z)}{H(z)}
$$
is an entire function with modulus is bounded by $C e^{(m+\eps) |z|}$ for all $\eps > 0$, and $\hat y_{xx} (\cdot, 0) \in L^2(\mR)$. It follows that 
$$
\supp y_{xx}(\cdot, 0) \subset [0, m]. 
$$
Since $y(m, x) \in \cM_D^\perp$, applying \Cref{lem-Phi}, we derive that 
$$
y(t, x) = 0 \mbox{ for } (t, x) \in [m, + \infty) \times [0, L]. 
$$
We thus can apply \Cref{pro-dir} to $y$. 

Since, by \Cref{pro-dir}, 
$$
\int_{0}^\infty \int_0^{+\infty} y^2(t, x) \Phi_x(t, x) \diff x \diff t =  \int_{\mR} \frac{|\hat u_m(z + i q/2)|^2 }{|H(z + i q/2)|^2} \Omega(z + i q/2) \, dz,    
$$
it follows from \eqref{thm2-p2-um} that 
\be\label{thm2-p2-I1}
\int_{0}^\infty \int_0^{+\infty} y^2(t, x) \Phi_x(t, x) \diff x \diff t  = e^{- m q} \int_{\mR}  |\hat \sigma_m (z - z_0)|^2 |\Xi (z + i q/2)|^2  \Omega(z + i q/2) \, dz. 
\ee

We have, by \eqref{thm2-p2-sigma} 
\be\label{thm2-p2-sigmam-1} 
|\hat \sigma_m (z)| \le c_1 e^{- c_2 \sqrt{m |z|}} \mbox{ for } z \in \mR.  
\ee
\be\label{thm2-p2-sigmam-2} 
|\hat \sigma_m (z)| \ge c_3 \mbox{ for } z \in \mR \mbox{ with }  |z| < c_4/ m, 
\ee
\be \label{thm2-p2-Xi1} 
\inf_{ z \in \mR} |\Xi (z + i q/2)| > 0, \quad |\Xi (z + i q/2)| \le C (|z| + 1) \mbox{ for } z \in \mR, 
\ee
\be \label{thm2-p2-Omega1}
|\Omega (z + i q/2)| \le c_5 e^{c_6 |z|^{1/3}},  
\ee
for some positive constant $c_1$, \dots, $c_6$, independent of $m$. 

Using \eqref{thm2-p2-sigmam-1}, \eqref{thm2-p2-sigmam-2}, \eqref{thm2-p2-Xi1}, \eqref{thm2-p2-Omega1}, and the fact $\Omega(z_0 + i q/2) = \omega < 0$,  it follows from \eqref{thm2-p2-I1} that 
$$
\int_{0}^\infty \int_0^{+\infty} y^2(t, x) \Phi_x(t, x) \diff x \diff t  < 0,  
$$
for large $m$. Fix such an $m$ and set 
$$
T = 2m. 
$$

We thus have just proved,  after scaling, that  there exists $U_1 \in H^{1/3}(0, T)$ such that  
\be\label{thm2-Y1Y2}
Y_1 (T, \cdot ) = 0 \quad \mbox{ and } \quad \proj_{\cM_D} Y_2(T, \cdot) = - \phi, 
\ee
where $Y_1, Y_2 \in X_T$ are the unique solution of the systems 
\begin{equation}\label{thm2-Y1}\left\{
\begin{array}{cl}
Y_{1, t}  + Y_{1, x}  + Y_{1, xxx}  = 0 &  \mbox{ in } 
(0, T) \times (0, L), \\[6pt]
Y_1(\cdot, 0) = Y_{1, x} (\cdot, L) = 0 & \mbox{ in }  (0, T), \\[6pt]
Y_{1}(\cdot, L) = U_1 & \mbox{ in } (0, T),\\[6pt]
Y_1(0, \cdot) = 0 & \mbox{ in } (0, L), 
\end{array}\right.
\end{equation}
\begin{equation} \label{thm2-Y2}\left\{
\begin{array}{cl}
Y_{2, t} + Y_{2, x}  + Y_{2, xxx}  + Y_1  Y_{1, x}   = 0 &
\mbox{ in } (0, T) \times (0, L), \\[6pt]
Y_2(\cdot, 0) = Y_{2, x}(\cdot, L)  = 0 & \mbox{ in }  (0, T), \\[6pt]
Y_{2}(\cdot, L) = 0 & \mbox{ in } (0, T). 
\end{array}\right.
\end{equation}

We now establish the local controllability for the time $T$ which is $2m$.   In what follows, we consider $T = 2m$. 
Fix $y_0, y_T \in L^2(0, L)$ with small $L^2(0, L)$-norms. We first consider the case 
$$
\rho = \| y_0\|_{L^2(0, L)} > 0, 
$$
\begin{equation}\label{CP-rem1}
\int_0^L y_0 \phi_1 \, dx \ge   {\bf c}_1 \rho \int_0^L \phi^2 \, dx , 
\end{equation}
and 
\begin{equation}\label{CP-rem2}
\|y_T \|_{L^2(0, L)} \le \frac{{\bf c}_1 e^{qT}  \rho}{10}.    
\end{equation}
for some fixed constant ${\bf c}_1$ independent of $\rho$.

For $r > 0$ and $y_0 \in L^2(0, L)$, denote $B_{r} (y_0)$ the open ball centered at 0 and of radius $r$ in $L^2(0, L)$. We also denote   $\overline{B_{r} (y_0)}$ its closure in $L^2(0, L)$.

Let $c$ be a small positive constant determined later. 
For $\varphi \in \overline{B_{c \rho}(y_0)}$, let   $u_1$ and $u_2$ be  controls in $H^{1/3}(0, T)$ for which the solutions $y_1$ and $y_2$ in $X_T$ of the systems 
\begin{equation}\label{Sys-y1-thm2}\left\{
\begin{array}{cl}
y_{1, t}  + y_{1, x}  + y_{1, xxx}  = 0 &  \mbox{ in } 
(0, T) \times (0, L), \\[6pt]
y_1(\cdot, 0) = y_1(\cdot, L) = 0 & \mbox{ in } (0, T), \\[6pt]
y_{1, x}(\cdot,L) = u_1 & \mbox{ in } (0, T),\\[6pt]
y_{1}(0, \cdot) = 0 & \mbox{ in } (0, L),
\end{array}\right.
\end{equation}
\begin{equation}\label{Sys-y2-thm2} \left\{
\begin{array}{cl}
y_{2, t}  + y_{2, x}  + y_{2, xxx} + y_1  y_{1, x}   = 0 &
\mbox{ in } (0, T) \times (0, L), \\[6pt]
y_2(\cdot, 0) = y_2(\cdot, L)  = 0 & \mbox{ in }  (0, T), \\[6pt]
y_{2, x}(\cdot, L) = u_2 & \mbox{ in } (0, T), \\[6pt]
y_{2}(0 , \cdot) = \proj_{\M_D} \varphi / \rho & \mbox{ in } (0, L),
\end{array}\right.
\end{equation}
satisfy
\[
y_1(T, \cdot) = 0 \quad \mbox{ and } \quad y_2(T, \cdot) = y_T/ \rho.
\]
Moreover, one can choose $u_1$ and  $u_2$ as  Lipschitz functions of $\proj_{\M_D} \varphi / \rho$ with the Lipschitz constants bounded by positive constants independent of $\rho$. For example, using \Cref{lem-projection} and \eqref{thm2-Y1Y2}, one can take \footnote{It is useful to note that $e^{-q T} \alpha_0 - \alpha_T > 0$ by \eqref{CP-rem1} and \eqref{CP-rem2} if $c$ is sufficiently small.}
$$
u_1 = (e^{-q T} \alpha_0 - \alpha_T) U_1,
$$
where $\alpha_0 \phi = \proj_{\cM_D} \varphi / \rho$ and $\alpha_T \phi  = \proj_{\cM_D} y_T/\rho$, and 
$$
u_2 = \cL \big(- \proj_{\cM_D^\perp}y_3(T, \cdot) + \proj_{\cM_D^\perp} y_T/ \rho \big) + \hat \cL(\proj_{\cM_D} \varphi/ \rho),
$$
where $y_3 \in X_T$ is the unique solution of the system 
\begin{equation} \left\{
\begin{array}{cl}
y_{3, t}  + y_{3, x}  + y_{3, xxx} + y_1  y_{1, x}   = 0 &
\mbox{ in } (0, T) \times (0, L), \\[6pt]
y_3(\cdot, 0) = y_3(\cdot, L)  = y_{3, x}(\cdot, L)= 0 & \mbox{ in }  (0, T), \\[6pt]
y_{3}(0 , \cdot) = \proj_{\M_D^\perp} \varphi / \rho & \mbox{ in } (0, T).
\end{array}\right.
\end{equation}
Here $\cL$ and $\hat \cL$ are the operators given in \Cref{pro-C}.

For $\varphi \in \overline{B_{c \rho}(y_0)}$, let  $y  \in X_T$ be the unique solution of the {\it backward} linear  KdV system 
\begin{equation}\label{sys-KdV-B}\left\{
\begin{array}{cl}
y_t  + y_x  + y_{xxx}  + y  y_x  = 0 &  \mbox{ in } (0, T) \times (0, L), \\[6pt]
y(\cdot, 0) = y_x(\cdot, L) = 0 & \mbox{ in }   (0, T), \\[6pt]
y_x(\cdot, 0) = v & \mbox{ in }  (0, T), \\[6pt]
y(T, \cdot)  = y_T  &  \mbox{ in }  (0, L). 
\end{array}\right.
\end{equation}
where $v (t) = \rho^{1/2} y_{1, x}(t, 0) + \rho y_{2, x}(t, 0)$ with $y_1$ and $y_2$ being defined by \eqref{Sys-y1-thm2} and \eqref{Sys-y2-thm2}. 

Note that $y \in X_T$ is well-defined if $\rho$ is sufficiently small by \Cref{pro-kdv3-NL}. We will denote 
\be
H_0(\varphi) = y(0, \cdot) \mbox{ in } (0, L) \mbox{ and }  H(\varphi) = y \mbox{ in } (0, T) \times (0, L).  
\ee 

 Consider the map
\[
\begin{array}{cccc}
\Lambda\colon  &  \overline{B_{c \rho} (y_0)} & \to & L^2(0, L) \\[6pt]
& \varphi & \mapsto & \varphi -  H_0(\varphi) +  y_0. 
\end{array}
\]
We will prove that
\begin{equation}\label{CP-claim1}
\Lambda (\varphi) \in \overline{B_{c \rho} (y_0)},
\end{equation}
and
\begin{equation}\label{CP-claim2}
\|\Lambda (\varphi) - \Lambda (\tvarphi) \|_{L^2(0, L)} \le \lambda \| \varphi - \tvarphi\|_{L^2(0, L)},
\end{equation}
for some $\lambda \in (0, 1)$. Assuming this, one derives from the contraction mapping theorem that there exists a unique $\varphi_0 \in \overline{B_{c \rho} (y_T)} $ such that $\Lambda (\varphi_0) = \varphi_0$.  As a consequence, with $y = H(\varphi_0)$, 
\[
y(0, \cdot) = y_0,
\]
and $y(\cdot, L)$ is hence a required control.

We next establish \eqref{CP-claim1} and \eqref{CP-claim2}. Indeed, assertion \eqref{CP-claim1} follows from the fact
\[
\|\varphi -  H_0 (\varphi) \|_{L^2(0, L)} \le C \| \varphi \|_{L^2(0, L)}^{3/2} \mbox{ for } \varphi \in \overline{B_{c\rho} (y_0)}.
\]
This can be proved using the approximation via the power series method as follows. Applying \Cref{pro-kdv3} to $y $,  we derive from \eqref{sys-KdV-B}  that 
\be
\| y\|_{X_T} \le C \rho^{1/2}. 
\ee

Set \footnote{The index $a$ stands the
approximation.}
\be\label{CP-decomposition1}
y_a = \rho^{1/2} y_1 + \rho y_2 \mbox{ in } (0, T) \times (0, L). 
\ee
We have 
\be \label{CP-y-ya}
(y - y_a)_t +(y - y_a)_x + (y - y_a)_{xxx} + (y-y_a) y_x + 
y_a( y - y_a)_x = h(t, x),  
\ee
where 
$$
h(t, x) = - y_a y_{a, x}  + \rho y_1 y_{1, x} = - \rho^{3/2} (y_1 y_2)_x  - \rho^2 y_2 y_{2, x}. 
$$

Applying  \Cref{pro-kdv3} to $y-y_a$, for small $\rho$, one can ignore (absorb) the contribution from the last two terms in the LHS of \eqref{CP-y-ya}, to obtain 
\begin{equation}\label{CP-yya}
\|y- y_a\|_{X_T} \le C \|  h   \|_{L^{1} \big( (0, T); L^2(0, L) \big)} \le  C \rho^{3/2}.
\end{equation}
Assertion~\eqref{CP-claim1} follows since $H_0 (\varphi) = y(0, \cdot)$ and $\varphi = y_a(0, \cdot)$.

\medskip
We next establish \eqref{CP-claim2}. To this end, we estimate
\[
 \Big( \varphi -  H_0(\varphi) \Big) - \Big( \tvarphi - H_0 (\tvarphi) \Big).
\]
For $\tvarphi \in \overline{B_{c \rho}(y_0)}$, denote $\tu_{1}, \tu_{2}$ and $\ty_{1}, \ty_{2}, \ty_a, \ty$ the corresponding functions which are defined in the same way as the functions  $u_{1}, u_{2}$,  and $y_{1}, y_{2}, y_a, y$ considered for  $\varphi$.

We have
\be\label{CP-yty}
(y - \ty)_{t} + (y - \ty)_{x} + (y - \ty)_{xxx} + y y_x - \ty \ty_x = 0,
\ee
\be \label{CP-yatya}
(y_a - \ty_a)_{t} + (y_a - \ty_a)_{x} + (y_a - \ty_a)_{xxx} + y_a y_{a, x} - \ty_a \ty_{a,x} =  g (t, x),
\ee
where
\be\label{CP-def-g}
 g (t, x) = \rho^{3/2}  \Big( (y_1 y_2)_x - (\ty_1 \ty_2)_x \Big)  + \rho^2 \Big( y_2 y_{2, x} - \ty_2 \ty_{2, x} \Big).
\ee

Applying \Cref{pro-kdv3-NL} to $y - \ty$, for $\rho$ small, one can ignore (absorb) the contribution of the last two terms in the LHS of \eqref{CP-yty}, to obtain 
\begin{equation}\label{CP-pp0}
\| y-\ty \|_{X_T} \le C \| y_{a, x} (\cdot, 0) - \ty_{a, x} (\cdot, 0)  \|_{H^{1/3}(0, T)} \le C \rho^{-1/2} \|\varphi- \tvarphi \|_{L^2(0, L)}. 
\end{equation}
Using  \eqref{CP-yya} for $y - y_a$ and similar fact for $\ty - \ty_a$, we obtain 
\begin{equation}\label{CP-pp1}
\| (y-y_a, \ty - \ty_a) \|_{X_T} \le  C \rho^{3/2}.  
\end{equation}
From the definition of $g$ in  \eqref{CP-def-g}, we deduce that 
\be\label{CP-pp3}
\| g \|_{L^{1}\big((0, T); L^2(0, L) \big)}  \le C \rho^{1/2} \| \varphi - \tvarphi \|_{L^2(0, L)}.
\ee
Applying \Cref{pro-kdv1} to $y_a - \ty_a$ after ignoring (absorbing) the contribution of $y_a y_{a, x} - \ty_a \ty_{a,x} $,  we derive from \eqref{CP-yatya} that 
\begin{equation}\label{CP-pp2}
 \|y_a - \ty_a\|_{X_T} \le C \rho^{-1/2} \|\varphi - \tvarphi \|_{L^2(0, L)} \le C \rho^{1/2}. 
\end{equation}

Set 
$$
Y = y -y_a - (\ty - \ty_a) \mbox{ in } (0, T) \times (0, L). 
$$
Using \eqref{CP-yty} and  \eqref{CP-yatya}, we have 
\be\label{CP-Y}
\partial_t Y  + \partial_x Y + \partial_{xxx} Y =  f(t, x) \mbox{ in } (0, T) \times (0, L), 
\ee
where 
$$
f(t, x) = - g(t, x) - \Big(y y_x - \ty \ty_x - (y_a y_{a, x} - \ty_a \ty_{a,x}) \Big).
$$
From \eqref{CP-pp0},  \eqref{CP-pp1}, \eqref{CP-pp2},  and \eqref{CP-pp3}, we obtain  
\be\label{CP-pp4}
\| y y_x - \ty \ty_x - (y_a y_{a, x} - \ty_a \ty_{a,x}) \|_{L^{1}\big((0, T); L^2(0, L) \big)} \le C \rho^{1/2}  \| \varphi - \tvarphi \|_{L^2(0, L)}.
\ee

Using \eqref{CP-pp3} and \eqref{CP-pp4},  and applying \Cref{pro-kdv3-NL} to $Y$, we derive from \eqref{CP-Y} that 
\[
\| (y - y_a - \ty + \ty_a) (T, \cdot) \|_{L^2(0, L)} \le 
C \rho^{1/2} \| \varphi - \phi \|_{L^2(0, L)}. 
\]
Assertion~\eqref{CP-claim2} follows.

We next consider the general case. One can bring this case to the previous case as follows. 
Set 
$$
\rho = \| y_0 \|_{L^2(0, L)} + \| y_T \|_{L^2(0, L)}. 
$$
Without loss of generality, one might assume that $\rho > 0$ since otherwise, one just takes zero as a control.

Fix $\eps > 0$ small. By \Cref{pro-dir},  there exists  $v_1 \in H^{1/3}(0, \eps)$ such that if $y_1 \in X_\eps$  is the solution of \eqref{Sys-y1-thm2} with $y_{1}(\cdot, L) = v_1$ and $y_2 \in X_\eps$ is the solution of \eqref{Sys-y2-thm2} with $y_2 (\cdot, L) = 0$ then
\[
y_1 (\eps, \cdot) =0,   \quad \mbox{ and } \quad   \int_{0}^{L} y_2(\eps, \cdot) \phi  \ge c. 
\]
Set 
$$
u =  \rho^{1/2} \gamma^{1/2} v_1 \mbox{ in } (0, \eps).  
 $$
Then the unique solution $y \in X_\eps$ of the KdV system \eqref{sys-KdV} verifies the properties of the previous case, for large positive $\gamma$. 
We are now in the position to apply the previous case with the initial datum $y(\eps, \cdot)$. The proof is complete.  \qed

\appendix

\section{Hardy type inequalities}

This section is devoted to the following result related to the Hardy inequality. 

\begin{lemma}\label{lemH1/3} Let $- \infty < a < b < c \le + \infty$, and $0< s < 1/2$ and let $u \in H^{s}(a, b)$. Set 
$$
v = \left\{ \begin{array}{cl} u & \mbox{ in } (a, b), \\[6pt]
0 & \mbox{ in } (b, c). 
\end{array} \right.
$$
Then $v \in H^{s}(a, c)$ and 
$$
\| v \|_{H^{s}(a, c)} \le C \| u\|_{H^{s}(a, b)}, 
$$
for some positive constant $C$ depending only on $a, b, c$, and $s$. 
\end{lemma}

\begin{proof} For notational ease, we assume that $a=-1$ and $b=0$. Without loss of generality, we then can assume that $c = + \infty$. Let $V \in H^s(\mR)$ be an extension of $u$ such that 
$$
\|V\|_{H^s(\mR)} \le C \| u \|_{H^s(-1, 0)}. 
$$
Applying \cite[Theorem 1.1]{Ng-Squa1} with $\gamma = - s$, $\tau =2$, $p=2$ to $V$, one obtains 
$$
\| |x|^{-s}V  \|_{L^2(\mR)} \le C \| V\|_{H^s(\mR)}. 
$$
The condition $s < 1/2$ is required here.   This yields 
\be\label{lemH1/3-p1}
\| |x|^{-s} u  \|_{L^2(-1, 0)} \le C \| u\|_{H^s(-1, 0)}. 
\ee
Using the equivalent Gagliardo-Nirenberg definition of the semi-norm $H^s$, we have 
$$
\| v\|_{H^s(-1, + \infty)}^2 \sim \int_{-1}^\infty \int_{-1}^\infty \frac{|v(x) - v(y)|^2}{|x-y|^{1 + 2 s}} \, d x \, dy + \int_{-1}^\infty |v|^2 \, dx.  
$$
Since
$$
\int_{-1}^\infty \int_{-1}^\infty \frac{|v(x) - v(y)|^2}{|x-y|^{1 + 2 s}} \, d x \, dy \le  \int_{-1}^0 \int_{-1}^0 \frac{|u(x) - u(y)|^2}{|x-y|^{1 + 2 s}} \, d x \, dy + C_s \int_{-1}^0 \frac{|u(x)|^2}{|x|^{2s}} \, dx, 
$$
and 
$$
\int_{-1}^\infty |v|^2 \, dx = \int_{-1}^0 |u|^2 \, dx, 
$$
we derive from \eqref{lemH1/3-p1} that 
$$
\| v\|_{H^s(-1, + \infty)} \le C \| u\|_{H^s(-1, 0)}. 
$$
The proof is complete. 
\end{proof}

\section{On the zeros of $\det_Q$}

We begin this section with 

\begin{lemma}\label{lem-p} Let $L>0$,  $z \in \mC$, and $\varphi \in C^\infty([0, L])$ be such that 
$$
\varphi_{xxx} + \varphi_x + i z \varphi = 0 \mbox{ in } [0, L], 
$$
and 
$$
\varphi(0) = \varphi(L) = \varphi_x (L) = 0. 
$$
Then 
$$
\Im (z) \ge 0. 
$$
\end{lemma}

\begin{proof}  Set 
$$
\Psi (t, x) = \Re \Big\{ e^{i z t} \varphi(x) \Big\} \mbox{ in } \mR_+ \times [0, L]. 
$$
Then $\Psi$ is a solution of the linearized KdV system 
\be
\left\{\begin{array}{cl}
\Psi_t + \Psi_x + \Psi_{xxx} = 0 & \mbox{ in } \mR_+ \times [0, L], \\[6pt]
\Psi(\cdot, 0) = \Psi(\cdot, L) = \Psi_x(\cdot, L) = 0 & \mbox{ in } \mR_+. 
\end{array}\right.
\ee
We then derive that 
$$
\frac{d}{dt} \int_0^L |\Psi(t, x)|^2 \, dx \le 0. 
$$
This implies 
$$
\Im(z) \ge 0. 
$$
The proof is complete. 
\end{proof}

The following result is useful. 
 
%
%

\begin{lemma} \label{lem-Bp} Let $L \in \cN_D$. Then $\det Q (z + i q / 2) \neq 0$  for  $z \in \mR$. 
\end{lemma}

\begin{proof}  We prove the assertion by contradiction. 
Assume that  $\det Q(z  + i q / 2) = 0$ for some $z \in \mR$.   Then  there exists $\varphi \in C^\infty[0, L]$ such that 
$$
\varphi_{xxx} + \varphi_x + i (z + i q/2) \varphi = 0 \mbox{ in } [0, L], 
$$
and 
$$
\varphi(0) = \varphi(L) = \varphi_x (L) = 0. 
$$
Applying \Cref{lem-p}, we have 
$$
\Im  (z + i q/ 2) = q/2 \le 0. 
$$
We have a contradiction. 
\end{proof}

\section{A lemma related to the moment method}

The following lemma from \cite{CKN20} is used in the proof of \Cref{pro-kdv1}. 

\begin{lemma}\label{lem-zeros} Let $\varphi$ be an analytic  function in $\mC$ such that $\varphi$ has a finite number of zeros on the real line, and 
\be\label{lem-zeros-varphi}
|\varphi (z)| \le c_1 e^{c_2 |z|^\alpha} \mbox{ in } \mC,  
\ee
for some $0 < \alpha < 1$, and $c_1, c_2 > 0$. 
Let $T_1, T_2 > 0$, and $h \in H^{s}(\mR)$ for some $s \le 0$  with support in $(0, T_1)$.  There exists $g \in C^\infty(\mR)$ with support in $[T_1, T_1 + T_2]$ such that if $z$ is a real solution  of order $m$ of the equation $\varphi (z) = 0$, then $z$ is a also a real solution of order $m$ of the equation $\hat h - \hat g = 0$, and 
\be
\| g\|_{H^k(\mR)} \le C_k \| h \|_{H^{s}(\mR)} \mbox{ for } k \in \N, 
\ee
for some positive constant $C_k$ depending only on $k$,  $T_1$, $T_2$, $s$,  and real zeros and their multiplicity of $\varphi$. 
\end{lemma}

\begin{proof} The proof of \Cref{lem-zeros} is as in \cite{CKN20},  where a special case is considered.  The construction of $g$, inspired by the moment method, see e.g. \cite{TT07},  can be done as follows.
Set $\eta(t) = e^{-1/ (t^2 - (T_2/2)^2)} \mathds{1}_{|t| < T_2}$ for $t \in \mR$. Assume that
$z_1$, \dots, $z_k$ are real, distinct solutions of the equation $ \varphi (z)  =0$,
and $m_1$, \dots, $m_k$ are the corresponding orders. Set, for $z \in
\mC$,
\[
\zeta(z) = \sum_{i=1}^k   \left(  \hat \eta(z - z_i) \mathop{\prod_{j=1}^k}_{j \neq i} (z
- z_j)^{m_j} \Big( \sum_{l=0}^{m_i} c_{i, l}  (z - z_i)^{l} \Big) \right),
\]
where $c_{i, l} \in \mC$ is chosen such that
\[
\frac{d^{l}}{dz^{l}} \Big( e^{ i (T_1 +T_2/2) z}\zeta(z)  \Big)_{z = z_i}= \frac{d^{l}}{dz^{l}} \hat
h_3 (z_i) \mbox{ for } 0 \le l \le m_i, \; 1 \le i \le k.
\]
This can be done since $\hat \eta (0) \neq 0$. Since
\[
|\hat \eta(z)| \le C e^{T_2 |\Im (z)|/2},
\]
and, by \cite[Lemma 4.3]{TT07},
\[
|\hat \eta(z)| \le C_1 e^{- C_2 |z|^{1/2}} \mbox{ for } z \in \mR,
\]
using \eqref{lem-zeros-varphi}, and 
applying Paley-Wiener's theorem, one can prove that $\zeta$ is the Fourier transform of a
function $\psi$ of class $C^1$; moreover,  $\psi$ has the support in $[-T_2/2, T_2/2]$. Set, for
$z \in \mC$,
\[
g(t) = \psi(t + T_1  + T_2/2).
\]
Using the fact $\hat g (z) = e^{ i  (T_1 + T_2/2) z} \zeta (z)$, one can check that $\hat g -
\hat h$ has zeros $z_1$, \dots, $z_k$ with the corresponding orders $m_1$, \dots,
$m_k$.  One can check that
\[
\| \psi \|_{H^k(\mR)} \le C_{T, L, k } \sum_{i=1}^k \sum_{l=0}^{m_i} \left|\frac{d^{l}}{dz^{l}}
\hat h (z_i)\right|,
\]
which yields
\[
\| \psi \|_{H^k(\mR)} \le C_{T, L, k}  \| h \|_{H^{s}(\mR)}.
\]
The required properties of $g$ follow.
\end{proof}

\section{Scilab program for computing $L_n$ and checking the local controllability property for $n=0, 1, 2, 3$.} \label{appendix-Scilab}

Here is the Scilab program which gives the results in \Cref{rem-scilab}. 

\medskip 

\begin{lstlisting}

clc
a=0;
b=0;
L=0;
q=0;
alpha=0;
beta=0;
value=0;

function[B] = B(x)
    B = (4*x^2*%e^(-6*x) - x^2)^(1/2);
endfunction

function[F] = F(x)
    F = B(x) * cos (B(x)) + x* sin (B(x));
endfunction

function[l] = l(x)
    l = (B(x)*B(x) - 3*x^2)^(1/2);
endfunction

k=0; j=-0.44; j0=j;  
//Here are the outcomes for n=0; 
//value=-0.0140641;  L=4.5183604; a=-0.5065520; b=4.6027563;  q=0.2354919;

//k=1; j=-0.67; j0=j; 
// // Here are the outcomes for n=1: 
// // value=-0.0061256; L=10.866906; a=-0.6903700; b=10.932497; q=0.1291104;


//k=2; j=-0.78; j0=j; 
////  Here are the outcomes for n=2: 
//// value=-0.0036196; L=17.177525; a=-0.7947960; b=17.232599; q=0.0933315;

//k=3; j=-0.86; j0=j; 
////  Here are the outcomes for n=3: 
//// value=-0.0024525; L=23.476776; a=-0.8687610; b=23.524949; q=0.0744156; 

while j > -2 
    j=j - 10^(-6); 
    if (B(j)> %pi + 2*k*%pi) & (B(j)<3*%pi/2+ 2*k*%pi) & (F(j)<0)  
    then j0=j0-10^(-6); 
    else 
        j=-3;
    end
end

a=j0
b=B(a)
L=l(a)
q= - 2*a*(a^2 + b^2)/ L^3
alpha=-a/L
beta=-b/L

function[p] = phix(x)
    p = -alpha*beta*%e^(alpha*x)*cos(beta*x) 
    + beta^2*%e^(alpha*x)*sin(beta*x)
    - 2*alpha*beta*%e^(-2*alpha*x) 
    + 3*alpha^2*%e^(alpha*x)*sin(beta*x)
    + 3*alpha*beta*%e^(alpha*x)*cos(beta*x);
endfunction

function[p] = P(x)
    p = inv_coeff([-q/2 + %i*x 1 0 1]);
endfunction

function[g] = G(z,x)
    lambda= roots(P(z));
    g = (lambda(1)*%e^(lambda(1)*L) - lambda(2)*%e^(lambda(2)*L))
    *%e^(lambda(3)*x)
    + (lambda(2)*%e^(lambda(2)*L) - lambda(3)*%e^(lambda(3)*L))
    *%e^(lambda(1)*x)
    +(lambda(3)*%e^(lambda(3)*L) - lambda(1)*%e^(lambda(1)*L))
    *%e^(lambda(2)*x);
endfunction

function[h] = H(z,x)
    h = abs(G(z,x))*abs(G(z,x))*phix(x);
endfunction

function[omega]=Omega(z)
        function[xi] = Xi(x)
        xi = abs(G(z,x))*abs(G(z,x))*phix(x);
        endfunction
    omega = intg(0,L,Xi);
end

value=Omega(0)
\end{lstlisting}

\providecommand{\bysame}{\leavevmode\hbox to3em{\hrulefill}\thinspace}
\providecommand{\MR}{\relax\ifhmode\unskip\space\fi MR }
\providecommand{\MRhref}[2]{%
  \href{http://www.ams.org/mathscinet-getitem?mr=#1}{#2}
}
\providecommand{\href}[2]{#2}


\begin{thebibliography}{10}

\bibitem{BLR92}
Claude Bardos, Gilles Lebeau, and Jeffrey Rauch, \emph{Sharp sufficient
  conditions for the observation, control, and stabilization of waves from the
  boundary}, SIAM J. Control Optim. \textbf{30} (1992), no.~5, 1024--1065.
  \MR{1178650}

\bibitem{BCG14}
K.~Beauchard, P.~Cannarsa, and R.~Guglielmi, \emph{Null controllability of
  {G}rushin-type operators in dimension two}, J. Eur. Math. Soc. (JEMS)
  \textbf{16} (2014), no.~1, 67--101. \MR{3141729}

\bibitem{BDE20}
Karine Beauchard, J\'{e}r\'{e}mi Dard\'{e}, and Sylvain Ervedoza, \emph{Minimal
  time issues for the observability of {G}rushin-type equations}, Ann. Inst.
  Fourier (Grenoble) \textbf{70} (2020), no.~1, 247--312. \MR{4105940}

\bibitem{BHHR15}
Karine Beauchard, Bernard Helffer, Raphael Henry, and Luc Robbiano,
  \emph{Degenerate parabolic operators of {K}olmogorov type with a geometric
  control condition}, ESAIM Control Optim. Calc. Var. \textbf{21} (2015),
  no.~2, 487--512. \MR{3348409}

\bibitem{BM20}
Karine Beauchard and Fr\'{e}d\'{e}ric Marbach, \emph{Unexpected quadratic
  behaviors for the small-time local null controllability of scalar-input
  parabolic equations}, J. Math. Pures Appl. (9) \textbf{136} (2020), 22--91.
  \MR{4076969}

\bibitem{BMM15}
Karine Beauchard, Luc Miller, and Morgan Morancey, \emph{2{D} {G}rushin-type
  equations: minimal time and null controllable data}, J. Differential
  Equations \textbf{259} (2015), no.~11, 5813--5845. \MR{3397310}

\bibitem{BAM20}
Assia Benabdallah, Franck Boyer, and Morgan Morancey, \emph{A block moment
  method to handle spectral condensation phenomenon in parabolic control
  problems}, Ann. H. Lebesgue \textbf{3} (2020), 717--793. \MR{4149825}

\bibitem{Bona03}
Jerry~L. Bona, Shu~Ming Sun, and Bing-Yu Zhang, \emph{A nonhomogeneous
  boundary-value problem for the {K}orteweg-de {V}ries equation posed on a
  finite domain}, Comm. Partial Differential Equations \textbf{28} (2003),
  no.~7-8, 1391--1436. \MR{1998942}

\bibitem{Bona09}
Jerry~L. Bona, Shu~Ming Sun, and Bing-Yu Zhang, \emph{A non-homogeneous boundary-value problem for the {K}orteweg-de
  {V}ries equation posed on a finite domain. {II}}, J. Differential Equations
  \textbf{247} (2009), no.~9, 2558--2596. \MR{2568064}

\bibitem{Bourgain93}
Jean Bourgain, \emph{Fourier transform restriction phenomena for certain lattice
  subsets and applications to nonlinear evolution equations. {II}. {T}he
  {K}d{V}-equation}, Geom. Funct. Anal. \textbf{3} (1993), no.~3, 209--262.
  \MR{1215780}

\bibitem{1877-Boussinesq}
Joseph Boussinesq, \emph{Essai sur la th{\'e}orie des eaux courantes},
  M\'{e}moires pr\'{e}sent\'{e}s par divers savants \`{a} l'Acad. des Sci.
  Inst. Nat. France, XXIII, pp. 1--680 (1877), 1--680.

\bibitem{CPR15}
Roberto~A. Capistrano-Filho, Ademir~F. Pazoto, and Lionel Rosier,
  \emph{Internal controllability of the {K}orteweg--de {V}ries equation on a
  bounded domain}, ESAIM Control Optim. Calc. Var. \textbf{21} (2015), no.~4,
  1076--1107. \MR{3395756}

\bibitem{Cerpa07}
Eduardo Cerpa, \emph{Exact controllability of a nonlinear {K}orteweg-de {V}ries
  equation on a critical spatial domain}, SIAM J. Control Optim. \textbf{46}
  (2007), no.~3, 877--899. \MR{2338431}

\bibitem{Cerpa14}
Eduardo Cerpa, \emph{Control of a {K}orteweg-de {V}ries equation: a tutorial}, Math.
  Control Relat. Fields \textbf{4} (2014), no.~1, 45--99. \MR{3191303}

\bibitem{CC09}
Eduardo Cerpa and Emmanuelle Cr\'{e}peau, \emph{Boundary controllability for
  the nonlinear {K}orteweg-de {V}ries equation on any critical domain}, Ann.
  Inst. H. Poincar\'{e} Anal. Non Lin\'{e}aire \textbf{26} (2009), no.~2,
  457--475. \MR{2504039}

\bibitem{CK02}
J.~E. Colliander and C.~E. Kenig, \emph{The generalized {K}orteweg-de {V}ries
  equation on the half line}, Comm. Partial Differential Equations \textbf{27}
  (2002), no.~11-12, 2187--2266. \MR{1944029}

\bibitem{Coron06}
Jean-Michel Coron, \emph{On the small-time local controllability of a quantum
  particle in a moving one-dimensional infinite square potential well}, C. R.
  Math. Acad. Sci. Paris \textbf{342} (2006), no.~2, 103--108. \MR{2193655}

\bibitem{Coron07}
Jean-Michel Coron, \emph{Control and nonlinearity}, Mathematical Surveys and Monographs,
  vol. 136, American Mathematical Society, Providence, RI, 2007. \MR{2302744}

\bibitem{CC04}
Jean-Michel Coron and Emmanuelle Cr\'{e}peau, \emph{Exact boundary
  controllability of a nonlinear {K}d{V} equation with critical lengths}, J.
  Eur. Math. Soc. (JEMS) \textbf{6} (2004), no.~3, 367--398. \MR{2060480}


\bibitem{CKN20}
Jean-Michel Coron, Armand Koenig, and Hoai-Minh Nguyen, \emph{On the small-time local controllability of a {KdV} system for
  critical lengths}, J. Eur. Math. Soc. (2022), doi:10.4171/JEMS/1307.

\bibitem{CKN-WT}
Jean-Michel Coron, Armand Koenig, and Hoai-Minh Nguyen, \emph{Lack of local
  controllability for a water-tank system when the time is not large enough},
  (2022).


\bibitem{GG08}
Olivier Glass and Sergio Guerrero, \emph{Some exact controllability results for the
  linear {K}d{V} equation and uniform controllability in the zero-dispersion
  limit}, Asymptot. Anal. \textbf{60} (2008), no.~1-2, 61--100. \MR{2463799}

\bibitem{GG10}
Olivier Glass and Sergio Guerrero, \emph{Controllability of the {K}orteweg-de
  {V}ries equation from the right {D}irichlet boundary condition}, Systems
  Control Lett. \textbf{59} (2010), no.~7, 390--395. \MR{2724598}

\bibitem{Holmer06}
Justin Holmer, \emph{The initial-boundary value problem for the {K}orteweg-de
  {V}ries equation}, Comm. Partial Differential Equations \textbf{31} (2006),
  no.~7-9, 1151--1190. \MR{2254610}

\bibitem{Kato83}
Tosio Kato, \emph{On the {C}auchy problem for the (generalized) {K}orteweg-de
  {V}ries equation}, Studies in applied mathematics, Adv. Math. Suppl. Stud.,
  vol.~8, Academic Press, New York, 1983, pp.~93--128. \MR{759907}

\bibitem{Koenig17}
Armand Koenig, \emph{Non-null-controllability of the {G}rushin operator in
  2{D}}, C. R. Math. Acad. Sci. Paris \textbf{355} (2017), no.~12, 1215--1235.
  \MR{3730500}

\bibitem{KdV}
Diederik~J. Korteweg and Gustave de~Vries, \emph{On the change of form of long
  waves advancing in a rectangular canal, and on a new type of long stationary
  waves}, Philos. Mag. (5) \textbf{39} (1895), no.~240, 422--443. \MR{3363408}

\bibitem{LRZ10}
Camille Laurent, Lionel Rosier, and Bing-Yu Zhang, \emph{Control and
  stabilization of the {K}orteweg-de {V}ries equation on a periodic domain},
  Comm. Partial Differential Equations \textbf{35} (2010), no.~4, 707--744.
  \MR{2753618}

\bibitem{LP15}
Felipe Linares and Gustavo Ponce, \emph{Introduction to nonlinear dispersive
  equations}, second ed., Universitext, Springer, New York, 2015. \MR{3308874}

\bibitem{Marbach18}
Fr\'{e}d\'{e}ric Marbach, \emph{An obstruction to small-time local null
  controllability for a viscous {B}urgers' equation}, Ann. Sci. \'{E}c. Norm.
  Sup\'{e}r. (4) \textbf{51} (2018), no.~5, 1129--1177. \MR{3942039}

\bibitem{MRRR19}
Philippe Martin, Ivonne Rivas, Lionel Rosier, and Pierre Rouchon, \emph{Exact
  controllability of a linear {K}orteweg--de {V}ries equation by the flatness
  approach}, SIAM J. Control Optim. \textbf{57} (2019), no.~4, 2467--2486.
  \MR{3981376}

\bibitem{PVZ02}
Gustavo Alberto~Perla Menzala, Carlos~Frederico Vasconcellos, and Enrique
  Zuazua, \emph{Stabilization of the {K}orteweg-de {V}ries equation with
  localized damping}, Quart. Appl. Math. \textbf{60} (2002), no.~1, 111--129.
  \MR{1878262}

\bibitem{Miura76}
Robert~M. Miura, \emph{The {K}orteweg-de {V}ries equation: a survey of
  results}, SIAM Rev. \textbf{18} (1976), no.~3, 412--459. \MR{404890}

\bibitem{MR02}
Luc Molinet and Francis Ribaud, \emph{On the low regularity of the
  {K}orteweg-de {V}ries-{B}urgers equation}, Int. Math. Res. Not. (2002),
  no.~37, 1979--2005. \MR{1918236}

\bibitem{Ng-Decay21}
Hoai-Minh Nguyen, \emph{Decay for the nonlinear {K}d{V} equations at critical
  lengths}, J. Differential Equations \textbf{295} (2021), 249--291.
  \MR{4268729}

\bibitem{Ng-Squa1}
Hoai-Minh Nguyen and Marco Squassina, \emph{Fractional
  {C}affarelli-{K}ohn-{N}irenberg inequalities}, J. Funct. Anal. \textbf{274}
  (2018), no.~9, 2661--2672. \MR{3771839}

\bibitem{Pazoto05}
Ademir~Fernando Pazoto, \emph{Unique continuation and decay for the
  {K}orteweg-de {V}ries equation with localized damping}, ESAIM Control Optim.
  Calc. Var. \textbf{11} (2005), no.~3, 473--486. \MR{2148854}

\bibitem{Rosier97}
Lionel Rosier, \emph{Exact boundary controllability for the {K}orteweg-de
  {V}ries equation on a bounded domain}, ESAIM Control Optim. Calc. Var.
  \textbf{2} (1997), 33--55. \MR{1440078}

\bibitem{Rosier04}
Lionel Rosier, \emph{Control of the surface of a fluid by a wavemaker}, ESAIM Control
  Optim. Calc. Var. \textbf{10} (2004), no.~3, 346--380. \MR{2084328}

\bibitem{RZ09}
Lionel Rosier and Bing-Yu Zhang, \emph{Control and stabilization of the
  {K}orteweg-de {V}ries equation: recent progresses}, J. Syst. Sci. Complex.
  \textbf{22} (2009), no.~4, 647--682. \MR{2565262}

\bibitem{Rudin-RC}
Walter Rudin, \emph{Real and complex analysis}, third ed., McGraw-Hill Book
  Co., New York, 1987. \MR{924157}

\bibitem{RZ96}
David~L. Russell and Bing~Yu Zhang, \emph{Exact controllability and
  stabilizability of the {K}orteweg-de {V}ries equation}, Trans. Amer. Math.
  Soc. \textbf{348} (1996), no.~9, 3643--3672. \MR{1360229}

\bibitem{Tao06}
Terence Tao, \emph{Nonlinear dispersive equations}, CBMS Regional Conference
  Series in Mathematics, vol. 106, Published for the Conference Board of the
  Mathematical Sciences, Washington, DC; by the American Mathematical Society,
  Providence, RI, 2006, Local and global analysis. \MR{2233925}

\bibitem{TT07}
Gerald Tenenbaum and Marius Tucsnak, \emph{New blow-up rates for fast controls
  of {S}chr\"{o}dinger and heat equations}, J. Differential Equations
  \textbf{243} (2007), no.~1, 70--100. \MR{2363470}

\bibitem{Whitham74}
Gerald~Beresford Whitham, \emph{Linear and nonlinear waves}, Wiley-Interscience
  [John Wiley \& Sons], New York-London-Sydney, 1974, Pure and Applied
  Mathematics. \MR{0483954}

\end{thebibliography}
\end{document}